\documentclass[11pt]{amsart}
\usepackage{amssymb,amsmath, amsthm, amsfonts}

\usepackage[usenames,dvipsnames]{xcolor}
\usepackage{graphicx}
\usepackage{listings}

\usepackage[normalem]{ulem}

\usepackage[margin=1in]{geometry}
\usepackage{lstautogobble}
\usepackage{enumerate}
\usepackage[shortlabels]{enumitem}
\usepackage{thmtools}
\usepackage{thm-restate}
\usepackage{amsthm}
\usepackage{verbatim}
\usepackage{accents}
\usepackage{mathtools}
\usepackage{physics}

\usepackage[colorlinks=true, allcolors=magenta,linkcolor=blue, citecolor=magenta]{hyperref}
\usepackage[capitalise]{cleveref}
\usepackage[bottom]{footmisc}
\crefformat{equation}{(#2#1#3)}

\usepackage{float}
\restylefloat{table}

\usepackage{mathrsfs}
\setlist{  
	listparindent=\parindent,
	parsep=0pt,
}

\theoremstyle{plain}
\newtheorem{thm}{Theorem}[section]
\newtheorem{prop}[thm]{Proposition}
\newtheorem{lemma}[thm]{Lemma}
\newtheorem{cor}[thm]{Corollary}

\theoremstyle{definition}

\newtheorem{myass}[thm]{Assumption}

\newtheorem{remark}[thm]{Remark}

\Crefname{thm}{Theorem}{Theorems}
\Crefname{prop}{Proposition}{Propositions}

\numberwithin{equation}{section} %Equation numbering

\DeclarePairedDelimiter\ipp{\langle}{\rangle}

\DeclarePairedDelimiter{\pa}{\lparen}{\rparen}

\DeclarePairedDelimiter{\jp}{\langle}{\rangle}

\DeclareMathOperator{\supp}{supp}

\DeclareMathOperator{\sgn}{sgn}

\DeclareMathOperator{\BMO}{BMO}

\newcommand{\M}{{\mathcal{M}}}

\newcommand{\p}{{\partial}}

\newcommand{\cre}{\color{red}}

\renewcommand{\d}{\mathsf{d}}

\newcommand{\R}{{\mathbb{R}}}

\newcommand{\N}{{\mathbb{N}}}
\newcommand{\Q}{{\mathbb{Q}}}

\newcommand{\K}{{\mathsf{K}}}
\newcommand{\Ss}{{\mathbb{S}}}

\newcommand{\g}{{\mathsf{g}}}

\newcommand{\Sc}{{\mathcal{S}}}

\renewcommand{\M}{{\mathbb{M}}}
\newcommand{\I}{\mathbb{I}}

\renewcommand{\k}{\mathsf{k}}

\newcommand{\nab}{\nabla}

\newcommand{\tl}{\tilde}

\newcommand{\ph}{\phantom{=}}
\newcommand{\nn}{\nonumber}

\newcommand{\XN}{X_N}

\newcommand{\ux}{X}

\newcommand{\ep}{\epsilon}
\newcommand{\vep}{\varepsilon}
\newcommand{\al}{\alpha}
\newcommand{\be}{\beta}
\newcommand{\ka}{\kappa}
\newcommand{\la}{\lambda}

\newcommand{\indic}{\mathbf{1}}

\newcommand{\Fr}{\mathsf{F}}

\newcommand{\As}{\mathsf{A}}

\newcommand{\E}{{\mathbb{E}}}

\newcommand{\Dm}{|\nabla|}

\newcommand{\rs}{\mathsf{r}}
\newcommand{\cd}{\mathsf{c}_{\mathsf{d},\mathsf{s}}}

\newcommand{\s}{\mathsf{s}}

\renewcommand{\k}{k}

\renewcommand{\P}{\mathcal{P}}
\newcommand{\Te}{\mathrm{Term}}

\newcommand{\as}{\mathsf{a}}
\let\div\relax
\DeclareMathOperator{\div}{\mathrm{div}}

\def\XXint#1#2#3{{\setbox0=\hbox{$#1{#2#3}{\int}$ }
		\vcenter{\hbox{$#2#3$ }}\kern-.6\wd0}}

\let\oldtocsection=\tocsection

\let\oldtocsubsection=\tocsubsection

\let\oldtocsubsubsection=\tocsubsubsection

\renewcommand{\tocsection}[2]{\hspace{0em}\oldtocsection{#1}{#2}}
\renewcommand{\tocsubsection}[2]{\hspace{1em}\oldtocsubsection{#1}{#2}}
\renewcommand{\tocsubsubsection}[2]{\hspace{2em}\oldtocsubsubsection{#1}{#2}}

\title[Another look at regularity in transport-commutator estimates]{Another look at regularity in transport-commutator estimates}

\author[E. Hess-Childs]{Elias Hess-Childs}
\address{Elias Hess-Childs, Carnegie Mellon University, Department of Mathematical Sciences, Pittsburgh, PA}
\email{ehesschi@andrew.cmu.edu}
\thanks{E.H.C was supported by NSF grants DMS-2342349, DMS-2424139.}
\author[M. Rosenzweig]{Matthew Rosenzweig}
\address{Matthew Rosenzweig, Carnegie Mellon University, Department of Mathematical Sciences, Pittsburgh, PA} 
\email{mrosenz2@andrew.cmu.edu}
\thanks{M.R. was supported by NSF grants DMS-2441170, DMS-2345533, DMS-2342349.}
\author[S. Serfaty]{Sylvia Serfaty}
\address{Sylvia Serfaty, Sorbonne Universit\'e,
 CNRS, Universit\'e de Paris,  Laboratoire Jacques-Louis Lions (LJLL), F-75005 Paris \& Institut Universitaire de France \&
Courant Institute of Mathematical Sciences, New York University}
\email{serfaty@cims.nyu.edu}
\thanks{S.S. was supported by NSF grant DMS-2247846 and by Institut Universitaire de France.}

\begin{document}
 \dedicatory{This article is dedicated to the memory of Ha\"{i}m Brezis whose life and work were a great source of inspiration.}
 
	\begin{abstract} 
		We are interested in how regular a transport velocity field must be in order to control Riesz-type commutators. Estimates for these commutators play a central role in the analysis of the mean-field limit and fluctuations for systems of particles with pairwise Riesz interactions, which we start by reviewing. Our first new result shows that the usual $L^\infty$ assumption on the gradient of the velocity field cannot, in general, be relaxed to a BMO assumption. We construct counterexamples in all dimensions and all Riesz singularities $-2< \s<\d$, except for the one-dimensional logarithmic endpoint $\s=0$. At this exceptional endpoint, such a relaxation is possible, a fact related to the classical Coifman-Rochberg-Weiss commutator bound for  the Hilbert transform. Our second result identifies a trade-off between the singularity of the interaction potential and the required regularity of the velocity field. Roughly speaking, smoother (less singular) interactions require stronger velocity control if one wants a commutator estimate in the natural energy seminorm determined by the potential. We formulate this principle for a broad class of potentials and show that, in the sub-Coulomb Riesz regime, the velocity regularity appearing in the known commutator inequality is sharp. Despite these negative findings, we show as our third result that a \emph{defective} commutator estimate holds for almost-Lipschitz transport fields. Such a defective estimate, which is a consequence of the celebrated Brezis-Wainger-Hansson inequality, allows us to prove rates of convergence when the mean-field density belongs to the scaling-critical Sobolev space.
	\end{abstract}
	\maketitle

	\section{Introduction}\label{sec:intro}
	
	\subsection{Mean-field limits}\label{ssec:introMF}
	In recent years, the study of \emph{mean-field limits} for interacting particle systems has seen significant progress. A basic model is the first-order system%
	\footnote{More general dynamics include positive temperature, non-pairwise interactions, non-translation-invariant kernels, and second-order/kinetic models.}
	\begin{equation}\label{eq:MFode}
			\dot{x}_i^t = \displaystyle\frac{1}{N}\sum_{1\leq j\leq N : j\neq i}\M\nabla\g(x_i^t-x_j^t) - \mathsf{V}^t(x_i^t), \qquad  i\in [N].
	\end{equation}
	Here, $\g$ is an \emph{interaction potential} to be specified, and $\mathsf{V}^t$ is a (possibly time-dependent) \emph{external field}. $\M$ is a $\d\times\d$ matrix with real entries that specifies the type of dynamics, with $\M=-\I$ (gradient/dissipative) and $\M$ antisymmetric (Hamiltonian/conservative) being the most common choices. The mean-field limit refers to the convergence as $N \to \infty$ of the {\it empirical measure} 
	\begin{equation}\label{eq:EMt}
		\mu_N^t\coloneqq \frac1N \sum_{i=1}^N \delta_{x_i^t}
	\end{equation}
	associated to a solution $\ux_N^t \coloneqq (x_1^t, \dots, x_N^t)$ of the system \eqref{eq:MFode}. Assuming the initial points $\ux_N^0$ are such that $\mu_N^0$ converges to a sufficiently regular measure $\mu^0$, a formal calculation leads one to expect that for $t>0$, $\mu_N^t$ converges to the solution of the \emph{mean-field equation}
	\begin{equation}\label{eq:MFlim}
			\partial_t \mu^t= \div (\mu^t(\mathsf{V}^t{-\M}\nabla \g*\mu^t)).
	\end{equation}
	Convergence of the empirical measure is qualitatively equivalent to {\it propagation of molecular chaos} (see \cite{golse_dynamics_2016,hauray_kacs_2014} and references therein). This latter notion means that if $f_N^0(x_1, \dots, x_N) = \mathrm{Law}(X_N^0)$ is $\mu^0$-chaotic (i.e.~the $k$-point marginals $f_{N;k}^0\rightharpoonup (\mu^0)^{\otimes k}$ as $N\rightarrow\infty$ for every fixed $k$), then $f_N^t = \mathrm{Law}(\ux_N^t)$ is $\mu^t$-chaotic. %Note that the mean-field limit makes sense for purely deterministic initial data, whereas propagation of chaos is a statement for random initial data. Nevertheless, one can leverage deterministic mean-field convergence rates to prove rates for propagation of chaos, as is possible with our methods (see \cref{rem:pc} below).
	
    Mean-field limits for systems of the form \eqref{eq:MFode} have a long history, beginning with regular drifts (typically, Lipschitz) and then progressing to increasingly more singular interactions over the years. In the interest of brevity, we refer to the introduction of our companion paper \cite{hess-childs_sharp_2025} for some guidance to the literature giving a proper review of these developments.
    
    Instead, we focus on the much more recent treatment of singular interactions, model examples of which are the family of \emph{log/Riesz} potentials\footnote{Since the log case may be understood as the $\s\rightarrow 0$ limit of the Riesz case, we always mean the term ``Riesz'' to include the log case.} given by 
	\begin{equation}\label{eq:gmod}
		\g(x)\coloneqq \begin{cases}  \frac{1}{\s} |x|^{-\s}, \quad  & \s \neq 0\\
			-\log |x|, \quad & \s=0.
		\end{cases}
	\end{equation}
	We  consider any dimension $\d\ge 1$ and assume that $-2<\s<\d$. Up to a normalizing constant $\cd$, the potential $\g$ is the fundamental solution of the fractional Laplacian $(-\Delta)^{\frac{\d-\s}{2}} = |\nab|^{\d-\s}$, i.e. $|\nabla|^{\d-\s}\g = \cd \delta_0$. The cases $0\le \s<\d$ and $-2<\s<0$ are respectively referred to as the \emph{singular} and \emph{nonsingular} regimes because of the behavior of $\g$ at the origin. The restriction $\s>-2$ is natural because the Riesz potential ceases to be conditionally positive definite at $\s=-2$.\footnote{{We say that a kernel $k:\R^\d\times\R^\d\rightarrow \R$ is \emph{conditionally positive definite (CPD)} iff for any signed Borel measure $\rho$ with zero mass, it holds that $\int_{(\R^\d)^2}k(x,y)d\rho(x)d\rho(y)\ge 0$. When the inequality is strict for any nonzero $\rho$, $k$ is said to be \emph{strictly} CPD (SCPD).}\label{fn:cpd_def}} The restriction to the \emph{potential} regime $\s<\d$, in which $\g$ is locally integrable, is to exclude the \emph{hypersingular} case, which is not of the mean-field type considered in this paper (e.g. see \cite{hardin_large_2018,hardin_dynamics_2021}). The case $\s=\d-2$ corresponds to the classical \emph{Coulomb} potential from electrostatics/gravitation. Based on this threshold, it is convenient to call the interaction \emph{sub-Coulomb} if $\s<\d-2$ and \emph{super-Coulomb} if $\s>\d-2$. More generally, long-range interactions of the form \eqref{eq:gmod} play an important role in physics \cite{dauxois_dynamics_2002}, approximation theory \cite{borodachov_discrete_2019}, and machine learning \cite{altekruger_neural_2023,hertrich_generative_2023, hertrich_wasserstein_2023, hagemann_posterior_2023}---just to name a few fields. Motivations for considering systems of the form \eqref{eq:MFode} are extensively reviewed in \cite{serfaty_lectures_2024}.  

    Mean-field convergence for the full range $-2<\s<\d$ has only been resolved in the last few years: the nonsingular case  $-2<\s<0$ \cite{rosenzweig_wasserstein_nodate}, the sub-Coulomb case $\s<\d-2$, \cite{hauray_wasserstein_2009,carrillo_derivation_2014}, the Coulomb/super-Coulomb case $\d-2\leq \s<\d$ \cite{duerinckx_mean-field_2016,carrillo_mass-transportation_2012,berman_propagation_2019,serfaty_mean_2020}, and the full singular case $0\leq \s<\d$ \cite{nguyen_mean-field_2022}. See also \cite{bresch_modulated_2019,rosenzweig_mean-field_2020,rosenzweig_global--time_2023, hess-childs_large_2023} for further extensions incorporating noise. These advances are due in large part to the \emph{modulated energy} method, which we review in the next subsection. %We also mention that the modulated energy method has been recently used to treat the nonsingular case $-2<\s<0$ under minimal assumptions \cite{rosenzweig_wasserstein_nodate}.

    \subsection{Modulated energy and commutators}\label{ssec:introMEcomm}
    Since the empirical measure $\mu_N^t$ associated with the microscopic system ~\eqref{eq:MFode} is a weak solution of the mean-field equation ~\eqref{eq:MFlim}, a natural approach for showing mean-field convergence is a weak-strong stability estimate for~\eqref{eq:MFlim}. For Riesz interactions, such stability is conveniently formulated in terms of the energy-based distance%In particular, for probability measures $\mu$ and $\nu$ with finite interaction energy, one introduces
	%\begin{equation}\label{eq:Edis}
	%	\|\mu-\nu\|_\g^2
	%	:= \frac12 \int_{(\R^\d)^2} \g(x-y)\,d(\mu-\nu)^{\otimes 2}(x,y).
	%\end{equation}
	%The use of this norm in dynamical settings originates in~\cite{Duerinckx2016,Serfaty2020}, where it is used to prove a weak--strong stability principle for~\eqref{eq:MFlim}.
	%This comparison between the empirical measure $\mu_N$ and the mean-field density $\mu$ is conveniently performed by considering a , or Riesz (squared) ``distance" between $\mu_N$ and $\mu$, defined by 
\begin{equation}\label{eq:modenergy}
\Fr_N(\ux_N, \mu) \coloneqq \frac12\int_{(\R^\d)^2 \setminus \triangle} \g(x-y) d\Big(\frac{1}{N} \sum_{i=1}^N \delta_{x_i} - \mu\Big)(x) d\Big(\frac{1}{N} \sum_{i=1}^N \delta_{x_i} - \mu\Big)(y),
\end{equation}
where we excise the diagonal $\triangle\coloneqq \{(x,y)\in(\R^\d)^2:x=y\}$ from the domain of integration in order to remove the infinite self-interaction of each particle.  The {\it modulated energy} \eqref{eq:modenergy} originated in the study of the statistical mechanics of Coulomb/Riesz gases \cite{sandier_1d_2015,sandier_2d_2015,rougerie_higher-dimensional_2016, petrache_next_2017} and later was used in the derivation of mean-field dynamics \cite{duerinckx_mean-field_2016,serfaty_mean_2020,nguyen_mean-field_2022} and following works. In the nonsingular case, where the diagonal contribution vanishes, the modulated energy coincides with the square of what is known in the statistics literature as \emph{maximum mean discrepancy (MMD)} \cite{gretton_kernel_2006, gretton_kernel_2007,gretton_kernel_2012}, a type of integral probability metric \cite{muller_integral_1997}. MMDs for varying choices of kernels are widely used in statistics and machine learning contexts as distances between probability measures (e.g. see \cite{modeste_characterization_2024, kolouri_generalized_2022,  hertrich_generative_2023}). %They have  advantages over Wasserstein metrics in terms of the curse of dimensionality \cite{MD2024}, and Riesz MMDs are attractive computationally on account of their slicing property \cite{KNSS2022, HWAH2023}.
We refer to  \cite[Chapter 4]{serfaty_lectures_2024} for a comprehensive discussion of the modulated energy.

The modulated energy approach to mean-field convergence consists of establishing a Gr\"onwall relation for $\Fr_N(\XN^t,\mu^t)$, where $\XN^t$ is a solution of the microscopic system \eqref{eq:MFode} and $\mu^t$ is a solution of the mean-field equation \eqref{eq:MFlim}. An essential point in this approach is to control the derivative of $\Fr_N$ along  a transport $v:\R^\d\rightarrow\R^\d$, that is
\begin{equation}\label{15}
 \frac{d}{dt}\Big|_{t=0} \Fr_N( (\I + tv)^{\oplus N} (\ux_N), (\I + tv)\# \mu)= {\frac12}\int_{(\R^\d)^2\setminus \triangle} 
\nabla\g(x-y)\cdot (v(x)-v(y))  d ( \mu_N- \mu)^{\otimes 2}(x,y),
\end{equation}
where for any configuration $\ux_N=(x_1, \dots, x_N)$, $\mu_N$ denotes $\frac1N \sum_{i=1}^N \delta_{x_i}$,  $\I:\R^\d\rightarrow\R^\d$ is the identity and $(\I+t v)^{\oplus N} (\ux_N) \coloneqq (x_1 + tv(x_1), \ldots, x_N+ tv(x_N))$. 
For applications to the mean-field limit, $v$ is the velocity field of the limiting evolution \eqref{eq:MFlim}. For applications to central limit theorems (CLTs) for the fluctuations (the next-order description beyond the mean-field limit), $v$ is the gradient of a test function evolved along the adjoint linearized mean-field flow \cite{huang_fluctuations_nodate, cecchin_convergence_2025}. These inequalities are at the core of the ``transport'' approach to fluctuations of canonical Gibbs ensembles \cite{leble_fluctuations_2018,bekerman_clt_2018,serfaty_gaussian_2023, peilen_local_2025}  and also appear in the loop (Dyson--Schwinger) equations in the random matrix theory literature, e.g. \cite{borot_asymptotic_2013-1, borot_asymptotic_2024, bauerschmidt_two-dimensional_2019}. Higher-order variations of the modulated energy are also important, a point we return to in \cref{ssec:RQhigher}. %, which are another avatar of the transport method.%, among others. Loop equations are common in mathematical physics and amount to transcribing the conservation of energy under translations, in effect equivalent to a transport method.

The desired control is a functional inequality of the form
\begin{equation}\label{eq:introcomm}
|\eqref{15} | \le C_1(\Fr_N(\ux_N, \mu) + C_2N^{-\alpha}),
\end{equation}
where $\al>0$ depends on $\d,\s$, $C_1>0$ is a constant depending on $\d,\s,v$, and $C_2>0$ depends on $\d,\s,\mu$. This form of control was first proved by Lebl\'{e} and the third author in \cite{leble_fluctuations_2018} in the 2D Coulomb case, then generalized to the super-Coulomb case in \cite{serfaty_mean_2020}. The proof relied on identifying  a {\it stress-energy tensor} structure in \eqref{15} and using integration by parts. Subsequently, the second author \cite{rosenzweig_mean-field_2020} observed that the expressions \eqref{15}  may also be viewed as the quadratic form of a \textit{commutator}, akin to the famous Calder\'{o}n commutator \cite{calderon_commutators_1980, coifman_commutateurs_1978, christ_polynomial_1987, seeger_multilinear_2019}. Indeed, ignoring the exclusion of the diagonal,
\begin{equation}
    \int_{(\R^\d)^2} (v(x)-v(y))\cdot\nabla\g(x-y)f(x)g(y)dxdy=\cd\Big\langle f ,\Big[v\cdot,\nabla |\nabla|^{\s-\d}\Big] g \Big\rangle_{L^2},
\end{equation}
where $\ipp{\cdot,\cdot}_{L^2}$ denotes the $L^2$ inner product. For this reason, estimates of the form \eqref{eq:introcomm} are called \emph{commutator estimates}.

This commutator perspective was used by the last two authors and Q.H. Nguyen \cite{nguyen_mean-field_2022} to show the desired functional inequality for the full singular regime as well as a broader class of Riesz-type interactions including Lennard-Jones-type potentials. The argument of \cite{nguyen_mean-field_2022} has also been extended to the nonsingular regime \cite{rosenzweig_wasserstein_nodate}. As recognized in \cite{rosenzweig_sharp_nodate}, the stress-energy and commutator perspectives are in fact two sides of the same coin. These functional inequalities are crucial not only for proving CLTs for fluctuations of Riesz gases \cite{leble_fluctuations_2018, serfaty_gaussian_2023, peilen_local_2025}, but also for deriving mean-field limits \cite{serfaty_mean_2020,rosenzweig_mean-field_2022-1,rosenzweig_mean-field_2022,nguyen_mean-field_2022,chodron_de_courcel_sharp_2023,rosenzweig_sharp_nodate, porat_singular_2025} and large deviation principles \cite{hess-childs_large_2023}. They have also been used to show the joint classical and mean-field limits for quantum systems of particles \cite{GP2022},  and the supercritical mean-field limits of classical \cite{han-kwan_newtons_2021, rosenzweig_rigorous_2023, menard_mean-field_2024, rosenzweig_lake_2025} and quantum systems of particles \cite{rosenzweig_quantum_2021, porat_derivation_2023}.

Let us also mention that for gradient dynamics at positive temperature (i.e. with Brownian noise in the microscopic dynamics \eqref{eq:MFode}), the modulated energy combines naturally with the previously used relative entropy \cite{jabin_quantitative_2018, guillin_uniform_2024, feng_quantitative_2024, rosenzweig_relative_2024} in the form of the \emph{modulated free energy} \cite{bresch_mean-field_2019, bresch_modulated_2019, bresch_mean_2023, chodron_de_courcel_sharp_2023, rosenzweig_modulated_2025, rosenzweig_relative_2024, cai_propagation_2024}, which is well-suited to proving entropic propagation of chaos and can even handle logarithmically attractive interactions \cite{bresch_mean-field_2019, bresch_mean_2023, chodron_de_courcel_attractive_2025}. Commutator estimates play the same essential role in the positive temperature setting.

The main problem for these functional inequalities has been to determine the optimal size of the additive error, i.e.~the exponent $\alpha$ in \eqref{eq:introcomm}, which a priori depends on $\d$ and $\s$. In formulating optimality, the appropriate comparison is with the minimal size of $\Fr_N(\ux_N,\mu)$. When $0\le \s<\d$, it is known that $\Fr(\ux_N,\mu) + \frac{\log(N\|\mu\|_{L^\infty})}{2\d N}\indic_{\s=0}\ge -C\|\mu\|_{L^\infty}^{\frac\s\d}N^{\frac{\s}{\d}-1}$ for $C=C(\d,\s)>0$ and this lower bound is sharp \cite{rosenzweig_sharp_nodate, hess-childs_sharp_2025}.\footnote{The term $\frac{\log(N\|\mu\|_{L^\infty})}{2\d N}$ in the $\s=0$ case is \emph{not} an additive error like $\|\mu\|_{L^\infty}^{\frac\s\d}N^{\frac{\s}{\d}-1}$, but rather a consequence of the fact that the right quantity to consider is $\Fr(\ux_N,\mu) + \frac{\log(N\|\mu\|_{L^\infty})}{2\d N}\indic_{\s=0}$.} In controlling the magnitude of \eqref{15}, one needs a right-hand side which is also nonnegative. Thus, the best error one may hope for is of size $N^{\frac{\s}{\d}-1}$. When $-2<\s<0$, the modulated energy is nonnegative (it is a squared MMD as mentioned above), and its minimal value is $\propto N^{\frac\s\d-1}$ \cite{hess-childs_optimal_nodate}. Moreover, the functional inequality holds without additive error because the interaction is nonsingular \cite{rosenzweig_wasserstein_nodate}. 
%\edit{This sharp error rate was obtained in the Coulomb and super-Coulomb case in \cite{RS2022}}.

Only very recently has the program to show the sharp $N^{\frac{\s}{\d}-1}$ error rate been completed by the authors and collaborators: first, the Coulomb case \cite{leble_fluctuations_2018,serfaty_gaussian_2023, rosenzweig_rigorous_2023}, then the entire (super-)Coulomb case \cite{rosenzweig_sharp_nodate}, and finally the entire singular Riesz case \cite{hess-childs_sharp_2025}. The completion of this program has yielded the sharp rates of convergence for the mean-field limit in the modulated energy distance. It has also been crucial for studying fluctuations around the mean-field limit in and out of thermal equilibrium \cite{leble_fluctuations_2018,serfaty_gaussian_2023,peilen_local_2025,huang_fluctuations_nodate, rosenzweig_commutator_nodate, rosenzweig_cumulants_nodate}.

    Summarizing the state of the art for these commutator estimates is the following theorem, combining the results of \cite{nguyen_mean-field_2022,rosenzweig_global--time_2023, rosenzweig_sharp_nodate,hess-childs_sharp_2025, rosenzweig_wasserstein_nodate}. In particular, taking $p=\infty$ in \eqref{eq:FIsupC}, \eqref{eq:FIsubC1} below yields the announced sharp $N^{\frac\s\d-1}$ additive error.
    
    To compactify the notation, given a point configuration $\XN$, reference measure $\mu$, and vector field $v$, let us write
    \begin{align}\label{eq:A1def}
        \As_1[\XN,\mu,v] \coloneqq \int_{(\R^\d)^2\setminus\triangle}\nab\g(x-y)\cdot(v(x)-v(y))d\Big(\frac1N\sum_{i=1}^N\delta_{x_i}-\mu\Big)^{\otimes 2}(x,y).
    \end{align}
    The subscript $1$ indicates that $\As_1$ is a first-order variation/commutator, a point whose significance  will become clearer in \cref{ssec:RQhigher}.

    \begin{thm}\label{thm:FI'}
    Let $-2<\s<\d$. Let $\mu \in L^1\cap L^p$ for $\frac{\d}{\d-\s}<p\le \infty$ with $\int_{\R^\d}d\mu = 1$, and associated to $\mu$, define the length scales
    \begin{align}
        \la &\coloneqq (N\|\mu\|_{L^p})^{-\frac{1}{\d}}\label{eq:ladef},\\\
        \ka &\coloneqq (N^{\frac{1}{\s+1}}\|\mu\|_{L^p})^{-\frac{1}{\d}}, \label{eq:kadef}
    \end{align}
    When $-2<\s\le 0$, assume that $\int_{(\R^\d)^2}|\g|(x-y)d|\mu|^{\otimes 2}<\infty$.\footnote{{This condition is to ensure that the modulated energy is finite. When $0<\s<\d$, it is implied by the assumption that $\mu\in L^1\cap L^p$ for $p>\frac{\d}{\d-\s}$. When $\s=0$, it is implied by a logarithmic moment assumption $\int_{\R^\d}\log(1+|x|)d|\mu|<\infty$; and when $-2<\s<0$, it is implied by the moment assumption $\int_{\R^\d}|x|^{|\s|/2}d|\mu|<\infty$.}\label{fn:finen}}    Let $v$ be a Lipschitz vector field. Then for any pairwise distinct point configuration $\XN\in(\R^\d)^N$, the following hold.
    
        \cite{rosenzweig_sharp_nodate} If $\s\ge \max(0,\d-2)$, then
        \begin{align}\label{eq:FIsupC}
            | \As_1[\XN,\mu,v]|
            \le C\|\nab v\|_{L^\infty}\Big(\Fr_N(\XN,\mu) - \frac{\log \la}{2N}\indic_{\s=0} + C_p\|\mu\|_{L^p}\la^{\frac{\d(p-1)}{p}-\s} \Big).
        \end{align}

        \cite{hess-childs_sharp_2025} If $0\le \s<\d-2$, then for any $\as \in (\d,\d+2)$, 
        \begin{multline}\label{eq:FIsubC1}
            | \As_1[\XN,\mu,v]| \\
            \le             C(\|\nab v\|_{L^\infty} + C_\as\|\Dm^{\frac{\as}{2}}v\|_{L^{\frac{2\d}{\as-2}}}\indic_{\as>2}) \Big(\Fr_N(\XN,\mu) - \frac{\log \la}{2N}\indic_{\s=0} + C_p\|\mu\|_{L^p}\la^{\frac{\d(p-1)}{p}-\s} \Big),
        \end{multline}
        \cite{nguyen_mean-field_2022,rosenzweig_global--time_2023} and for any $\frac{\d}{\d-\s-1}<p\le \infty$, 
        \begin{multline}\label{eq:FIsubC2}
            | \As_1[\XN,\mu,v]|
           \\ \le C(\|\nab v\|_{L^\infty} + \|\Dm^{\frac{\d-s}{2}}v\|_{L^{\frac{2\d}{\d-\s-2}}} )\Big(\Fr_N(\XN,\mu) + C_p\|\mu\|_{L^p}\ka^{\frac{\d(p-1)}{p}-\s}(1-(\log\ka) \indic_{\s=0})\Big).
        \end{multline}
        
        \cite{rosenzweig_wasserstein_nodate} If $-2<\s <0$, then
        \begin{align}\label{eq:FInonsing}
             |\As_1[\XN,\mu,v]| \le C(\|\nab v\|_{L^\infty} + \|\Dm^{\frac{\d-\s}{2}}v\|_{L^{\frac{2\d}{\d-\s-2}}}\indic_{\s<\d-2})\Fr_N(\XN,\mu)
        \end{align}
    In all cases, the constant $C>0$ depends only on $\d,\s$ and the constants $C_\as, C_p>0$ additionally depend on $\as,p$, respectively.    
    \end{thm}

    Given a measure $\nu$ with density in $\Sc(\R^\d)$ and $\int_{\R^\d}d\nu=1$, letting $f\coloneqq \nu-\mu$, we can find a sequence  of pairwise distinct point configurations $\XN$ such that, as $N \to \infty$,
    \begin{align}
        \Fr_N(\XN,\mu)\rightarrow \frac12\int_{(\R^\d)^2}\g(x-y)f(x)f(y)dxdy = \frac{\cd}{2}\|f\|_{\dot{H}^{\frac{\s-\d}{2}}}^2
    \end{align}
    and
    \begin{align}
        \As_1[\XN,\mu,v] \rightarrow \frac12\int_{(\R^\d)^2}\nab\g(x-y)\cdot (v(x)-v(y))f(x)f(y)dxdy =\frac12\Big\langle f,\comm{v\cdot}{\nab \Dm^{\s-\d}}f\Big\rangle_{L^2}.
    \end{align}
    Thus, a commutator estimate with an additive error involving the singular zero-mass measure $\frac1N\sum_{i=1}^N\delta_{x_i}-\mu$ implies a commutator estimate without an additive error for the regular zero-mass measure $f$. More precisely, \cref{thm:FI'} implies the following  standard or \emph{unrenormalized} commutator estimate (see \cite[Proposition 3.1]{nguyen_mean-field_2022}).
    
    \begin{prop}\label{prop:FI'}
        Let $\d\geq 1$ and $-2<\s<\d$. Then there exists $C>0$, depending only on $\d$ and $\s$, such that for all $f,g\in\mathcal{S}(\R^\d)$,\footnote{If $\s=0$, then implicit is the requirement that $f$ and $g$ are zero-mean to avoid the low-frequency divergence in $\dot{H}^{-\frac{\d}{2}}(\R^\d)$.} we have
		\begin{multline}\label{eq:CE}
			\bigg|\int_{(\R^\d)^2} (v(x)-v(y))\cdot\nabla\g(x-y)f(x)g(y)dxdy\bigg|
			\\\leq C\Big(\|\nabla v\|_{L^\infty}+\||\nabla|^{\frac{\d-\s}{2}} v\|_{L^\frac{2\d}{\d-\s-2}}\indic_{\s<\d-2}\Big)\|f\|_{\dot{H}^{\frac{\s-\d}{2}}}\|g\|_{\dot{H}^{\frac{\s-\d}{2}}}.
		\end{multline}
    \end{prop}

    It is straightforward to check that the preceding right-hand side scales the same way as the left-hand side, and that both $v$ terms on the right-hand side scale the same way as well. Moreover, the presence of the $\dot{H}^{\frac{\s-\d}{2}}$ seminorm is natural, given this corresponds to the square root of the energy associated to $\g$.
    
    %The second expression is the $L^2$ quadratic form of the commutator $\comm{v\cdot}{\nab \Dm^{\s-\d}}$ evaluated at the regular (signed) measure $f$.
    In fact, the philosophy of \cite{rosenzweig_mean-field_2020, nguyen_mean-field_2022, rosenzweig_sharp_nodate}, and to a more implicit degree the works \cite{leble_fluctuations_2018,  duerinckx_mean-field_2016, serfaty_mean_2020, serfaty_gaussian_2023}, proceeds in the reverse order: first show the {unrenormalized} commutator estimate of \cref{prop:FI'}, then combine it with a renormalization procedure, implemented via charge smearing (i.e. regularization of the $\delta_{x_i}$), to allow for the irregular test measure $\frac1N\sum_{i=1}^N\delta_{x_i}$. It is this renormalization step, which is the most technically involved part of the proof, that leads to the additive $N$-dependent errors in the singular case $0\le \s<\d$. For this reason, the estimates \eqref{eq:FIsupC}, \eqref{eq:FIsubC1}, \eqref{eq:FIsubC2} are sometimes referred to as \emph{renormalized} commutator estimates to distinguish them from the unrenormalized estimate \eqref{eq:CE}. In the nonsingular case $-2<\s<0$, no renormalization is necessary, and one can apply a commutator estimate directly, leading to the absence of additive errors in \eqref{eq:FInonsing}.

    \subsection{Transport regularity}\label{ssec:introObj}
    The question we take up in this paper is the optimal regularity assumption for the transport velocity $v$ in the aforementioned commutator estimates. Compared to sharpness of the additive errors in these estimates, discussed above, this question has received little attention in the literature.
    
   In this direction, it is natural to ask whether the Lipschitz assumption on $v$ may be weakened while still preserving the scaling. For example, the space of \emph{bounded mean oscillation} (BMO) \cite{https://doi.org/10.1002/cpa.3160140317} is a common relaxation of $L^\infty$, and one can ask whether $\|\nab v\|_{L^\infty}$ may be replaced by $\|\nab v\|_{BMO}$. It is also unclear whether, in the sub-Coulomb regime $\s<\d-2$, the additional term of the form $\|\Dm^{\frac{\as}{2}}v\|_{L^{\frac{2\d}{\as-2}}}$ is genuinely necessary and how small the exponent $\as$ may be taken (by Sobolev embedding, $\|\Dm^{\frac{\as}{2}}v\|_{L^{\frac{2\d}{\as-2}}} \lesssim \|\Dm^{\frac{\as'}{2}}v\|_{L^{\frac{2\d}{\as'-2}}}$ for $2(\d+1)>\as'\ge \as$). Moreover, there is a gap between our two sub-Coulomb estimates \eqref{eq:FIsubC1}, \eqref{eq:FIsubC2}. In the former, the $N$-dependent additive error is sharp but requires a stronger regularity condition on $v$ (given that $\as>\d$ by assumption) compared to the latter, in the $N$-dependent error is not sharp. 
	%$\||\nabla|^{\frac{\d-\s}{2}} v\|_{L^{\frac{2\d}{\d-\s-2}}}$
	%is genuinely necessary, or whether it can be replaced by a strictly weaker norm with the same scaling.

   % Our first contribution is to show that 
    
	%The primary objective of this note is to show that the following commutator estimate, established in \cite{NRS2021}, is sharp with respect to the norms imposed on the transport field $v:\R^d\to\R^d$.

	%In~\cite{NRS2021} the estimate~\eqref{eq:CE} is proved in greater generality, for a class of singular ``Riesz-like'' potentials that includes the exact Riesz kernels~\eqref{eq:gmod}. Moreover, \eqref{eq:CE} serves as the continuum input for the corresponding renormalized commutator estimate at the level of the modulated energy~\eqref{eq:modenergy}, which in turn yields mean-field convergence in the modulated-energy distance.

    This regularity question is not only for the sake of mathematics, but is also motivated by applications. It is a general principle in PDE that if an equation has a scaling invariance, then a natural function space to study the equation is one that is invariant under the equation's scaling \cite{klainerman_pde_2010}. Such \emph{scaling-critical} or \emph{scaling-invariant} spaces are expected to be thresholds for the well-posedness/ill-posedness of the equation along with other phenomena, such as finite-time blowup vs. global existence.

For $a,b,c>0$, set $\tl{\mu}(t,x) \coloneqq  a\mu(bt,cx)$. When $\mathsf{V}^t = 0$, the reader may check that if $\mu$ is a solution to \eqref{eq:MFlim}, then $\tl{\mu}$ is also a solution provided that $bc^{\d-2-\s}=a$. Fixing $b=1$ so that the solution lifespan is preserved under rescaling, we arrive at the relation $a=c^{\d-2-\s}$. We seek (semi)norms that are invariant under the transformation $\mu\mapsto \tl{\mu}$.

For $1\leq p\leq\infty$, the reader may check that
\begin{equation}
\|\tl{\mu}\|_{L^p} = c^{\d-\s-2-\frac{\d}{p}} \|\mu\|_{L^p},
\end{equation}
which implies that $p=\frac{\d}{\d-\s-2}$ is the scaling-critical exponent. It is possible to choose such a $p$ if and only if $\s\leq \d-2$. In other words, the potential can be at  most Coulomb in order for there to be a critical $L^p$ space. More generally, letting $\dot{W}^{a,p}$ denote the homogeneous Bessel potential (fractional Sobolev) space (e.g. see \cite[Section 3.3]{stein_singular_1970} or \cite[Section 1.3]{grafakos_modern_2014}), we see that
\begin{equation}
\|\tl\mu\|_{\dot{W}^{a,p}} = c^{\d-\s-2 + a -\frac{\d}{p}}\|\mu\|_{\dot{W}^{a,p}},
\end{equation}
which shows that $\dot{W}^{a,p}$ is a critical space if and only if $a=\frac{\d}{p} + \s+2-\d$. Other scales of function spaces, such as H\"older-Zygmund spaces (e.g. see \cite[Section 4]{stein_singular_1970}), are also interesting to consider as critical spaces. We will return to this point in the next section.

%we  For the cases $d=1,0\leq s<d$ and $d\geq2, d-2<s<d$, we look for a critical H\"older/Zygmund space $\dot{\mathcal{C}}^r$. We compute
%\begin{equation}
%\|\mu_\la\|_{\dot{\mathcal{C}}^r} = \la^{d-s-2+r} \|\mu\|_{\dot{C}^r},
%\end{equation}
%which implies that $r=s+2-d$.

When $\mu \in \dot{W}^{\frac{\d}{p} + \s+2-\d,p}$, for $1<p<\infty$, some basic Fourier analysis shows that the associated mean-field velocity $v=\M\nabla\g\ast\mu$ is in $\dot{W}^{\frac{\d}{p}+1,p}$.  Unfortunately, $\dot{W}^{\frac{\d}{p}+1,p}$ does not embed into the homogeneous Lipschitz space $\dot{W}^{1,\infty}$. If $f\in \dot{W}^{\frac{\d}{p}+1,p}$, then one only has $\|\nab f\|_{BMO} \leq C\|f\|_{\dot{W}^{\frac{\d}{p}+1,p}}$. Even when $\s+2-\d=k$ is an integer and taking $p=\infty$, having $\mu\in\dot{W}^{k,\infty}$ does not imply that $v\in \dot{W}^{1,\infty}$. This is problematic because all of the commutator estimates of \cref{thm:FI'} feature $\|\nab v\|_{L^\infty}$ on the right-hand side. Thus, to show mean-field convergence when the mean-field density $\mu^t$ belongs to the critical space, it would be desirable to have a commutator estimate with a regularity assumption that is satisfied when $\mu\in \dot{W}^{\frac{\d}{p} + \s+2-\d,p}$.

\subsection{Our contributions}
Having laid out the question of interest concerning transport regularity in commutator estimates for the modulated energy and motivated such inequalities by their role in studying mean-field and related limits, we come to the new contributions of the present paper. We describe them here at a high level. Precise statements of our main results (\Cref{thm:CEF1,thm:CEF2,thm:FI,thm:MF} and \cref{cor:CEF2}) and detailed comments about the proofs are deferred to \cref{sec:MR}.
    
Our first contribution is to show that estimate~\eqref{eq:CE} can fail if either of the vector field norms is weakened (\cref{thm:CEF1}). We also derive necessary conditions for analogous commutator estimates to hold for a broader class of interaction potentials (\cref{thm:CEF2}). In particular, the norm controlling $v$ must involve enough derivatives to compensate for the high-frequency decay of $\hat\g$. This demonstrates a more general principle of regularity trade-off between the interaction and transport for commutator estimates. All of these results are obtained through explicit counterexamples. Given that that the renormalized commutator estimates of \cref{thm:FI'} imply unrenormalized commutator estimates, as explained in \cref{ssec:introMEcomm}, it suffices to show these counterexamples when $\frac1N\sum_{i=1}^N\delta_{x_i}-\mu$ is replaced by a zero-mean Schwartz function.

Our second contribution concerns mean-field convergence when the solution $\mu^t$ of the mean-field equation belongs to the critical space $\dot{W}^{\frac{\d}{p}+\s+2-\d,p}$ for $1<p<\infty$. Provided that $p \le \frac{2\d}{\d-\s-2}$, one has $\|\Dm^{\frac{\d-\s}{2}}u^t\|_{L^{\frac{2\d}{\d-\s-2}}} \lesssim \|\mu^t\|_{\dot{W}^{\frac{\d}{p}+\s+2-\d,p}}$. But, since in general $\nab u^t$ is only in BMO when $\mu^t\in \dot{W}^{\frac{\d}{p}+\s+2-\d,p}$, the results described in the previous paragraph give a negative answer to our earlier question about the possibility of a commutator estimate holding in this case, regardless of the size of the additive error. That said, the situation is not completely hopeless. As recognized in \cite{rosenzweig_mean-field_2022-1,rosenzweig_mean-field_2022}, even in this borderline regularity case, one can still obtain a Gr\"onwall-type\footnote{More precisely, an estimate in the form of Osgood's lemma, which is a generalization of Gr\"onwall's lemma. See \cite[Section 3.1.1]{bahouri_fourier_2011}.} estimate for the modulated energy provided that a \emph{defective} commutator estimate holds.

More precisely, we say that a defective version holds if there exists a locally integrable function $\rho:[0,\infty)\rightarrow [0,\infty)$, called the \emph{defect}, such that for all $\ep>0$,
\begin{multline}\label{eq:commdefec}
    \int_{(\R^\d)^2\setminus\triangle}(v(x)-v(y))\cdot\nab\g(x-y)d\Big(\frac1N\sum_{i=1}^N\delta_{x_i}-\mu\Big)^{\otimes 2}(x,y) \\
    \le \rho(\ep)C_v\Big(\Fr_N(\XN,\mu) + C_\mu N^{-\alpha} \Big) + {\ep}^{\beta}N^{\gamma}C_v'\Big(\Fr_N(\XN,\mu) + C_\mu' N^{-\alpha} \Big),
\end{multline}
where $\alpha > 0$, $\beta,\gamma\ge 0$ and $C_v,C_v',C_\mu,C_\mu'>0$ are constants depending on $\d,\s$ and quantitatively on norms of $v,\mu$, respectively. 

If $\rho({\ep}) = O(|\log {\ep}|)$ as ${\ep}\rightarrow 0^+$, it turns out that a defective commutator estimate suffices to conclude an estimate for the modulated energy that in turn implies mean-field convergence, at least up to some time $T_*>0$ which is independent of $N$, but a priori is smaller than the maximal time of existence for $\mu^t$. These works considered only the Coulomb case, where the natural scale-invariant space is $L^\infty$, whose norm is a nonincreasing (conserved when $\M$ is antisymmetric) quantity when $\mathsf{V}^t=0$. In the present paper, we generalize the results of \cite{rosenzweig_mean-field_2022-1,rosenzweig_mean-field_2022} to the entire Riesz case without any restrictions on the interval of time for which mean-field convergence holds (\cref{thm:MF}). Importantly, in the sub-Coulomb case $\s<\d-2$, we show that if the initial mean-field density $\mu^\circ \in L^{\frac{\d}{\d-\s-2}}$, which is scaling-critical, then, there is a unique global solution to the mean-field equation \eqref{eq:MFlim} which is a (quantitative) mean-field limit of the microscopic system \eqref{eq:MFode}. %under mild assumptions on $\mathsf{V}^t$

The main technical result is a new defective commutator estimate (\cref{thm:FI}) with a defect $\rho(\ep) = O(|\log\ep|^{\theta})$ for $\theta\in (0,1)$, in contrast to the $\theta=1$ defect present in \cite{rosenzweig_mean-field_2022-1,rosenzweig_mean-field_2022}. This improvement is made possible by the celebrated \emph{Brezis-Wainger-Hansson inequality} \cite{brezis_note_1980,hansson_imbedding_1979}: if $1<p<\infty$, $1\le q\le \infty$, and $aq>\d$, then 
\begin{align}\label{eq:BW}
     \forall f\in\Sc(\R^\d), \qquad \|f\|_{L^\infty} \le C\|f\|_{\dot{W}^{\frac{\d}{p},p}}\Big(1+\log^{\frac{p-1}{p}}(1+\frac{\|f\|_{\dot{W}^{a,q}}}{\|f\|_{\dot{W}^{\frac{\d}{p},p}}})\Big),
\end{align}
where $C=C(\d,p,a,q)>0$. Remark that the exponent $\frac{p-1}{p}$ is sharp; and moreover, no estimate of the form
\begin{align}
    \|f\|_{L^\infty} \le \vep\|f\|_{\dot{W}^{\frac{\d}{p},p}}\Big(1+\log^{\frac{p-1}{p}}(1+\frac{\|f\|_{\dot{W}^{a,q}}}{\|f\|_{\dot{W}^{\frac{\d}{p},p}}})\Big) + C_\vep,
\end{align}
for any $\varepsilon>0$, can hold.

Brezis and Wainger \cite[Theorem 1]{brezis_note_1980} first proved the estimate \eqref{eq:BW}, which is a  generalization of an earlier $\d=p=2$ inequality due to Brezis and Gallou\"et \cite{brezis_nonlinear_1980}. In the PDE literature, inequalities of this form are sometimes called Brezis-Gallou\"et inequalities or logarithmic interpolation. Independently, Hansson \cite{hansson_imbedding_1979}  proved an inequality expressed in potential-theoretic language that is equivalent to \eqref{eq:BW}, and an alternative proof is also given in \cite{engler_alternative_1989}. %Accordingly, \eqref{eq:BW} is known as the \emph{Brezis-Wainger-Hansson inequality}.
   %and \cite{} as some examples of subsequent extensions.

The Brezis-Wainger-Hansson inequality is a substitute for the false endpoint Sobolev embedding $\|f\|_{L^\infty} \lesssim \|f\|_{\dot{W}^{\frac{\d}{p},p}}$, except when $\d=p=1$. %Clearly, \eqref{eq:BW} cannot hold with $(a,q) = (\frac{\d}{p},p)$ because that would imply the false endpoint Sobolev estimate. One has various substitute estimates at this endpoint. For instance, the estimate is true if one replaces $L^\infty$ by BMO on the left-hand side.
Although functions in $\dot{W}^{\frac{\d}{p},p}$ are not bounded, they  are still exponentially integrable in a quantitative sense made precise by the famous Trudinger-Moser inequality \cite{trudinger_imbeddings_1967, moser_sharp_1970}. In some sense, the estimate \eqref{eq:BW} is dual to Trudinger-Moser (see \cite[Theorem A]{brezis_note_1980}).

%\com{MHR - Not sure if this is the best place for these last two paragraphs. I want to highlight Brezis's work early on, but it's a bit of a digression from ``our contributions.''}

%\subsection{Organization of paper}\label{ssec:introOrg}
%We briefly comment on the organization of the text.

\subsection{Notation}\label{ssec:introN}
Finally, let us review the basic notation used throughout the paper, mostly following the conventions of \cite{nguyen_mean-field_2022,rosenzweig_global--time_2023, rosenzweig_sharp_nodate, hess-childs_sharp_2025}.

Given nonnegative quantities $A$ and $B$, we write $A\lesssim B$ if there exists a constant $C>0$, independent of $A$ and $B$, such that $A\leq CB$. If $A \lesssim B$ and $B\lesssim A$, we write $A\sim B$. Throughout, $C$ will denote a generic constant that may change from line to line. %Also $(\cdot)_+$ denotes the positive part of a number.

$\N$ denotes the natural numbers excluding zero. {For $N\in\N$, we abbreviate $[N]\coloneqq \{1,\ldots,N\}$.} $\R_+$ denotes the positive reals. Given $x\in\R^\d$ and $r>0$, $B(x,r)$ and $\p B(x,r)$ respectively denote the ball and sphere centered at $x$ of radius $r$. Given a function $f$, we denote its support by $\supp f$. The notation $\nabla^{\otimes k}f$ denotes the $k$-tensor field with components $(\p_{i_1}\cdots\p_{i_k}f)_{1\leq i_1,\ldots,i_k\leq \d}$. {For $x\in\R$, $\jp{x}\coloneqq (1+|x|^2)^{1/2}$ is the Japanese bracket.}

$\P(\R^\d)$ denotes the space of Borel probability measures on $\R^\d$. If $\mu$ is absolutely continuous with respect to Lebesgue measure, we shall abuse notation by letting $\mu$ denote both the measure and its Lebesgue density. When the measure is clearly understood to be Lebesgue, we shall simply write $\int_{\R^{\d}}f$ instead of $\int_{\R^\d}fdx$.

We use the notation $\hat{f}$ or $\widehat{f}$ to denote the Fourier transform $\hat{f}(\xi) \coloneqq \int_{\R^\d}f(x)e^{-2\pi i \xi\cdot x}dx$ and let $\jp{\nab}$ and $\Dm$ denote the Fourier multipliers with symbols $\jp{2\pi \xi}$ and $|2\pi \xi|$, respectively.

Several function spaces appear in the paper.  $\Sc(\R^\d)$ and $\Sc'(\R^\d)$ denote the spaces of Schwartz functions and tempered distributions, respectively. For $s\in\R$ and $1<p<\infty$,  $W^{s,p}(\R^\d)$ denotes the inhomogeneous Bessel potential space  defined by
\begin{align}
    W^{s,p}\coloneqq \{f\in \Sc' : \jp{\nab}^sf \in L^p\}, \qquad \|f\|_{W^{s,p}} \coloneqq \|\jp{\nab}^sf\|_{L^p}.
\end{align}
When $s=k$ is an integer, $W^{s,p}$ agrees with the usual Sobolev space
\begin{align}
    W^{k,p} = \{f\in L^p : \nab^{\otimes j}f \in L^p, \ 0\le j\le k\}, \qquad \|f\|_{W^{k,p}} \coloneqq \sum_{j=0}^k \|\nab^{\otimes j}f\|_{L^p},
\end{align}
the latter, of course, making sense for the endpoints $p=1,\infty$ too. When $p=2$, we use the customary notation $W^{s,2}=H^s$. {The homogeneous Bessel potential space $\dot{W}^{s,p}(\R^\d)$ is defined by
\begin{align}
    \dot{W}^{s,p} \coloneqq \{f\in \Sc'/\mathscr{P} : \Dm^s f\in L^p\}, \qquad \|f\|_{\dot{W}^{s,p}} \coloneqq \|\Dm^s f\|_{L^p},
\end{align}
where $\mathscr{P}(\R^\d)$ denotes the space of polynomials on $\R^\d$ and $\Sc'/\mathscr{P}$ denotes the quotient space where we identify any two tempered distributions whose difference is a polynomial. This identification turns the seminorm $\|\cdot\|_{\dot{W}^{s,p}}$ into a norm. Similarly, for $s=k$ integer, $\dot{W}^{s,p}$ agrees with
\begin{align}
    \dot{W}^{k,p} = \{f\in\Sc'/\mathscr{P} : \nab^{\otimes k}f\in L^p\}, \qquad \|f\|_{\dot{W}^{k,p}} \coloneqq \|\nab^{\otimes k}f\|_{L^p}.
\end{align}
It is a basic fact that for $s \ge 0$, $f\in W^{s,p}$ iff $f\in L^p\cap\dot{W}^{s,p}$.} For $\theta>0$, we let $\mathcal{C}^\theta(\R^\d)$ denote the inhomogeneous H\"older-Zygmund space, which for $\theta\in (0,1]$ is defined by $f\in C_{loc}$ such that
\begin{align}
    \|f\|_{\mathcal{C}^\theta} \coloneqq \|f\|_{L^\infty} +  \sup_{|h|>0} \frac{\|f(\cdot+h)+f(\cdot-h)-2f(\cdot)\|_{L^\infty}}{|h|^{\theta}} < \infty,
\end{align}
and for $\theta>1$, we inductively define $\|f\|_{\mathcal{C}^\theta} \coloneqq \|f\|_{L^\infty} + \|\nab f\|_{\mathcal{C}^{\theta-1}}$. $\mathcal{C}^{\theta}$ coincides with the Besov space $B_{\infty,\infty}^\theta$, and the norms are equivalent (see \cite[Section 2.11]{bahouri_fourier_2011}). For noninteger values of $\theta$, it holds that $\mathcal{C}^\theta = C^{\lfloor{\theta}\rfloor, \{\theta\}}$, the space of $\lfloor{\theta}\rfloor$-times continuously differentiable functions such that the $\lfloor{\theta}\rfloor$-th derivative is $\{\theta\}$-H\"older continuous, which we endow with its usual norm. When $\theta$ is an integer, $\mathcal{C}^\theta$ is strictly larger than $W^{\theta,\infty}=C^{\theta-1,1}$. {The homogeneous H\"older-Zygmund space $\dot{\mathcal{C}}^\theta$ is defined to be the space of all $f\in C_{loc}$ such that
\begin{align}
    \|f\|_{\dot{\mathcal{C}}^\theta} \coloneqq \sup_{|h|>0} \frac{\|f(\cdot+h)+f(\cdot-h)-2f(\cdot)\|_{L^\infty}}{|h|^{\theta}} < \infty
\end{align}
when $\theta\in (0,1]$ and inductively for $\theta>1$ as before.}

As is customary in the literature, a superscript $\cdot$ denotes the homogeneous space/seminorm.

%$C^{k,\alpha}(\R^\d)$ denotes the inhomogeneous space of $k$-times differentiable functions on $\R^\d$ whose $k$-th derivative is $\alpha$-H\"{o}lder continuous, for $\al\in [0,1]$ (i.e. $\alpha=0$ is bounded and $\alpha=1$ is Lipschitz). As per convention, a $\dot{}$ superscript denotes the homogeneous space/seminorm.}

{%$W^{k,\infty}(\R^\d)$ denotes the inhomogeneous Sobolev space of functions on $\R^\d$ whose derivatives up to order $k$ are in $L^\infty$.
%$C^{k,\alpha}(\R^\d)$ denotes the inhomogeneous space of $k$-times differentiable functions on $\R^\d$ whose $k$-th derivative is $\alpha$-H\"{o}lder continuous, for $\al\in [0,1]$ (i.e. $\alpha=0$ is bounded and $\alpha=1$ is Lipschitz). As per convention, a $\dot{}$ superscript denotes the homogeneous space/seminorm.} $\Sc(\R^\d)$ and $\Sc'(\R^\d)$ denotes the space of Schwartz functions and the space of tempered distributions, respectively.

\subsection{Acknowledgments}\label{ssec:introAck}
The second author thanks the Courant Institute of Mathematical Sciences, NYU for its hospitality, where part of the work for this project was carried out during a visit in November 2021.} This paper is based upon work supported by the National Science Foundation under Grant No. DMS-2424139, while the first author was in residence at the Simons Laufer Mathematical Sciences Institute in Berkeley, California, during the Fall 2025 semester.

\section{Main results}\label{sec:MR}
In this section, we present the precise statements of our main results previously advertised in \cref{ssec:introObj} and provide some comments on their proofs.

    \subsection{BMO inequalities}\label{ssec:MRbmo}
	
	Our first result demonstrates that, in general, the Lipschitz control on the transport field in~\eqref{eq:CE} cannot be relaxed to $\|\nabla v\|_{BMO}$. \cref{thm:CEF1} below shows that, except for the 1D logarithmic endpoint $(\d,\s)=(1,0)$, such a replacement is false. When $(\d,\s)=(1,0)$, the commutator estimate \emph{does} in fact hold. %and we prove the following inequality.
	
    \begin{thm}\label{thm:CEF1}
    \
    \begin{enumerate}[(i)]
   \item\label{item:CEF11} Let $\d\ge 1$ and $-2<\s<\d$ with $(\d,\s) \ne (1,0)$. For every $M>0$, there exist a $C^\infty$ compactly supported vector field $v$ and zero-mean Schwartz functions $f$ and $g$ such that 
		\begin{equation}\label{eq:CEF1}
			\frac{\displaystyle \bigg|\int_{(\R^\d)^2}(v(x)-v(y))\cdot\nabla\g(x-y)f(x)g(y)\bigg|		}{\displaystyle \Big(\|\nabla v\|_{BMO}+\||\nabla|^{\frac{\d-\s}{2}} v\|_{L^{\frac{2\d}{\d-\s-2}}}\indic_{\s<\d-2}\Big)
				\|f\|_{\dot{H}^{\frac{\s-\d}{2}}}	\|g\|_{\dot{H}^{\frac{\s-\d}{2}}}}\geq M.
		\end{equation}

	\item\label{item:CEF12} If $(\d,\s)=(1,0)$, then there exists $C>0$ such that for all $v$ with $v'\in \BMO$ and all $f,g\in\mathcal{S}$ with zero mean, we have
		\begin{equation}\label{eq:BMOCE}
			\bigg|\int_{\R^2} \frac{v(x)-v(y)}{x-y}f(x)g(y)\bigg|\leq C\| v'\|_{BMO}\|f\|_{\dot{H}^{-\frac{1}{2}}}\|g\|_{\dot{H}^{-\frac{1}{2}}}.
		\end{equation}
    \end{enumerate}
	\end{thm}
	
Letting $H$ denote the Hilbert transform $ H f (x)\coloneqq \PV \frac{1}{\pi}\int_{\R} \frac{f(y)}{y-x}dy$, ~\eqref{eq:BMOCE} is equivalent to
	\begin{equation}
		\big\| [v,H]\big\|_{\dot{H}^{-1/2}\rightarrow \dot{H}^{1/2}}\lesssim \|v'\|_{BMO}.
	\end{equation}
Thus, \eqref{eq:BMOCE} may be viewed as a Sobolev-space Coifman--Rochberg--Weiss-type commutator theorem \cite{coifman_factorization_1976}, which in its most elementary form says that $\|\comm{v}{H}\|_{L^p\rightarrow L^p} \lesssim \|v\|_{BMO}$.
	
The proofs of \cref{thm:CEF1}\ref{item:CEF11}, \ref{item:CEF12} are presented  in \Cref{ssec:CEX11,ssec:CEX12}, respectively. The first assertion of the theorem is shown through an explicit counterexample. We construct a velocity field $v$ with $\nabla v\in \BMO$ and a mild logarithmic singularity at the origin and test against a family of rescaled, compactly supported, zero-mean functions $f_r$. The commutator grows like $\log\log(1/r)$ as $r\to0$, while the denominator in \eqref{eq:CEF1} remains uniformly bounded, yielding an arbitrarily large ratio. The proof of the second assertion essentially involves appropriately integrating by parts and grouping terms into components that can be individually bounded via classical Kenig-Ponce-Vega-type commutator estimates \cite{kenig_well-posedness_1993,li_kato-ponce_2019}. %Here we also use an equivalent formulation of~\eqref{eq:CE} for Schwartz functions with Fourier transforms supported away from $0$.

\subsection{Regularity trade-off}\label{ssec:MRreg}
	
The BMO counterexample indicates that the Lipschitz control $\|\nabla v\|_{L^\infty}$ plays an essential role in \eqref{eq:CE},
so we do not attempt to weaken it further. Our second main quantifies the trade-off between the regularity of the interaction potential and the high-frequency regularity of the transport field required to control the commutator in the natural energy norm. Since this phenomenon is not specific to Riesz potentials, we work with a larger class of potentials for which the interaction energy defines a positive definite form.

	\begin{myass}[Admissible interaction potentials]\label{ass:pot}
		Suppose that $\g:\R^\d\setminus\{0\}\rightarrow \R$ satisfies the following:
		\begin{enumerate}
			\item $\g\in C^\infty(\R^\d\setminus\{0\})$,
			\item $\g$ is radially symmetric,\footnote{Throughout the paper, we use radial symmetry to abuse notation by writing $\g(|x|) = \g(x)$.}
			\item $\g$ is conditionally positive definite (see \cref{fn:cpd_def}) and $\hat{\g}$ nonincreasing.
		\end{enumerate}
	\end{myass}

	Associated to $\g$, we define the energy seminorm
    \begin{align}\label{eq:g_norm}
		\|f\|_{\g}^2 \coloneqq \int_{(\R^\d)^2}\g(x-y)f(x)f(y) = \int_{\R^\d}\hat{\g}|\hat{f}|^2,
    \end{align}
    which we note is nonnegative (possibly $+\infty$) for zero-mean $f$. When dealing with inhomogeneous potentials, it is convenient to introduce the notation $\g_t(x) \coloneqq \g(x/t)$ for $t>0$.
    %The additional scale parameter $t$ is harmless for homogeneous kernels (e.g.\ Riesz), but is natural for inhomogeneous kernels (see Remark~\ref{rem:Gauss}).

	%\begin{mydef} For $\g$ satisfying Assumption~\ref{ass:pot} and $t>0$, we define
	%	\begin{equation}\label{eq:g_norm}
	%		\|f\|_{\g,t}^2
	%		:=\int_{(\R^\d)^2}\g\Big(\frac{x-y}{t}\Big)f(x)f(y)\,dx\,dy
	%		=t^\d\int_{\R^\d}\hat\g(t\xi)\,|\hat f(\xi)|^2\,d\xi.
	%	\end{equation}
	%\end{mydef}
	
	Given a function $\sigma:[0,\infty)\to(0,\infty]$, we denote the Fourier multiplier with symbol $\sigma$ by
	$\sigma(|\nabla|)f\coloneqq  \mathcal{F}^{-1}\!\big(\sigma(|\xi|)\hat f(\xi)\big)$. \cref{thm:CEF2} below shows that if $\sigma$ grows too slowly relative to $\hat\g$ at high frequency, then the commutator cannot be controlled by the energy norm $\|\cdot\|_{\g_t}$ and the velocity field norm $\|\sigma(|\nabla|)v\|_{L^p}$.
	
	\begin{thm}\label{thm:CEF2}
		Suppose that $\g$ satisfies \cref{ass:pot}. Fix $t>0$ and $p\in[1,\infty]$.
		Assume furthermore that
		\begin{equation}\label{eq:decay}
			r^2\hat\g(r)\to 0 \qquad \text{as } r\to\infty
		\end{equation}
		and $\sigma:[0,\infty)\to [0,\infty]$ satisfies
		\begin{align}
			\lim_{r\to\infty} \sqrt{\hat\g(tr)}\|\sigma\|_{L^\infty([r-1,r+1])}=0, \qquad 2\le p\le \infty, \label{eq:sig_cond}\\
            \lim_{r\rightarrow\infty} \sqrt{\hat\g(tr)}\max_{n\le \d+1}\|\p_r^n\sigma\|_{L^\infty([r-1,r+1])}=0, \qquad 1\le p<2 .\label{eq:sig_cond'}
		\end{align}
		Then for every $M\ge 0$, there exist a $C^\infty$ vector field $v$ and zero-mean Schwartz functions
		$f,g$ such that
		\begin{equation}\label{eq:GCEF}
			\frac{\displaystyle
				\bigg|\int_{(\R^\d)^2} (v(x)-v(y))\cdot\nabla\g(x-y)f(x)g(y)\bigg|
			}{\displaystyle
				\Big(\|\nabla v\|_{L^\infty}+\|\sigma(|\nabla|)v\|_{L^p}\Big)
				\|f\|_{\g_t}\|g\|_{\g_t}
			}\ \ge\ M.
		\end{equation}
	\end{thm}
	
	The decay condition \eqref{eq:decay} forces a genuine trade-off: without it, one can be in a regime where a commutator estimate holds with $\|\nabla v\|_{L^\infty}$ alone. Condition \eqref{eq:sig_cond} means that on unit-width
	frequency shells $|\xi|\sim r$, the multiplier $\sigma(r)$ grows strictly slower than
	$\hat\g(tr)^{-1/2}$; and \eqref{eq:sig_cond'} extends this growth condition to derivatives of $\sigma$ up to high enough order. Thus, \cref{thm:CEF2} may be read as saying that any commutator bound measured in $\|\cdot\|_{\g_t}$ forces $v$ to possess at least ``as many derivatives'' as are needed to compensate for $\sqrt{\hat\g}$ at high frequency.

    Note that since $\hat\g$ is nonincreasing, if \eqref{eq:sig_cond} holds for some $t_*>0$, then it holds for all $t\ge t_*$. Furthermore, if $\hat\g(\xi)$ is comparable to a polynomial in $\xi$, then if \eqref{eq:sig_cond} holds for some $t>0$, it holds for any $t>0$. When $1\le p<2$, some regularity of the symbol is needed for $\sigma(\Dm)v\in L^p$, but the condition \eqref{eq:sig_cond'} can almost certainly be further weakened in terms of the number of derivatives (cf. the assumptions on the symbol in the H\"ormander-Mikhlin multiplier theorem). Lastly, considering unit-width frequency shells $[r-1,r+1]$ is not essential: one could consider shells of any fixed radius mutatis mutandis.
    
	Specializing \cref{thm:CEF2} to Riesz potentials $\g$ defined by \eqref{eq:gmod}, one has $\hat\g(\xi) \propto |\xi|^{\s-\d}$. The condition \eqref{eq:decay} is satisfied if $\s<\d-2$, i.e. $\g$ is sub-Coulomb. The condition \eqref{eq:sig_cond} is equivalent to $\sup_{\tau\in [r-1,r+1]}\sigma(\tau) \ll |r|^{\frac{\d-\s}{2}}$ as $r\rightarrow\infty$. In particular, taking $\sigma(\tau)=\tau^{\frac\as2}$, we see that among the scaling-invariant norms $\|\Dm^{\frac\as2} v\|_{L^{p}}$ with $p=\frac{2\d}{\as-2}$, the minimal exponent for which the commutator estimate holds is $\as=\d-\s$. Decreasing $\as$, equivalently increasing $p$, yields a strictly weaker norm for which the commutator estimate fails.
    
    %so the threshold $\hat\g^{-1/2}$ corresponds to $|\xi|^{(\d-\s)/2}$. In particular, among scaling-invariant norms $\||\nabla|^\mathsf{a} v\|_{L^p}$ with $\mathsf{a}-\frac{\d}{p}=1$, the minimal exponent for which the commutator estimate holds is $\mathsf{a}=\frac{\d-\s}{2}$, and lowering $\mathsf{a}$ (hence increasing $p$) yields a strictly weaker norm for which the commutator estimate fails.

    Finally, to emphasize this point of regularity trade-off, we note that \cref{thm:CEF2} implies that $v$ must be $C^\infty$ for the commutator estimate to hold for any interaction potential $\g$ with $\hat\g$ decaying super-polynomially. Indeed, if for any $n\ge 0$, $|\hat\g(\xi)| \lesssim_n |\xi|^{-n}$, then \eqref{eq:GCEF} holds for any $\sigma(\Dm) = \Dm^{n}$. In particular, if $\hat\g$ decays sub-exponentially, then by the Paley-Wiener-Schwartz theorem, $v$ must be at least analytic for a commutator estimate to hold. The regularity requirements are even more stringent when $\hat\g$ has sub-Gaussian decay: if a commutator estimate were to hold, $v$ must be super-analytic. %potentials  Gaussian potentials (which correspond to infinitely negative order differential operators). 	In particular, if a commutator estimate were to hold for a Gaussian kernel, then $v$ must be more regular than analytic. Indeed, $e^{\pi^2 b|\nabla|^2}v\in L^\infty$ implies that $v$ is Gevrey of index $1/2$ (in $L^\infty$).

    We summarize the above described consequences of \cref{thm:CEF2} with the following corollary.
	
	\begin{cor}\label{cor:CEF2}
    \
    \begin{enumerate}[(i)]
		\item\label{item:CEF21} Suppose that $\g$ is given by ~\eqref{eq:gmod} for $\s<\d-2$. Then for all $2<\mathsf{a}<\d-\s$ and $M>0$, there exist a $C^\infty$ vector field $v$ and zero-mean Schwartz functions $f$ and $g$ such that
		\begin{equation}\label{eq:CEF2}
			\frac{\displaystyle \bigg|\int_{(\R^\d)^2} (v(x)-v(y))\cdot\nabla\g(x-y)f(x)g(y)\bigg|
			}{\displaystyle \Big(\|\nabla v\|_{L^\infty}+\||\nabla|^{\frac\as2} v\|_{L^{\frac{2\d}{\as-2}}}\Big)
				\|f\|_{\dot{H}^{\frac{\s-\d}{2}}}\|g\|_{\dot{H}^{\frac{\s-\d}{2}}}}\geq M.
		\end{equation}
        \item\label{item:CEF22} Suppose that $\g$ is admissible and $\hat\g$ has super-polynomial decay. Then for any $n,M,t>0$, there exist a $C^\infty$ vector field $v$ and zero-mean Schwartz functions $f$ and $g$ such that
        \begin{align}
            \frac{\displaystyle \bigg|\int_{(\R^\d)^2} (v(x)-v(y))\cdot\nabla\g(x-y)f(x)g(y)\bigg|
			}{\displaystyle \|\jp{\nab}^n v\|_{L^1}
				\|f\|_{\g_t}\|g\|_{\g_t}}\geq M.
        \end{align}

        \item\label{item:CEF23} Suppose that $\g$ is admissible and for some $c,m>0$, $|\hat\g(\xi)| \lesssim e^{-c|\xi|^m}$. Then for any $b<\frac{ct^m}{2}$ and $M>0$, there exist a $C^\infty$ vector field $v$ and zero-mean Schwartz functions $f,g$ such that
        \begin{align}
            \frac{\displaystyle \bigg|\int_{(\R^\d)^2} (v(x)-v(y))\cdot\nabla\g(x-y)f(x)g(y)y\bigg|
			}{\|e^{b\Dm^m}v\|_{L^1}\|f\|_{\g_t} \|g\|_{\g_t}}\geq M.
        \end{align}
	\end{enumerate}
    \end{cor}

    	\begin{remark}\label{rem:Gauss}
        %Using that $|v(x)-v(y)|\le \|\nabla v\|_{L^\infty}|x-y|$ and the Gaussian estimate
		%$|z|\,|\nabla\phi_a(z)|\lesssim \phi_a(z/t)$ for $t>1$, we obtain for $f,g\ge 0$ the bound
        Suppose $\g$ is such that there exists $t>0$ so that $|x||\nab\g(x)| \lesssim \g(x/t)$. Such a condition is obviously satisfied if $\g$ is polynomial, exponential, or Gaussian. Then using $|v(x)-v(y)|\le \|\nabla v\|_{L^\infty}|x-y|$,  we obtain for $f,g\ge 0$ the bound
		\begin{align}
		\notag\bigg|\int_{(\R^\d)^2}(v(x)-v(y))\cdot\nabla\g(x-y)f(x)g(y)\bigg|
			&\lesssim \|\nabla v\|_{L^\infty}\int_{(\R^\d)^2}\g\Big(\frac{x-y}{t}\Big)f(x)g(y)
			\\ & \lesssim  \|\nabla v\|_{L^\infty} \|f\|_{\g_t} \|g\|_{\g_t}.
		\end{align}
		This observation shows that under the assumption $f,g$ have a definite sign, only Lipschitz regularity is needed for a commutator estimate to hold. In light of \cref{cor:CEF2}, a  general commutator estimate for $f,g$ without a definite sign necessarily requires much higher regularity on $v$.
	\end{remark}

    The proof of \cref{thm:CEF2} is given in \cref{ssec:CEX2}. We construct $v,f,g$ so that their Fourier supports sit on a few carefully chosen (and well separated) frequency shells. This forces one contribution to the commutator to be ``off-resonant'' and therefore small while the other contribution is ``resonant'' and stays bounded below. At the same time, the same frequency localization makes the velocity norm essentially depend on $\sup_{r\in[k-1,k+1]}\sigma(r)$ when $2\le p\le \infty$, while $\|f\|_{\g_t}$ scales like $\sqrt{\hat\g(tk)}$; after normalizing by $\|f\|_{\g_t}$ this produces a factor $1/\sqrt{\hat\g(tk)}$ in the final ratio. Letting $k\to\infty$ then produces the desired ratio growth.

\subsection{Defective commutator estimates and mean-field convergence}\label{ssec:MRmf}
We now come to our last set of results on mean-field convergence when the solution $\mu^t$ of the mean-field equation  \eqref{eq:MFlim}, for $\g$ as in \eqref{eq:gmod}, belongs to the critical space $\dot{W}^{\frac{\d}{p}+\s+2-\d,p}$ for $1<p<\infty$. 

Although the commutator estimate fails when $v$ is not Lipschitz (except for $(\d,\s)=(1,0)$) by \cref{thm:CEF1}\ref{item:CEF11}, as is the case for $v=\M\nab\g\ast\mu^t$  when $\mu^t\in \dot{W}^{\frac{\d}{p}+\s+2-\d,p}$, a defective commutator estimate still holds. This is the content of \cref{thm:FI} below. We exclude the case $\s=0$ from the theorem because doing so simplifies the presentation and this case is already covered by the argument of \cite{rosenzweig_mean-field_2022-1}.\footnote{Strictly speaking, \cite{rosenzweig_mean-field_2022-1} only considers the case $(\d,\s) = (2,0)$, but an examination of the argument shows that it applies to all log cases, replacing $L^\infty$ by the appropriate critical space.}

{
\begin{thm}\label{thm:FI}
Assume that $-2<\s<\d$ and $\s\ne 0$. Let $v\in \dot{W}^{\frac{\d}{p}+1,p}$ for $1<p<\infty$ and $\mu\in L^1$ with $\int_{\R^\d}d\mu =1$ and $\int_{(\R^\d)^2}|\g|(x-y)d|\mu|^{\otimes2}<\infty$. Further, assume the following:
    \begin{enumerate}[(a)]
        \item if $-2<\s\le-1$, then $\int_{\R^\d}|x|^{|\s|-1}d|\mu| < \infty$;
        \item if $-1<\s<0$, then $\int_{\R^\d}|x|^{r}d|\mu| < \infty$ for some $r<\frac{|\s|}{2}$;\footnote{Obviously, by interpolation, if $\int_{\R^\d}|x|^{r'}d|\mu|<\infty$ for some $r'\ge \frac{r}{2}$, then $\int_{\R^\d}|x|^{r}d|\mu|<\infty$ for any $0\le r\le r'$, in particular for some $r<\frac{|\s|}{2}$.}
        \item if $-1\le \s<\d-1$, then $\mu\in L^q$ for some $\frac{\d}{\d-\s-1}<q\leq \infty$ ($q=1$ is allowed if $\s=-1$);
        \item if $\s\ge \d-1$, then $\mu\in \dot{\mathcal{C}}^{\theta}$ for some $\theta>\s+1-\d$.\footnote{If $\mu\in L^1 \cap \dot{\mathcal{C}}^\theta$ for some $\theta>0$, then necessarily $\mu\in L^p$ for any $1\le p\le \infty$ because of the interpolation estimate $\|f\|_{L^p} \le \|f\|_{L^1}^{1-\frac{\d(p-1)}{p(\d+\theta)}}\|f\|_{\dot{\mathcal{C}}^\theta}^{\frac{\d(p-1)}{p(\d+\theta)}}$.}
    \end{enumerate}
    
    Then for any pairwise distinct configuration $\XN\in (\R^\d)^N$ and $\ep>0$,  the following holds:
    \\
    If $\max(\d-2,0)\le \s<\d$, then
    \begin{multline}\label{eq:FIdcsupC}
        |\As_1[\XN,\mu,v]| \le C_p\|v\|_{\dot{W}^{\frac{\d}{p}+1,p}}(1+|\log\ep|)^{1-\frac1p}\Big(\Fr_N(\XN,\mu) + C_q\|\mu\|_{L^q}\la^{\frac{\d(q-1)}{q}-\s} \Big)\\
        + {CN^{\frac{2(\s+1)}{\s}-1}\|v\|_{\dot{\mathcal{C}}^1}\ep}\Big(\Fr_N(\XN,\mu)  + C_p\|\mu\|_{L^p}\la^{\frac{\d(p-1)}{p}-\s}\Big)^{\frac{\s+1}{\s}}\\
        +\ep(1+\|\mu\|_{L^1})\|v\|_{\dot{\mathcal{C}}^1} \Big(  C_q\|\mu\|_{L^1}^{1-\frac{(\s+1)q}{\d(q-1)}} \|\mu\|_{L^q}^{\frac{(\s+1)q}{\d(q-1)}} +  C_\theta\|\mu\|_{L^1}^{1-\frac{\s-\d+1}{\theta}} \big(\|\mu\|_{\dot{\mathcal{C}}^\theta}^{\frac{\s-\d+1}{\theta}} + \ep^{\d-\s-1}\|\mu\|_{L^\infty}^{\frac{\s-\d+1}{\theta}}\big)\indic_{\s\ge \d-1}\Big).
    \end{multline}
    If $0<\s<\d-2$, then
    \begin{multline}\label{eq:FIdcsubC1}
        |\As_1[\XN,\mu,v]| \\
        \le \Big(C_p\|v\|_{\dot{W}^{\frac{\d}{p}+1,p}}(1+|\log\ep|)^{1-\frac1p}  + C_a\|\Dm^{\frac{\as}{2}}v\|_{L^{\frac{2\d}{\as-2}}}\indic_{\substack{\as>2}}\Big) \Big(\Fr_N(\XN,\mu) + C_q\|\mu\|_{L^q}\la^{\frac{\d(q-1)}{q}-\s} \Big) \\
        + {CN^{\frac{2(\s+1)}{\s}-1}\|v\|_{\dot{\mathcal{C}}^1}\ep}\Big(\Fr_N(\XN,\mu)  + C_p\|\mu\|_{L^p}\la^{\frac{\d(p-1)}{p}-\s}\Big)^{\frac{\s+1}{\s}} 
        +\ep(1+\|\mu\|_{L^1})\|v\|_{\dot{\mathcal{C}}^1}  C_q\|\mu\|_{L^1}^{1-\frac{(\s+1)q}{\d(q-1)}} \|\mu\|_{L^q}^{\frac{(\s+1)q}{\d(q-1)}}
    \end{multline}
    or
    \begin{multline}\label{eq:FIdcsubC2}
        |\As_1[\XN,\mu,v]| \le  \Big(C_p\|v\|_{\dot{W}^{\frac{\d}{p}+1,p}}(1+|\log\ep|)^{1-\frac1p}  + C\|\Dm^{\frac{\d-\s}{2}}v\|_{L^{\frac{2\d}{\d-\s-2}}}\Big)\Big(\Fr_N(\XN,\mu) + C_{{q}}\|\mu\|_{L^{{q}}}\ka^{\frac{\d({q}-1)}{{q}}-\s}\Big) \\
        + {CN^{\frac{2(\s+1)}{\s}-1}\|v\|_{\dot{\mathcal{C}}^1}\ep}\Big(\Fr_N(\XN,\mu)  + C_p\|\mu\|_{L^p}\la^{\frac{\d(p-1)}{p}-\s}\Big)^{\frac{\s+1}{\s}} +\ep(1+\|\mu\|_{L^1})\|v\|_{\dot{\mathcal{C}}^1}  C_q\|\mu\|_{L^1}^{1-\frac{(\s+1)q}{\d(q-1)}} \|\mu\|_{L^q}^{\frac{(\s+1)q}{\d(q-1)}}.
    \end{multline}
    If $-2<\s<0$, then
    \begin{multline}\label{eq:FIdcnonsing}
       |\As_1[\XN,\mu,v]|\le  \Big(C_p\|v\|_{\dot{W}^{\frac{\d}{p}+1,p}}(1+|\log\ep|)^{1-\frac1p}  + C\|\Dm^{\frac{\d-\s}{2}}v\|_{L^{\frac{2\d}{\d-\s-2}}}\Big)\Fr_N(\XN,\mu) \\
        +\Bigg(C_{\vartheta,\vartheta'}\|v\|_{\dot{\mathcal{C}}^1}\ep^{\vartheta'}\Big(\int_{\R^\d}|x|^{|\s|-\vartheta}d\mu + (1+\|\mu\|_{L^1})\Fr_N^{\frac{|\s|-\vartheta}{|\s|}}(\XN,\mu)\Big) + C\|v\|_{\dot{\mathcal{C}}^1}\ep\Bigg)\indic_{-1< \s<0}\\
        +  {C\|v\|_{\dot{\mathcal{C}}^1}\ep}\Bigg(\int_{\R^\d}|x|^{|\s|-1}d\mu +  (1+\|\mu\|_{L^1})\Fr_N^{\frac{|\s|-1}{|\s|}}(\XN,\mu)\Bigg) \indic_{-2<\s\le-1}\\
        +\ep(1+\|\mu\|_{L^1})\|v\|_{\dot{\mathcal{C}}^1}\Big( \int_{\R^\d}|x|^{|\s|-1}d\mu + \Fr_N^{\frac{|\s|-1}{|\s|}}(\XN,\mu)\indic_{-2<\s\le -1}  + C_q\|\mu\|_{L^1}^{1-\frac{(\s+1)q}{\d(q-1)}} \|\mu\|_{L^q}^{\frac{(\s+1)q}{\d(q-1)}}\indic_{-1\le \s<0}\Big).
    \end{multline}
Above, {$\vartheta \in [|\s|-r,|\s|)$}, $0<\vartheta'<\vartheta$, $C=C(\d,\s)>0$ and $C_\vartheta, C_p,C_q,C_a,C_\theta>0$ additionally depend on $\vartheta, p,q,a,\theta$, respectively.
\end{thm}
}

\begin{remark}\label{rem:dcompinfty}
The endpoint $p=\infty$ also holds if we replace $\dot{W}^{1,\infty}$ with the Zygmund space $\dot{\mathcal{C}}^1$. The reader should remember that, per our notation (recall \cref{ssec:introN}), the limit as $p\rightarrow \infty$ of $\dot{W}^{\frac{d}{p}+1,p}$ is not the homogeneous Lipschitz space, but rather the ill-behaved space $\{f\in\Sc' : \jp{\nab}f \in L^\infty\}$, which does not coincide with the Lipschitz space, and is usually not considered in the harmonic analysis literature. See \cref{ssec:RQmf} for further comments on this point.
\end{remark}

With this defective commutator estimate in hand, one can apply it to $v=\mu^t$ to obtain a Gr\"onwall-type bound for the modulated energy (a more precise estimate is given during the proof of \cref{thm:MF} in \cref{ssec:BWmp}). This bound implies mean-field convergence provided that the initial modulated energy vanishes sufficiently fast as $N\rightarrow\infty$, an assumption which holds a.s. for most interesting choices of randomization of the initial particle positions (see \cref{rem:randID} below).  

{
\begin{thm}\label{thm:MF}
Let $\d\ge 1$ and $-2<\s<\d$ and $\s \ne 0$. Suppose that for $T>0$, $\int_0^T(\|\mathsf{V}^t\|_{\dot{W}^{\frac{\d}p+1,p}} + \|\Dm^{\frac{\d-\s}{2}}\mathsf{V}^t\|_{L^{\frac{2\d}{\d-\s-2}}}\indic_{\s<\d-2})dt < \infty$. Let $\mu$ be a solution to the mean-field equation \eqref{eq:MFlim} in $L^\infty([0,T], \mathcal{P}(\R^\d)\cap \dot{W}^{\frac{\d}{p}+\s+2-\d,p}(\R^\d))$ for $1<p\le \frac{2\d}{\d-\s-2}$. Suppose further that $\int_{\R^\d}|x|^{\frac{|\s|}{2}}d\mu^t<\infty$ for every $t\in [0,T]$.\footnote{This condition ensures that the modulated energy/MMD is finite (see \cite[Example 4]{modeste_characterization_2024}) and that all our moment assumptions in the statement of the theorem are satisfied. Through a Gr\"onwall argument, the reader may check that if $\mu^0$ satisfies this moment condition, then $\mu^t$ does uniformly in $[0,T].$}

For $q>\frac{\d}{\d-\s-1}$, $\vartheta \in (\frac{|\s|}{2},|\s|)$, and $0<\vartheta'<\vartheta$, let
\begin{align}
    \mathscr{E}^t \coloneqq \Fr_N(\XN^t,\mu^t)  + \sup_{t\le \tau\le T}\zeta^\tau,
\end{align}
where
\begin{align}
\zeta^\tau \coloneqq \ep\begin{cases}\Big(C_q\|\mu^\tau\|_{L^q}\la_\tau^{\frac{\d(q-1)}{q}-\s} + C_q\|\mu^\tau\|_{L^q}^{\frac{(\s+1)q}{\d(q-1)}}\indic_{\s<\d-1}\\
    + C(\|\mu^\tau\|_{\dot{\mathcal{C}}^{\s+2-\d}}^{\frac{\s+1-\d}{\s+2-\d}} + \ep^{\d-\s-1}\|\mu^\tau\|_{L^\infty}^{\frac{\s-\d+1}{\s+2-\d}}) \indic_{\s\ge\d-1}\Big), & {\max(\d-2,0)\le \s<\d} \\  C_q\Big(\|\mu^\tau\|_{L^{{q}}}\ka_\tau^{\frac{\d({q}-1)}{{q}}-\s} + \|\mu^\tau\|_{L^q}^{\frac{(\s+1)q}{\d(q-1)}}\Big) , & {0<\s<\d-2} \\ C_{\vartheta,\vartheta'}\ep^{\vartheta'-1}\int_{\R^\d}|x|^{|\s|-\vartheta}d\mu^\tau + C(1+C_q\|\mu^\tau\|_{L^q}^{\frac{(\s+1)q}{\d(q-1)}}), & {-1<\s<0} \\ C\int_{\R^\d}|x|^{|\s|-1}d\mu^\tau, & {-2<\s\le -1}, \end{cases}
\end{align}
and for $\al,\delta>0$,\footnote{When $0<\s<\d$, the definition of $\ep$ is given implicitly. See the proof of \cref{thm:MF} for why the implicit equation has a solution.}
\begin{align}\label{eq:MFepdef}
    \ep \coloneqq
    \begin{cases}
        \delta N^{-1-\frac{2(\s+1)}{\s}-\al}(\mathscr{E}^0)^{-\frac{1}{\s}}, & {0 < \s<\d}\\
        N^{-\al}, & {-2<\s<0}.
    \end{cases}
\end{align}

There exists a $\delta>0$ depending only on $\d,\s,p,\al,\int_0^T \|\mu^t\|_{\dot{W}^{\frac{\d}{p}+\s+2-\d,p}}dt$, such that for any solution $\XN$ of the microscopic system \eqref{eq:MFode}, the following holds for $t\in [0,T]$: if $0< \s<\d$,
\begin{align}
     \mathscr{E}^t \le C\mathscr{E}^0\exp\Big(\int_0^{t} C_p\|u^{t'}\|_{\dot{W}^{\frac{\d}{p}+1,p}}\Big(1+|\log(\delta N^{1-\frac{2(\s+1)}{\s}-\al}(\mathscr{E}^0)^{-\frac1\s})|\Big)^{1-\frac1p}dt'\Big),
\end{align}
if $-1<\s<0$, then
\begin{align}
    (\mathscr{E}^t)^{\frac{\vartheta}{|\s|}} \le \exp\Big({\frac{C_p\vartheta}{|\s|}\int_0^t \|u^{\tau}\|_{\dot{W}^{\frac{\d}{p}+1,p}}(1+|\log N^{-\al}|)^{1-\frac1p} d\tau}\Big)\bigg((\mathscr{E}^0)^{\frac{\vartheta}{|\s|}} + \frac{\vartheta}{|\s|}\int_0^t C_{\vartheta,\vartheta'}\|u^{t'}\|_{\dot{\mathcal{C}}^1}\ep^{\vartheta}dt'\bigg),
\end{align}
and if $-2<\s\le -1$, then
\begin{align}
    (\mathscr{E}^t)^{\frac{1}{|\s|}} \le \exp\Big({\frac{C_p}{|\s|}\int_0^t \|u^\tau\|_{\dot{W}^{\frac{\d}{p}+1,p}}(1+|\log N^{-\al}|)^{1-\frac1p} d\tau}\Big)\bigg((\mathscr{E}^0)^{\frac{1}{|\s|}} + \frac{1}{|\s|}\int_0^t C\|u^{t'}\|_{\dot{\mathcal{C}}^1}\ep dt'\bigg).
\end{align}
Above, $C>0$ depends only on $\d,\s$, and $C_{\vartheta,\vartheta'},C_q,C_p$ additionally depend on $(\vartheta,\vartheta'),q,p$, respectively.

In particular, if $\Fr_N(\XN^0,\mu) = O(N^{-m})$ as $N\rightarrow\infty$, for some $m>0$, then  $\mu_N^t \rightarrow \mu^t$ as $N\rightarrow\infty$ in $H^{-r}(\R^\d)$ uniformly on $[0,T]$ for any $r>\frac\d2$.
%if $|\Fr_N(\XN^0,\mu^0)|e^{A|\log N|^{1-\frac1p}}\rightarrow 0$ as $N\rightarrow\infty$, for any $A>0$, then $\mu_N^t \rightarrow \mu^t$ as $N\rightarrow\infty$ in $H^{-r}(\R^\d)$ uniformly on $[0,T]$ for any $r>\frac\d2$.
\end{thm}
}

\begin{remark}\label{rem:randID}
    A natural way to produce initial point configurations $X_N$ such that $|\Fr_N(\XN,\mu)| \le CN^{-m}$ for $C,m>0$ independent of $N$ is by taking $X_N$ to be distributed according to a \emph{modulated Gibbs measure} (a terminology introduced in \cite{rosenzweig_modulated_2025})
    \begin{align}
        d\Q_{N,\be}(\mu)(\XN) \coloneqq \frac{e^{-\be N\Fr_N(\XN,\mu)}d\mu^{\otimes N}(\XN)}{\K_{N,\be}(\mu)}, \quad \K_{N,\be}(\mu)\coloneqq \int_{(\R^\d)^N}e^{-\be N\Fr_N(\XN,\mu)}d\mu^{\otimes N}(\XN),
    \end{align}
    where $\be\in [0,\infty]$ is a parameter, which may be interpreted as inverse temperature. The extreme $\be=0$ corresponds to $\mu$-iid points, while the extreme $\be=\infty$ formally corresponds to points that are uniformly distributed among the set of minimizers of $\XN\mapsto \Fr_N(\XN,\mu)$. When $\be=0$, one has $\E[\Fr_N(\XN,\mu)] = CN^{-1}$ and in fact, this scaling holds with high probability. When $\be=\infty$, one has $\min \Fr_N(\XN,\mu) = (-1)^{\sgn(\s)} C N^{\frac{\s}{\d}-1}$. One can produce scaling laws in $N$ interpolating between these two extremes by varying $\beta$. We  refer to \cite[Chapter 5]{serfaty_lectures_2024} for further details.
\end{remark}

\begin{remark}
    There is a large literature on the well-posedness of equations of the form \eqref{eq:MFlim}. We refer to the introduction of \cite{serfaty_mean_2020} for some comments in this direction. Here, we briefly mention that when $\s<\d-2$, solutions of the mean-field equation \eqref{eq:MFlim} are unique and global for probability density initial data in the critical space $L^{\frac{\d}{\d-\s-2}}$. This follows from essentially the same argument as for the Coulomb case $\s=\d-2$ with initial density in $L^\infty$ (e.g. see \cite{yudovich_non-stationary_1963,lin_hydrodynamic_2000,serfaty_mean_2014}).
\end{remark}

\cref{sec:BW} is devoted to the proofs of \Cref{thm:FI,thm:MF}. The defective commutator estimate, which is the workhorse, proceeds along lines similar to \cite{rosenzweig_mean-field_2022-1, rosenzweig_mean-field_2022}. Namely, we mollify the vector field $v$ at some scale $\ep$ and apply \cref{thm:FI'} with $v_\ep$. We then need to separately estimate the error term $\As_1[\XN,\mu,v-v_\ep]$ in terms of $\Fr_N(\XN,\mu),\ep,N$. Of course, as \cref{thm:FI'} requires at least Lipschitz regularity, which is not satisfied when $v\in\dot{W}^{\frac{\d}{p}+1,p}$, we have to pay a price in terms of divergence as $\ep\rightarrow 0$ (i.e. the defect). Simply repeating the arguments from \cite{rosenzweig_mean-field_2022-1, rosenzweig_mean-field_2022} does not yield the estimates of \cref{thm:FI}, and consequently also not yield \cref{thm:MF}. \cite{rosenzweig_mean-field_2022-1} is specific to the log case, which is well-suited to the log-Lipschitz regularity of vector fields in $\dot{W}^{\frac{\d}{p},p}$ (more generally, in the Zygmund space $\dot{\mathcal{C}}^1$). While the cruder argument in \cite{rosenzweig_mean-field_2022}, devised to handle the non-logarithmic nature of the Coulomb potential when $\d\ne 2$, leads to factors $\propto\log N$ in front of $\Fr_N(\XN^t,\mu^t)$. Such factors are only acceptable if one restricts to short times. 

As advertised in \cref{ssec:introObj}, the key new ingredient in our proof is the Brezis-Wainger-Hansson inequality \eqref{eq:BW}. More precisely, a corollary of \eqref{eq:BW}, obtained through the classical Morrey argument, is that if $1<p<\infty$, then for any $f\in\Sc(\R^\d)$,
\begin{align}\label{eq:BWll}
    \forall x,y\in\R^\d, \qquad |f(x)-f(y)| \le C \| f\|_{\dot{W}^{\frac{\d}{p}+1,p}}|x-y|\Big(1+|\log|x-y||\Big)^{1-\frac{1}{p}},
\end{align}
where $C=C(\d,p)>0$.\footnote{Strictly speaking, \cite{brezis_note_1980} shows the inequality with the inhomogeneous Sobolev norm $\| f\|_{{W}^{\frac{\d}{p}+1,p}}$ on the right-hand side. Replacing $f$ with $f_\la\coloneqq f(\la\cdot)$ and $(x,y)$ with $\la^{-1}(x,y)$, then letting $\la\rightarrow\infty$ yields \eqref{eq:BWll}.} This inequality may be rephrased as saying $L^\infty\cap \dot{W}^{\frac{\d}{p}+1,p}$ embeds into the log-Lipschitz space $\mathrm{Lip}_{\infty,\infty}^{(1,-1+\frac1p)}$ \cite{janson_generalizations_1980,MR2262450}. %If $p=\infty$, then the inequality says that $f$ is log-Lipschitz, which follows from the fact that Zygmund-class functions are a fortiori log-Lipschitz together with standard the embeddings
%\begin{align}
%    \dot{W}^{\frac{\d}{p}+1,p} \subset BMO \subset \dot{\mathcal{C}}^1.
%\end{align}
%If $p=1$, then it says that $f$ is Lipschitz. Of course, neither of these two cases is allowed above.
In applying \cref{thm:FI'} with $v_\ep$ as part of the proof, we need to quantify the divergence of $\|\nab v_\ep\|_{L^\infty}$ as $\ep\rightarrow 0^+$. In \cite{rosenzweig_mean-field_2022-1, rosenzweig_mean-field_2022}, the fact that $\dot{\mathcal{C}}^1$ functions are log-Lipschitz was used to bound $\|\nab v_\ep\|_{L^\infty} \lesssim \|v\|_{\dot{\mathcal{C}}^1}{(1+|\log \ep|)}$. Note that $\|v\|_{\dot{\mathcal{C}}^1}\lesssim \|v\|_{\dot{W}^{\frac{\d}{p}+1,p}}$ by Sobolev embedding. The modulus of continuity bound \eqref{eq:BWll} allows us to improve this bound to $\|\nab v_\ep\|_{L^\infty} \lesssim \|v\|_{\dot{W}^{\frac{\d}{p}+1,p}}(1+|\log\ep|)^{1-\frac1p}$ (see \cref{lem:mollerrs}). Combining with better estimates for the mollification error $v-v_\ep$ (see \cref{lem:molvepLL}), this improvement then yields \cref{thm:FI}. We refer to \cref{ssec:BWdc} for the complete proof.

To obtain the result of mean-field convergence \cref{thm:MF} from \cref{thm:FI}, we need to choose $\ep=\ep_t$ appropriately in a possibly time-dependent fashion. $\ep_t$ needs to be sufficiently small depending on $N,\Fr_N(\XN^t,\mu^t)$ to compensate for the divergent prefactors of $N$ and the powers of $\Fr_N(\XN^t,\mu^t)$ with exponent ${>}1$ in the estimates of \cref{thm:FI}. Thanks to the improvement $(1+|\log\ep|)^{1-\frac1p}$, this is possible to do and is still suitable for closing the differential inequality for $\Fr_N(\XN^t,\mu^t)$. We refer to \cref{ssec:BWmp} for the details. At the risk of being repetitive, we emphasize that if $1-\frac1p$ were replaced by $1$, then we would be stuck with the same short-time restriction as in \cite{rosenzweig_mean-field_2022}.

\begin{remark}
The reader should note that \Cref{thm:FI,thm:MF} do not cover the case $p=\infty$ (they also do not cover $p=1$, but this is less relevant). We return to this exception in \cref{ssec:RQmf}.
\end{remark}

%The relevance of the modulus of continuity estimate \eqref{eq:BWll} is that if the mean-field density $\mu \in \dot{W}^{\frac{\d}{p}+\s+2-\d,p}$, which is the aforementioned critical space for equation \eqref{eq:MFlim}, then the mean-field vector field $u\coloneqq \M\nabla\g\ast\mu \in \dot{W}^{\frac{\d}{p}+1,p}$. This is sharp, in the sense that one can produce examples of $\mu$ such that $u$ is not in...

\section{Counterexamples}\label{sec:CEX}
This section is devoted to the proofs of our counterexample results, \Cref{thm:CEF1,thm:CEF2}. We have already explained in \cref{ssec:MRreg} how  \cref{cor:CEF2} follows from \cref{thm:CEF2}, so we do not discuss this result further.
	
\subsection{Proof of \cref{thm:CEF1}\ref{item:CEF11}}\label{ssec:CEX11}
	
	%We first prove Proposition~\ref{prop:CEF1}. In particular, we construct a velocity field $v$ with $\nabla v\in \BMO$ and a mild logarithmic singularity at the origin, and test against a family of rescaled, compactly supported, zero-mean functions $f_r$. The commutator grows like $\log\log(1/r)$ as $r\to0$, while the denominator in \eqref{eq:CEF1} remains uniformly bounded, yielding an arbitrarily large ratio.

	Let $\chi\in C^\infty_c(\R^\d)$ such that $0\leq \chi\leq 1$, $\supp\chi\subset B(0,1/2)$, and $\chi\equiv 1$ on $B(0,1/4)$. Set $v\coloneqq (v_1,\ldots,v_\d)$, where 
		\begin{equation}
			v_1(x) \coloneqq -x_1\log(\log(1/|x|))\chi(x) \qquad \text{and} \qquad v_i(x) \coloneqq 0, \ 2\le i\le \d.
		\end{equation}
	Evidently, $v$ is compactly supported. It can also be made $C^\infty$ by convolution with $\chi_\ep$, for $\ep=\ep(r)>0$, where the parameter $r$ is as below. We omit the implementation of this step.
        
    We claim that $v'\in \BMO(\R)$ when $\d=1$ and $\nabla^{\otimes \d} v\in L^{\frac{\d}{\d-1}}(\R^\d)$ when $\d\geq 2$. Indeed, when $\d=1$
		\begin{equation}
			\frac{d}{dx}\Big(-x\log\big(\log(1/|x|)\big)\Big)=-\log\big(\log(1/|x|)\big)+\frac{1}{\log(1/|x|)}\in \BMO(\R),
		\end{equation}
	and thus so is $v$ since multiplication by $\chi$ preserves $\BMO(\R)$. On the other hand, when $\d\geq 2$ for all $k\geq 1$ there exists $C_{\d,k}>0$  such that
		\begin{equation}
			\forall |x|\le \frac12, \qquad \big |\nabla^{\otimes k} \log\big(\log(1/|x|)\big) \big|\leq \frac{C_{\d,k}}{|x|^k \log(1/|x|)}.
		\end{equation}
		With the Leibniz rule and the fact that $\supp\chi'\subset B(0,\frac14)^c$, the preceding estimate implies that
		\begin{equation}
		\forall |x| \le \frac12, \qquad	|\nabla^{\otimes \d} v(x)|\leq \frac{C_{\d,k}'}{|x|^{\d-1} \log(1/|x|)}.
		\end{equation}
		Since $\chi$ is supported on $B(0,\frac{1}{2})$, it follows that
		\begin{equation}
			\int_{\R^\d} |\nabla^{\otimes \d}v|^\frac{\d}{\d-1}\leq C\int_{B(0,1/2)} \frac{1}{|x|^{\d-1} \log(1/|x|)}= C\int_0^{1/2} \frac{1}{r\log(1/r)^{\frac{\d}{\d-1}}}<\infty,
		\end{equation}
		where in the first equality, we have used a radial coordinate change. By Sobolev embedding, this also implies that $\nabla v\in \BMO(\R^\d)$ and $|\nabla|^{\frac{\d-\s}{2}}v\in L^{\frac{2\d}{\d-\s-2}}(\R^\d)$ for all $-2< \s<\d-2$.
		
		Next, suppose that $f\in C^\infty_c(\R^\d)$ is a zero-mean function supported in $B(0,1/4)$ such that
		\begin{equation}
			\int_{(\R^\d)^2} \frac{|(x-y)^1|^2}{|x-y|^{\s+2}} f(x)f(y)>0.
		\end{equation}
		Such an $f$ exists since if $f$ is symmetric under permutation of coordinates, then
		\begin{align}
		\int_{(\R^\d)^2} \frac{|(x-y)^1|^2}{|x-y|^{\s+2}} f(x)f(y)
		&= \cd\int_{\R^\d}|\xi|^{\s-\d-2}\Big(|\xi|^2-(\d-\s)|\xi^1|^2\Big) |\hat f(\xi)|^2 \nn\\
        &= \cd \s\int_{\R^\d}|\xi|^{\s-\d-2}|\xi^1|^2|\hat{f}(\xi)|^2.
		\end{align}
		Let $f_r(x) \coloneqq r^{-\d}f(r^{-1}x)$ for $r>0$, so that
		\begin{multline}\label{eq:CES}
			\int_{(\R^\d)^2} (v(x)-v(y))\cdot\nabla\g(x-y) f_r(x)f_r(y)
			\\=\int _{(\R^\d)^2}\log\big(\log(1/|x|)\big) \frac{|(x-y)^1|^2}{|x-y|^{\s+2}} f_r(x)f_r(y) +\int_{(\R^\d)^2}  \log(\frac{\log(1/|x|)}{\log(1/|y|)}) \frac{y^1(x-y)^1}{|x-y|^{\s+2}} f_r(x) f_r(y).
		\end{multline}
		Changing coordinates, the first term on the right-hand side of~\eqref{eq:CES} equals
		\begin{equation}\label{eq:CEST1}
			r^{-\s}\log(\log(1/r))\int_{(\R^\d)^2} \frac{\log\big(\log(1/|x|)-\log(r)\big)}{\log\big(\log(1/r)\big)} \frac{|(x-y)^1|^2}{|x-y|^{\s+2}} f(x)f(y),
		\end{equation}
		while the second term equals
		\begin{equation}\label{eq:CEST2}
			r^{-\s}\int_{(\R^\d)^2}  \log(\frac{\log(1/|x|)-\log(r)}{\log(1/|y|)-\log(r)}) \frac{y^1(x-y)^1}{|x-y|^{\s+2}} f(x) f(y).
		\end{equation}
		By the dominated convergence theorem,
		\begin{equation}\label{eq:CEST1Lim}
			\lim_{r\rightarrow 0}\int_{(\R^\d)^2} \frac{\log\big(\log(1/|x|)-\log(r)\big)}{\log\big(\log(1/r)\big)} \frac{|(x-y)^1|^2}{|x-y|^{\s+2}} f(x)f(y)=\int_{(\R^\d)^2} \frac{|(x-y)^1|^2}{|x-y|^{\s+2}} f(x)f(y) 
		\end{equation}
		and
		\begin{equation}\label{eq:CEST2Lim}
			\lim_{r\rightarrow 0}\int_{(\R^\d)^2}  \log(\frac{\log(1/|x|)-\log(r)}{\log(1/|y|)-\log(r)}) \frac{y^1(x-y)^1}{|x-y|^{\s+2}} f(x) f(y)=0.
		\end{equation}
		Combining~\eqref{eq:CES}-\eqref{eq:CEST2Lim}, we see that there exits $C>0$ depending only on $\d,\s,f$ such that 
		\begin{equation}
        \int_{(\R^\d)^2} (v(x)-v(y))\cdot\nabla\g(x-y) f_r(x)f_r(y) \geq C^{-1}r^{-\s}\log(\log(1/r))-Cr^{-\s}
		\end{equation}
		for all sufficiently small $r$.
		
		Since $\|f_r\|_{\dot{H}^{\frac{\s-\d}{2}}}^2=r^{-\s}\|f\|_{\dot{H}^{\frac{\s-\d}{2}}}^2$, in total we thus have found that for all sufficiently small $r$ 
		\begin{equation}
        \frac{\displaystyle\bigg|\int_{(\R^\d)^2} (v(x)-v(y))\cdot\nabla\g(x-y) f_r(x)f_r(y)\bigg|}{\displaystyle\Big(\|\nabla v\|_{\BMO}+\indic_{\s<\d-2}\||\nabla|^{\frac{\d-\s}{2}}v\|_{L^{\frac{2\d}{\d-\s-2}}}\Big)\|f_r\|_{\dot{H}^{\frac{\s-\d}{2}}}^2}\geq C^{-1}\log(\log(1/r))-C.
		\end{equation}
		Taking $r\rightarrow 0$ concludes the claim.

	\begin{remark}
		We note that if $\d=1$ and $\s=0$, the above counter example fails since
		\begin{equation}
			 \int_{(\R^\d)^2} \frac{|(x-y)^1|^2}{|x-y|^2}f(x)f(y) =\bigg(\int_{\R^\d} f\bigg)^2=0
			 \end{equation}
		for all zero-mean functions $f$.
	\end{remark}
	
	\subsection{Proof of \cref{thm:CEF1}\ref{item:CEF12}}\label{ssec:CEX12}
	%Next, we prove the commutator estimate Theorem~\ref{thm:BMOCE}. This essentially involves appropriately integrating by parts and grouping terms into components that can be individually bounded via classical commutator estimates. Here we also use an equivalent formulation of~\eqref{eq:CE} for Schwartz functions with Fourier transforms supported away from $0$.

	We use an equivalent formulation of~\eqref{eq:CE} for Schwartz functions with Fourier transforms supported away from $0$. If $F$ and $G$ are Schwartz with Fourier transforms supported away from the origin, then $f\coloneqq\Dm F$ and $g\coloneqq \Dm G$ are both Schwartz with Fourier transforms supported away from the origin. Additionally, since $\Dm\g = \cd\delta_0$, it holds that
		\begin{equation}\label{eq:CECor}
			\int_{\R^2} (v(x)-v(y))\cdot\nabla\g(x-y) f(x) g(y)=\mathsf{c}\int_{\R}v\cdot\Big(F'\Dm G +G'\Dm F\Big).
		\end{equation}
		Since $\| F\|_{\dot{H}^{\frac{1}{2}}}=\cd \|f\|_{\dot{H}^{-\frac{1}{2}}}$ and similarly for $G$ and $g$, it suffices to show that there exists $C>0$ such that
		\begin{equation}
			\bigg|\int_{\R} v\big( F'|\nabla|G+G'|\nabla|F\big)\bigg|\leq C\|v'\|_{\BMO}\|F\|_{\dot{H}^{1/2}}\|G\|_{\dot{H}^{1/2}}
		\end{equation}
		for all $F,G\in\mathcal{S}(\R)$ with Fourier support away from the origin.
		
		To this end, we integrate by parts 
		\begin{align}
			\int_{\R} v\big( F'|\nabla|G+G'|\nabla|F\big)&= -\int_{\R} \Big(v'\big(F\Dm G+G\Dm F\big)+ v F\p_x\Dm G +v G\p_x\Dm F\Big) \nn\\
			\notag&=-\int_{\R}\Big(v'\big(F\Dm G+G\Dm F\big) -\p_x\Dm^{\frac{1}{2}}(vG)\Dm^{\frac{1}{2}}F -\p_x\Dm^{\frac{1}{2}}(vF)\Dm^{\frac{1}{2}}G \Big)\nn\\
			&=-\int_{\R} \Big(v'F\Dm G-\Dm^{\frac{1}{2}}v'F\Dm^{\frac{1}{2}}G\Big)-\int_{\R}\Big(v'G\Dm F-\Dm^{\frac{1}{2}}v'G\Dm^{\frac{1}{2}}F\Big)
			\nn\\&\nn\ph+\int_{(\R)}\Big( v\big(\Dm^{\frac{1}{2}}F'\Dm^{\frac{1}{2}}G+\Dm^{\frac{1}{2}}G'\Dm^{\frac{1}{2}}F\big)+3 v'\Dm^{\frac{1}{2}}F \Dm^{\frac{1}{2}}G\Big)
			\\&\ph+ \int_{\R}\Big(T(v, F) \Dm^{\frac{1}{2}}G+ T(v, G) \Dm^{\frac{1}{2}}F\Big),\label{eq:Rieszsup0}
		\end{align}
		where $T$ is the bilinear operator
		\begin{equation}
			T(v, H) \coloneqq \p_x\Dm^{\frac{1}{2}}(v H) - \Dm^{\frac{1}{2}}v' H- v\Dm^{\frac{1}{2}}H' -\frac32 v'\Dm^{\frac{1}{2}}H
		\end{equation}
		defined for $H\in\mathcal{S}(\R)$.
		
		First, we note that after integrating by parts
		\begin{equation}
			\int_{\R} v\big(\Dm^{\frac{1}{2}}F'\Dm^{\frac{1}{2}}G+\Dm^{\frac{1}{2}}G'\Dm^{\frac{1}{2}}F\big)=-\int_{\R} v'\Dm^{\frac{1}{2}}F\Dm^{\frac{1}{2}}G.
		\end{equation}
		We can thus split~\eqref{eq:Rieszsup0} into a sum of three terms defined by
		\begin{align}
			&\text{Term}_1\coloneqq -\int_{\R} \Big(v'F\Dm G-\Dm^{\frac{1}{2}}v'F\Dm^{\frac{1}{2}}G-v'\Dm^{\frac{1}{2}}F \Dm^{\frac{1}{2}}G\Big),
			\\&\text{Term}_2 \coloneqq -\int_{\R} \Big(v'G\Dm F-\Dm^{\frac{1}{2}}v'G\Dm^{\frac{1}{2}}F-v'\Dm^{\frac{1}{2}}F \Dm^{\frac{1}{2}}G\Big),
			\\&\text{Term}_3 \coloneqq \int_{\R}\Big(T(v,F)\Dm^{\frac{1}{2}}G+ T(v,G)\Dm^{\frac{1}{2}}F\Big).
		\end{align}
		We bound each of these terms separately.
		
		We begin with $\text{Term}_1$. Integrating by parts, we have that
		\begin{equation}
			\text{Term}_1= -\int_{\R}\Big(\Dm^{\frac{1}{2}}(v' F) - \Dm^{\frac{1}{2}}v' F - v'\Dm^{\frac{1}{2}}F\Big)\Dm^{\frac{1}{2}}G.
		\end{equation}
		The expression inside the brackets is a Kenig-Ponce-Vega type commutator \cite{kenig_well-posedness_1993}, and by \cite[Remark 1.3, (1.7)]{li_kato-ponce_2019} with $s = \frac{1}{2}$, $f = F$, $g =v'$, we have that
		\begin{equation}
			\big\|\Dm^{\frac{1}{2}}(v'F) - \Dm^{\frac{1}{2}}v' F - v'\Dm^{\frac{1}{2}}F\big\|_{L^2} \le C\|v'\|_{\BMO}  \|F\|_{\dot{H}^{1/2}} \
		\end{equation}
		for some $C>0$. So by Cauchy-Schwarz,
		\begin{equation}\label{eq:Term1}
			|\text{Term}_1|\le C\|v'\|_{\BMO}\|F\|_{\dot{H}^{1/2}}\|G\|_{\dot{H}^{1/2}}.
		\end{equation}
		Swapping $F$ and $G$, we can bound $\text{Term}_2$ identically.
		
		For $\text{Term}_3$, we note that $T$ is a higher-order commutator and by \cite[Corollary 1.4(2)]{li_kato-ponce_2019} applied with $s= \frac32$, $f=H$, $g=v$, $A^s= \p_x\Dm^{\frac{1}{2}}$, $p=2$, $s_2=1$, $s_1 = \frac{1}{2}$, we obtain
		\begin{equation}
			\|T(v, H)\|_{L^2} \le C \|\Dm v\|_{\BMO} \|H\|_{\dot{H}^{1/2}},
		\end{equation}
		for some $C>0$. Note that since the Hilbert transform is bounded on BMO, one may bound $\|\Dm v\|_{\BMO} \le C\|v'\|_{\BMO}$. By Cauchy-Schwarz, it then follows that
		\begin{equation}\label{eq:Term3}
			|\text{Term}_3|\le C \|v'\|_{\BMO} \|F\|_{\dot{H}^{1/2}}\|G\|_{\dot{H}^{1/2}}.
		\end{equation}
        
		Combining~\eqref{eq:Term1} for $\Te_1,\Te_2$ and and~\eqref{eq:Term3} for $\Te_3$, the proof is complete.	
	
	\subsection{Proof of \cref{thm:CEF2}}\label{ssec:CEX2}
    %We now turn to the proof of Proposition~\ref{prop:GCEF}. We build $v,f,g$ so that their Fourier supports sit on a few carefully chosen (and well separated) frequency shells. This forces one contribution to the commutator to be “off-resonant’’ and therefore small while the other contribution is “resonant’’ and stays bounded below. At the same time, the same frequency localization makes the velocity norm essentially depend on $\sup_{s\in[k-1,k+1]}\sigma(s)$, while $\|f\|_{\g,t}$ scales like $\hat\g(tk)^{1/2}$; after normalizing by $\|f\|_{\g,t}$ this produces a factor $\hat\g(tk)^{-1/2}$ in the final ratio. Letting $k\to\infty$ then produces the desired ratio growth.

	Let $\phi$ be a radially symmetric Schwartz function on $\R^\d$ such that the Fourier transform $\hat\phi$ satisifes $0\leq \hat\phi\leq 1$, $\supp \hat\phi\subset B(0,1/4)$, and $\hat\phi\equiv 1$ on $B(0,1/8)$. Letting $e_1\coloneqq (1,0,\dotsc,0)$ and $k>2$, define the (real) vector field $v$ by
		\begin{equation}
			\hat{v}(\xi)\coloneqq  \imath e_1\Big( \hat\phi(\xi+k e_1)- \hat{\phi}(\xi-k e_1)\Big).
		\end{equation}
		and (real) functions $f,g$ by
		\begin{equation}
			\hat{f}(\xi)\coloneqq \Big(\hat\phi\big(\xi+(k+1)e_1)\big)+\hat\phi\big(\xi-(k+1)e_1\big)\Big),
		\end{equation}
		\begin{equation}
			\hat{g}(\xi)\coloneqq\Big(\hat\phi(\xi+e_1)+\hat\phi(\xi-e_1)\Big).
		\end{equation}
		Clearly, $f$ and $g$ are both Schwartz and have Fourier transforms supported away from $0$. We then note that
		\begin{equation}
			\int_{(\R^\d)^2} (v(x)-v(y))\cdot\nabla\g(x-y)f(x) g(y)=\int_{\R^\d} v\cdot\nabla\g*f g+\int_{\R^\d} v\cdot \nabla\g*g f.
		\end{equation}
		By Plancherel's theorem,
		\begin{equation}
			\int_{\R^\d} v\cdot\nabla\g*f g=\int_{(\R^\d)^2} \hat{v}(\xi_1)\cdot 2\pi i (\xi_1-\xi_2)\hat\g(\xi_1-\xi_2)\hat{f}(\xi_1-\xi_2)\hat{g}(\xi_2).
		\end{equation}
		Unpacking the definitions of $v,f$ and $g$, the preceding right-hand side is equal to
		\begin{multline}
			-\sum_{\ep_i\in\{\pm 1\}} \ep_1\int_{(\R^\d)^2} \hat\phi(\xi_1+\ep_1 ke_1) e_1\cdot (\xi_1-\xi_2)\hat{\g}(\xi_1-\xi_2)\hat{\phi}\big(\xi_1-\xi_2+\ep_2(k+1)e_1\big)\hat{\phi}(\xi_2+\ep_3 e_1)
			\\=-\sum_{\ep_i\in\{\pm 1\}} \ep_1\int_{(\R^\d)^2}\hat\phi(\zeta_1)\hat\phi(\zeta_2)\hat\phi\big(\zeta_1-\zeta_2+((\ep_3+\ep_2) +(\ep_2-\ep_1)k)e_1\big)
			\\\times e_1\cdot\big(\zeta_1-\zeta_2+(\ep_3-\ep_1k)e_1\big)\hat\g\big(\zeta_1-\zeta_2+(\ep_3-\ep_1k)e_1\big),
		\end{multline}
		where the final equality follows from a change of variables. Since $\supp \hat{\phi}\subset B(0,1/4)$, the integrals in the sum above are only nonzero if 
		\begin{equation}
			|\zeta_1-\zeta_2+((\ep_3+\ep_2) +(\ep_2-\ep_1)k)e_1)|\leq \frac{1}{4}
		\end{equation}
		for some $\zeta_i$ in $B(0,1/4)$. Since $k>2$, we must have $(\ep_1,\ep_2,\ep_3)=\pm(1,1,-1)$. Using symmetry, we have shown that
		\begin{equation}
			\int_{\R^\d} v\cdot\nabla \g*f g=2\int_{(\R^\d)^2} \hat\phi(\zeta_1)\hat\phi(\zeta_2)\hat\phi(\zeta_1-\zeta_2) \big((\zeta_2-\zeta_1)\cdot e_1 +1+k\big)\hat\g\big(\zeta_2-\zeta_1+(1+k)e_1\big).
		\end{equation}
		Since $|\zeta_i|\leq \frac{1}{4}$ on the support of $\hat\phi(\zeta_i)$ and $\hat\g$ is decreasing by \cref{ass:pot}, we may crudely bound the right-hand side above using Young's convolution inequality to obtain
		\begin{equation}\label{eq:CB1}
			\bigg|\int_{\R^\d} v\cdot\nabla\g*f g\bigg|\leq Ck\hat\g(k) \|\hat\phi\|_{L^{3/2}}^3,
		\end{equation}
		where $C>0$ is a constant only depending on $\d$.
		
		By the same reasoning as above,
		\begin{multline}
			\int_{\R^\d} v\cdot \nabla \g*g f =-\sum_{\ep_i\in\{\pm 1\}} \ep_1 \int_{(\R^\d)^2} \hat{\phi}(\zeta_1)\hat\phi(\zeta_2)\hat\phi\Big(\zeta_1-\zeta_2+\big((\ep_2+\ep_3)+(\ep_3-\ep_1)k\big)e_1\Big)
			\\\times e_1\cdot\Big(\zeta_1-\zeta_2+\big(\ep_3+(\ep_3-\ep_1)k\big)e_1\Big)\hat\g\Big(\zeta_1-\zeta_2+\big(\ep_3+(\ep_3-\ep_1)k\big)e_1\Big).
		\end{multline}
		For an integral in the sum on the preceding right-hand side to be nonzero, it must now be the case that $(\ep_1,\ep_2,\ep_3)=\pm(1,-1,1)$. Using symmetry, we have thus shown that
		\begin{equation}
			\int_{\R^\d} v\cdot \nabla \g*g f =2\int_{(\R^\d)^2}\hat\phi(\zeta_1)\hat\phi(\zeta_2)\hat\phi(\zeta_1-\zeta_2)\big(1-e_1\cdot(\zeta_1-\zeta_2)\big)\hat\g(\zeta_2-\zeta_1+e_1).
		\end{equation}
		Since $(1-e_1\cdot(\zeta_1-\zeta_2))\geq \frac{1}{2}$ and $|\zeta_2-\zeta_1+e_1|\leq 2$ on the support of $\hat\phi(\zeta_1)\hat\phi(\zeta_2)$, we have found the lower bound
		\begin{equation}\label{eq:CB2}
			\int_{\R^\d} v\cdot \nabla \g*g f\geq \hat\g(2)\int_{(\R^\d)^2} \hat\phi(\zeta_1)\hat\phi(\zeta_2)\hat\phi(\zeta_1-\zeta_2)\geq C^{-1}\hat\g(2),
		\end{equation}
		where the last inequality uses that $\hat\phi=1$ on $B(0,1/8)$ and $C>0$ only depends on $\phi$. 
		
		On the other hand, it is immediate from Plancherel's theorem and Fourier support considerations that
		\begin{align}
			\|g\|_{\g_t} \lesssim \|\phi\|_{L^2}, \label{eq:GB}\\
            \| f\|_{\g_t}\lesssim  \hat\g(tk)^{1/2}\|\phi\|_{L^2} \label{eq:FB}
		\end{align}
		for implicit constants depending on $\d,\g,t$, while Fourier inversion shows that
		\begin{equation}\label{eq:VB1}
			\|\nabla v\|_{L^\infty}\lesssim k\|\phi\|_{L^1}.
		\end{equation}
		On the other hand, since $\supp\widehat v\subset \{k-1 \le |\xi|\le k+1\}$, we have
		$\sigma(|\xi|)\le \sup_{r\in[k-1,k+1]}\sigma(r)$ on $\supp\widehat v$. {If $2\le p\le \infty$, then by the Hausdorff-Young inequality and the Fourier support of $v$, 
		\begin{equation}\label{eq:VB2}
			\|\sigma(|\nabla|)v\|_{L^p}
			\lesssim
			\Big(\sup_{r\in[k-1,k+1]}\sigma(r)\Big)\|\hat\phi\|_{L^\frac{p-1}{p}},
		\end{equation}
		for any $p\in[1,\infty]$, with constants depending only on $\d,p$. If $1\le p<2$, then under the assumption that $\p_x^n\sigma$ is essentially bounded away from the origin for $n\le \d+1$, it follows from Young's inequality that
        \begin{align}\label{eq:VB2'}
            \|\sigma(|\nabla|)v\|_{L^p}
			\lesssim \max_{n\le \d+1}\sup_{r\in [k-1,k+1]} |\p_r^n\sigma(r)| \|\phi\|_{L^p}.
        \end{align}}
		%Indeed, for $1\le p\le2$ this follows from Hausdorff--Young and for $2\le p\le\infty$ from Bernstein/Nikolskii and Plancherel.

		Combining~\eqref{eq:CB1} and~\eqref{eq:CB2}-\eqref{eq:VB2'}, in total have found that
		\begin{multline}
			\frac{\displaystyle\bigg|\int_{(\R^\d)^2} (v(x)-v(y))\cdot\nabla\g(x-y)f(x)g(y)\bigg|}{\displaystyle\Big(\|\nabla v\|_{L^\infty} + \| \sigma(|\nabla|)v\|_{L^p}\Big)\|f\|_{\g_t}\|g\|_{\g_t}}\\\geq \frac{C^{-1}-Ck\hat\g(k)}{C\Big(k+\indic_{1\le p<2}\max_{n\le \d+1}\|\p_r^n\sigma\|_{L^\infty([k-1,k+1])} + \indic_{2\le p\le \infty}\|\sigma\|_{L^\infty([k-1,k+1])}\Big)\hat\g(tk)^{1/2}},
		\end{multline}
		for some $C=C(\d,\s,\phi)>0$. By~\eqref{eq:decay}, we have that $k\hat{g}(k) $ and $k\hat{\g}(tk)^{1/2}$ vanish as $k\rightarrow\infty$, while by hypothesis,
		\begin{align}
		    \lim_{k\rightarrow\infty}\Big(\indic_{1\le p<2}\max_{n\le \d+1}\|\p_r^n\sigma\|_{L^\infty([k-1,k+1])} + \indic_{2\le p\le \infty}\|\sigma\|_{L^\infty([k-1,k+1])}\Big)\hat{\g}(tk)^{1/2}=0.
		\end{align}
		Taking $k\rightarrow \infty$ completes the proof.

\section{Defective commutator estimates and mean-field convergence}\label{sec:BW}
This section is devoted to the results \Cref{thm:FI,thm:MF} advertised in \cref{ssec:MRmf} about defective commutator estimates and scaling-critical modulated energy estimates, respectively. The remainder of this section is organized as follows. In \cref{ssec:BWmoll}, we show some preliminary rates of convergence for mollifiers. In \cref{ssec:BWme}, we recall some properties of the modulated energy that are useful in the sequel. In \cref{ssec:BWdc}, we prove our defective commutator estimate \cref{thm:FI}, which is the primary technical result. Finally, in \cref{ssec:BWmp}, we prove \cref{thm:MF}.

Some comments on notation. Given $\mu\in L^p$, we use throughout this section the notation $\la,\ka$ from \eqref{eq:ladef}, \eqref{eq:kadef}.
%\begin{align}\label{eq:ladefLp}
%     \la\coloneqq (N\|\mu\|_{L^p})^{-\frac{p}{\d(p-1)}},
%\end{align}
%which corresponds to the microscale $\propto N^{-1/\d}$ when $p=\infty$.
When $\mu=\mu^t$, we write $\la_t,\ka_t$ to indicate the time dependence.

\subsection{Mollification}\label{ssec:BWmoll}
Let $\chi\in C_c^\infty(\R^\d)$ be an even bump function which is $1$ on $B(0,\frac14)$, zero outside $B(0,1)$, and satisfies $1\geq \chi\geq 0$, and $\int_{\R^\d} \chi dx = 1$. Let $v:\R^\d\rightarrow\R^\d$ be a vector field, and set
\begin{equation}
v_{\ep}(x) \coloneqq (v\ast\chi_\ep)(x), \qquad \chi_\ep(x) \coloneqq \ep^{-\d}\chi(x/\ep),
\end{equation}
where the convolution is performed componentwise. It thus holds that $v_\vep$ is $C^\infty$ and
\begin{equation}
\forall  k\ge 0, \qquad \|\nabla^{\otimes k} v_\ep\|_{L^\infty} \lesssim_k \ep^{-k}\| v\|_{L^\infty}.
\end{equation}
Additionally, if $\|\cdot\|_{X}$ is any translation-invariant semi-norm, then we also have $\|v_\ep\|_{X} \leq \|v\|_{X}$.

The purpose of this subsection is to establish estimates for the mollification error $v-v_\ep$ as well as to estimate how the Lipschitz seminorm of $v_\ep$ diverges as $\ep\rightarrow 0$ under certain regularity assumptions on $v$. The main result is the following lemma. %We work under the assumption that $v$ either belongs to the log Lipschitz or H\"older-Zygmund scale of spaces. 

\begin{lemma}\label{lem:mollerrs}
For any integer $k\ge 0$ and reals $0<b\le a\le 1$, there exists a constant $C=C(a)>0$ such that
\begin{align}
    \|\nab^{\otimes k}v-\nab^{\otimes k}v_\ep\|_{L^\infty} &\le C\|\nab^{\otimes k}v\|_{\dot{\mathcal{C}}^a}\ep^a,\label{eq:mollerrs1}\\
    \|\nab^{\otimes k}v-\nab^{\otimes k}v_\ep\|_{\dot{\mathcal{C}}^b} &\le C\|\nab^{\otimes k}v\|_{\dot{\mathcal{C}}^a}\ep^{a-b}.\label{eq:mollerrs1'}
\end{align}
Suppose that there exists $\delta_k>0$ and an increasing function $\omega_k:[0,\delta_k]\rightarrow [0,\infty)$ such that $|\nab^{\otimes k}v(x)-\nab^{\otimes k}v(y)|\le \omega_k(|x-y|)$. Then there is a constant $C>0$ such that
\begin{align}
    \|\nab^{\otimes k}v-\nab^{\otimes k}v_\ep\|_{L^\infty} \le C\omega_k(2\ep), \qquad {0\le \ep \le \frac{\delta_k}{2},}\label{eq:mollerrs2} \\
    \|\nab^{\otimes k}v_\ep\|_{L^\infty} \le C\ep^{-1}\omega_{k-1}(2\ep), \qquad {0<\ep\le \frac{\delta_{k-1}}{2}.}\label{eq:mollerrs3}
\end{align}
\end{lemma}
\begin{proof}
Observe that since $\chi$ is even and $\int_{\R^\d}\chi_\ep=1$,
\begin{align}
v(x)-v_\ep(x) = \int_{\R^\d}\pa*{v(x)-v(x-y)}\chi_\ep(y)dy &= \int_{\R^\d}\pa*{v(x)-v(x+y)}\chi_\ep(-y)dy \nn\\
&=  \int_{\R^\d}\pa*{v(x)-v(x+y)}\chi_\ep(y)dy. \label{eq:vvepdiff}
\end{align}
Hence,
\begin{equation}
v(x)-v_\ep(x) = \frac{1}{2}\int_{\R^\d}\pa*{2v(x)-v(x-y)-v(x+y)}\chi_\ep(y)dy,
\end{equation}
which implies that
\begin{align}
    \|v-v_\ep\|_{L^\infty} &\le \frac{1}{2}\int_{\R^\d}\left\|\frac{(2v(\cdot)-v(\cdot-y)-v(\cdot+y))}{|y|^{a}}\right\|_{L^\infty} |y|^{a}\chi_\ep(y)dy \nn\\
    &\le \frac{\|v\|_{\dot{\mathcal{C}}^a}}{2}\int_{\R^\d} |y|^{a}\chi_\ep(y)dy \nn\\
    &=\frac{\ep^{a}\|v\|_{\dot{\mathcal{C}}^a}}{2}\int_{\R^\d}|y|^a\chi(y)dy,
\end{align}
where the penultimate inequality follows the change of variable $y/\ep \mapsto y$. Replacing $v$ by $\nab^{\otimes k}v$ in the preceding argument, for any $k\ge 0$, then yields the desired conclusion \eqref{eq:mollerrs1}.

Let $0<b\le a\le 1$. %If $b=a$, then \eqref{eq:mollerrs1'} is trivially true, so suppose that $0<b<a$.
Using the interpolation estimate $\|\cdot\|_{\dot{\mathcal{C}}^\theta}\le \|\cdot\|_{\dot{\mathcal{C}}^{\theta_1}}^t \|\cdot\|_{\dot{\mathcal{C}}^{\theta_2}}^{1-t}$ for $\theta_1\le \theta\le \theta_2$ and $t = \frac{\theta_2-\theta}{\theta_2-\theta_1}$, we have
\begin{align}
    \|v-v_\ep\|_{\dot{\mathcal{C}}^b} \le \|v-v_\ep\|_{L^\infty}^{\frac{a-b}{a}} \|v-v_\ep\|_{\dot{\mathcal{C}}^a}^{\frac{b}{a}} \le (C\|v\|_{\dot{\mathcal{C}}^a}\ep^a)^{\frac{a-b}{a}}(2\|v\|_{\dot{\mathcal{C}}^a})^{\frac{b}{a}} = C'\|v\|_{\dot{\mathcal{C}}^a}\ep^{a-b},
\end{align}
where we have used \eqref{eq:mollerrs1} to obtain the second inequality. Replacing $v$ by $\nab^{\otimes k}v$ then yields \eqref{eq:mollerrs1'}.

Now suppose that $|v(x)-v(y)| \le \omega(|x-y|)$ for all $|x-y|\le \delta$ for an increasing $\omega: [0,\delta]\rightarrow [0,\infty)$. From the identity \eqref{eq:vvepdiff}, it follows that if $\ep \le \delta/2$,
\begin{align}
    \|v-v_\ep\|_{L^\infty} \le \int_{\R^\d}\omega(y)\chi_\ep(y)dy = \int_{B(0,2)}\omega(\ep y)\chi(y)dy \le \omega(2\ep)\int_{B(0,2)}\chi(y)dy,
\end{align}
where we have again made the change of variable $y/\ep \mapsto y$ and used the monotonicity of $\omega$. Replacing $v$ by $\nab^{\otimes k}v$ then yields \eqref{eq:mollerrs2}.

Finally, using that $\nabla\chi_\ep$ is odd by assumption that $\chi_\ep$ is even, and therefore $\int_{\R^\d}\nabla\chi_\ep = 0$, we have
\begin{align}
\nabla^{\otimes k} v_\ep(x) = \int_{\R^\d}\pa*{\nab^{\otimes k-1}v(x-y)-\nab^{\otimes k-1}v(x)}\otimes\nabla\chi_\ep(y)dy.
\end{align}
Suppose that $|\nab^{\otimes k-1}v(x)-\nab^{\otimes k-1}v(y)| \le \omega_{k-1}(|x-y|)$ for $|x-y|\le \delta_{k-1}$, where $\omega_{k-1}:[0,\delta_{k-1}]\rightarrow[0,\infty)$ is increasing. Making the change of variable $y/\ep \mapsto y$,
\begin{align}
\|\nabla^{\otimes k} v_\ep\|_{L^\infty} &\leq \frac{1}{\ep}\int_{\R^\d}|\nab^{\otimes k-1}v(x-\ep y)-\nab^{\otimes k-1}v(x)| |\nabla\chi(y)|dy \nn\\
&\le\frac{1}{\ep}\int_{\R^\d}\omega_{k-1}(\ep y) |\nabla\chi(y)|dy \nn\\
&\leq \frac{\omega_{k-1}(2\ep)}{\ep}\int_{\R^\d} |\nabla\chi(y)|dy,
\end{align}
where we use the monotonicity of $\omega_{k-1}$ to obtain the final line. This yields \eqref{eq:mollerrs3} and completes the proof.
\end{proof}

\subsection{Properties of modulated energy}\label{ssec:BWme}
In this subsection, we recall some properties of the modulated energy \eqref{eq:modenergy} that are important for the proofs of \Cref{thm:FI,thm:MF}.

The first result is the following proposition taken from \cite[Proposition 2.15]{hess-childs_sharp_2025} in the singular case $0\le \s<\d$ and \cite[Example 4]{modeste_characterization_2024} in the nonsingular case $-2<\s<0$, which shows the coercivity of the modulated energy in the sense that it controls a squared negative-order Sobolev norm up to $O(N^{-1+\frac{\s p}{\d(p-1)}})$ error for densities $\mu\in L^p$. In particular, this yields that the modulated energy is almost nonnegative and that vanishing of the modulated energy implies weak convergence of the empirical measure to the target measure $\mu$. %Such a coercivity estimate is already known in the sub-Coulomb case \cite[Proposition 2.4]{NRS2021} (see also \cite[Corollary 4.22]{SerfatyLN} for a coercivity estimate in the Coulomb/super-Coulomb case), but its $N$-dependent additive error is much larger than that shown in the lemma below. Moreover, the estimate presented below is sharper compared to the previous works in that the control is $H^{-\frac{\d}{2}-\varepsilon}$, for any $\varepsilon>0$, which is the best one can hope for given the Dirac mass is not in $H^{-\frac{\d}{2}}$. 

\begin{prop}\label{prop:coer}
Let $0\le \s<\d$ and $p>\frac{\d}{\d-\s}$. There exists a constant $C=C(\d,\s,p)>0$ such that the following holds: for any $\mu\in L^1\cap L^p$ with $\int_{\R^\d}d\mu = 1$ and $\int_{(\R^\d)^2}|\g|(x-y)d|\mu|^{\otimes 2}(x,y)<\infty$ if $\s=0$, $\la \le 1$, and any pairwise distinct configuration $X_N\in (\R^\d)^N$, it holds that
\begin{multline}\label{eq:coer1}
\left\|\frac1N\sum_{i=1}^N \delta_{x_i}-\mu\right\|_{H^{-r/2}}^2  \le C \Big(\Fr_N(\ux_N,\mu) +  \frac{(\g(\la) +C)}{2N}\indic_{\s=0} \\
+ C \frac{\g(\la)}{2N}\indic_{\s>0} +C\|\mu\|_{L^p}\la^{\frac{\d(p-1)}{p}-\s}\Big),
\end{multline}
If $-2<\s<0$, then for any $\mu \in L^1$ with $\int_{\R^\d}d\mu = 1$ and $\int_{\R^\d}|\g|(x-y)d|\mu|^{\otimes2}<\infty$,
\begin{align}\label{eq:coer2}
    \Fr_N(\XN,\mu) = \cd\left\|\frac1N\sum_{i=1}^N\delta_{x_i}-\mu\right\|_{\dot{H}^{\frac{\s-\d}{2}}}^2.
\end{align}
\end{prop}

The second result, taken from \cite[Corollary 2.14]{hess-childs_sharp_2025}, shows that the modulated energy controls the small-scale interactions. %We will use this control in the proof of \cref{thm:FI}.

\begin{prop}\label{cor:MEcount}
Let $p>\frac{\d}{\d-\s}$ and $\mu \in L^1\cap L^p$. Define the nearest-neighbor type distance
\begin{align}\label{eq:rsi_def}
\rs_i \coloneqq \frac14\min\Big(\min_{1\le j\le N: j\ne i} |x_i-x_j|, \la\Big).
\end{align}
There exists $C= C(\d,\s,p)>0$, such that for every $\eta\le\la$,
\begin{multline}\label{eq:MEcount2}
	\Fr_N(X_N,\mu) + {\frac{(\g(\eta) + C)}{2N}\indic_{\s=0} + C\frac{\g(\eta)}{2N}\indic_{\s>0}} +C\|\mu\|_{L^p}\eta^{\frac{\d(p-1)}{p}-\s} \\
	\ge \begin{cases}\displaystyle \frac{1}{2N^2}\sum_{i=1}^N \g(4\rs_i), & {\s >0} \\ \displaystyle \frac{1}{2N^2}\sum_{i=1}^N \g(4\rs_i/\eta), & {\s=0},\end{cases}
\end{multline}
and
\begin{multline}\label{eq:MEcount1}
	\Fr_N(X_N,\mu) + {\frac{(\g(\eta) + C )}{2N}\indic_{\s=0} + C\frac{\g(\eta)}{2N}\indic_{\s>0}} +C\|\mu\|_{L^p}\eta^{\frac{\d(p-1)}{p}-\s} \\
	\ge \begin{cases}\displaystyle\frac{1}{2C N^2}\sum_{\substack{1\le i\ne j \le N \\ |x_i-x_j|\le \eta}} |x_i-x_j|^{-\s}, & {\s> 0} \\ \displaystyle\frac{1}{2CN^2}\sum_{\substack{1\le i\ne j \le N \\ |x_i-x_j|\le \eta}} -\log\left(\frac{x_i-x_j}{\eta}\right) , & {\s=0}. \end{cases}
\end{multline}
\end{prop}

The final result, which is taken from \cite{rosenzweig_wasserstein_nodate} and is an elementary consequence of H\"older's inequality and the Fourier representation of Riesz potentials, shows that in the nonsingular case $-2<\s<0$, the modulated energy/squared MMD controls moments of order $<|\s|/2$.

\begin{prop}\label{prop:moms}
Suppose that $-2<\s<0$, and let $0<r<\frac{|\s|}{2}$. There exists a constant $C=C(\d,\s,r)>0$ such that for any $x_0\in \R^\d$, $\mu\in L^1$ with $\int_{\R^\d}d\mu = 1$ and $\int_{(\R^\d)^2}|\g|(x-y)d|\mu|^{\otimes 2}<\infty$ and $\XN\in (\R^\d)^N$, it holds that
\begin{align}\label{eq:moms}
    \int_{\R^\d}|x-x_0|^rd\Big(\frac1N\sum_{i=1}^N\delta_{x_i}-\mu\Big) \le C (1+\|\mu\|_{L^1})^{{\cre1-\frac{2r}{|\s|}}}\Fr_N^{\frac{r}{|\s|}}(\XN,\mu).
\end{align}
\end{prop}

\subsection{Defective commutator estimate}\label{ssec:BWdc}

We now combine the functional inequalities of \cref{thm:FI'} with the mollification error estimates from \cref{lem:mollerrs} to obtain a defective commutator estimate in the form of \cref{thm:FI}
 below. As commented previously, this weakening of the Lipschitz assumption is necessary in the case $v=\M\nabla\g\ast\mu$, where $\mu$ belongs to the critical space $\dot{W}^{\frac{\d}{p}+\s+2-\d,p}$ for $1<p<\infty$. Below, $\chi_\ep$ and $v_\ep$ are as in the previous subsections. %Our first step is to prove the following lemma (cf. \cite[]{Rosenzweig}, \cite[]{Rosenzweig}), which estimates the the first variation of the modulated energy along the transport field $v-v_\ep$, or more simply, the error from replacing $v$ by $v_\ep$ in the left-hand side of \eqref{eq:expFIwts}.

In what follows, introduce the commutator kernel notation
\begin{align}\label{eq:kvdef}
    k_v(x,y) \coloneqq \nab\g(x-y)\cdot (v(x)-v(y)).
\end{align}

The proof of \cref{thm:FI} follows similar lines to \cite{rosenzweig_mean-field_2022,rosenzweig_mean-field_2022-1}. The main technical ingredient is \cref{lem:molvepLL} below, which controls the error from replacing $v$ by the mollified vector field $v_\ep$ (cf. \cite[Lemma 4.2]{rosenzweig_mean-field_2022-1} and \cite[Lemma 4.3]{rosenzweig_mean-field_2022}). The proof of the lemma is streamlined in comparison with previous works, taking advantage of better mollification estimates in \cref{lem:mollerrs} and small-scale control in \cref{cor:MEcount}. Note that we omit the log case $\s=0$ in \cref{lem:molvepLL} as this is covered by the same argument in \cite{rosenzweig_mean-field_2022-1}, and the omission simplifies the presentation.

\begin{lemma}\label{lem:molvepLL}
Assume that $-2<\s<\d$ and $\s\ne 0$. Let $v\in \dot{\mathcal{C}}^1$ and $\mu\in L^1$ with $\int_{\R^\d}d\mu =1$ and $\int_{(\R^\d)^2}|\g|(x-y)d|\mu|^{\otimes2}<\infty$. Further, assume the following:
    \begin{enumerate}[(a)]
        \item if $-2<\s\le-1$, then $\int_{\R^\d}|x|^{|\s|-1}d|\mu| < \infty$;
        \item if $-1<\s<0$, then $\int_{\R^\d}|x|^{r}d\mu < \infty$ for some $r<\frac{|\s|}{2}$;%\footnote{Obviously, by interpolation, if $\int_{\R^\d}|x|^{r'}d|\mu|<\infty$ for some $r'\ge \frac{r}{2}$, then $\int_{\R^\d}|x|^{r}d|\mu|<\infty$ for any $0\le r\le r'$, in particular for some $r<\frac{|\s|}{2}$.}
        \item if $-1\le \s<\d-1$, then $\mu\in L^p$ for some $\frac{\d}{\d-\s-1}<p\leq \infty$ ($p=1$ is allowed if $\s=-1$);
        \item if $\s\ge \d-1$, then $\mu\in \dot{\mathcal{C}}^{\theta}$ for some $\theta>\s+1-\d$.%\footnote{If $\mu\in L^1 \cap \dot{\mathcal{C}}^\theta$ for some $\theta>0$, then necessarily $\mu\in L^p$ for any $1\le p\le \infty$ because of the interpolation estimate $\|f\|_{L^p} \le \|f\|_{L^1}^{1-\frac{\d(p-1)}{p(\d+\theta)}}\|f\|_{\dot{\mathcal{C}}^\theta}^{\frac{\d(p-1)}{p(\d+\theta)}}$.}
    \end{enumerate}

    Then for any pairwise distinct configuration $\XN\in (\R^\d)^N$ and $\ep\ge0$,
\begin{multline}\label{eq:molvepLL}
\Bigg|\int_{(\R^\d)^2\setminus\triangle} \k_{(v-v_\ep)}(x,y)d\Big(\frac{1}{N}\sum_{i=1}^N\delta_{x_i}-\mu\Big)^{\otimes 2}\Bigg| \\
\le {CN^{\frac{2(\s+1)}{\s}-1}\|v\|_{\dot{\mathcal{C}}^1}\ep}\Big(\Fr_N(\XN,\mu)  + C_p\|\mu\|_{L^p}\la^{\frac{\d(p-1)}{p}-\s}\Big)^{\frac{\s+1}{\s}}\indic_{0<\s<\d} \\
+\Bigg(C_{\vartheta,\vartheta'}\|v\|_{\dot{\mathcal{C}}^1}\ep^{{\vartheta'}}\Big(\int_{\R^\d}|x|^{|\s|-\vartheta}d\mu + (1+\|\mu\|_{L^1})\Fr_N^{\frac{|\s|-\vartheta}{|\s|}}(\XN,\mu)\Big) + C\|v\|_{\dot{\mathcal{C}}^1}\ep\Bigg)\indic_{-1< \s<0}\\
+  {C\|v\|_{\dot{\mathcal{C}}^1}\ep}\Bigg(\int_{\R^\d}|x|^{|\s|-1}d\mu +  (1+\|\mu\|_{L^1})\Fr_N^{\frac{|\s|-1}{|\s|}}(\XN,\mu)\Bigg) \indic_{-2<\s\le-1}\\
+ \ep(1+\|\mu\|_{L^1})\|v\|_{\dot{\mathcal{C}}^1} \begin{cases} \int_{\R^\d}|x|^{|\s|-1}d\mu + \Fr_N^{\frac{|\s|-1}{|\s|}}(\XN,\mu) , & {-2<\s\le -1} \\ C_q\|\mu\|_{L^1}^{1-\frac{(\s+1)q}{\d(q-1)}} \|\mu\|_{L^q}^{\frac{(\s+1)q}{\d(q-1)}}, & {-1\le \s<\d-1} \\ C_\theta\|\mu\|_{L^1}^{1-\frac{\s-\d+1}{\theta}} \|\mu\|_{\dot{\mathcal{C}}^\theta}^{\frac{\s-\d+1}{\theta}} + C_\theta \ep^{\d-\s-1}\|\mu\|_{L^1}^{1-\frac{\s-\d+1}{\theta}}\|\mu\|_{L^\infty}^{\frac{\s-\d+1}{\theta}}, & {\s\ge \d-1} \end{cases},
\end{multline}
where {$\vartheta \in [|\s|-r,|\s|)$, $0<\vartheta'<\vartheta$}, $C=C(\d,\s)>0$ and $C_{\vartheta,\vartheta'}, C_p,C_\theta>0$ additionally depend on $(\vartheta,\vartheta'), p,\theta$, respectively.
\end{lemma}

Before proving \cref{lem:molvepLL}, let us show how  \cref{thm:FI} follows from it and \cref{thm:FI'}.

\begin{proof}[Proof of \cref{thm:FI}]
Writing $v=v_\ep+(v-v_\ep)$ and applying the triangle inequality, we find that
\begin{multline}\label{eq:dcomm0}
    \Bigg|\int_{(\R^\d)^2\setminus\triangle} \k_{v}(x,y)d\Big(\frac{1}{N}\sum_{i=1}^N\delta_{x_i}-\mu\Big)^{\otimes 2}\Bigg| \le \Bigg|\int_{(\R^\d)^2\setminus\triangle} \k_{v_\ep}(x,y)d\Big(\frac{1}{N}\sum_{i=1}^N\delta_{x_i}-\mu\Big)^{\otimes 2}\Bigg| \\
    + \Bigg|\int_{(\R^\d)^2\setminus\triangle} \k_{(v-v_\ep)}(x,y)d\Big(\frac{1}{N}\sum_{i=1}^N\delta_{x_i}-\mu\Big)^{\otimes 2}\Bigg|.
\end{multline}
Depending on the value of $\s$, we apply the estimates \eqref{eq:FIsupC}, \eqref{eq:FIsubC1}, \eqref{eq:FIsubC2}, or \eqref{eq:FInonsing} from \cref{thm:FI'} to the first term on the right-hand side of \eqref{eq:dcomm0}. We use \cref{lem:mollerrs}\eqref{eq:mollerrs3} and the modulus of continuity bound \eqref{eq:BWll} to control
\begin{align}
    \|\nab v_\ep\|_{L^\infty} \le C\|v\|_{\dot{W}^{\frac{\d}{p}+1,p}}\Big(1+|\log\ep|\Big)^{1-\frac1p},
\end{align}
where $C=C(\d,p)>0$. If $v\in \dot{W}^{\frac{\as}{2},\frac{2\d}{\as-2}}$, then by commutativity of differentiation and mollification,
\begin{align}
    \|\Dm^{\frac{\as}{2}}v_\ep\|_{L^{\frac{2\d}{\as-2}}} \le \|\Dm^{\frac{\as}{2}}v\|_{L^{\frac{2\d}{\as-2}}}.
\end{align}
All together, we conclude that 
\begin{multline}\label{eq:kvepcommfin}
    \Bigg|\int_{(\R^\d)^2\setminus\triangle} \k_{v_\ep}(x,y)d\Big(\frac{1}{N}\sum_{i=1}^N\delta_{x_i}-\mu\Big)^{\otimes 2}\Bigg|  \\
    \le     C\|v\|_{\dot{W}^{\frac{\d}{p}+1,p}}(1+|\log\ep|)^{1-\frac1p}\Big(\Fr_N(\XN,\mu)  + C_q\|\mu\|_{L^q}\la^{\frac{\d(q-1)}{q}-\s} \Big)\indic_{\s\ge \max(0,\d-2)}\\
   + C\min\Bigg(\Big(\|v\|_{\dot{W}^{\frac{\d}{p}+1,p}}(1+|\log\ep|)^{1-\frac1p}  + \|\Dm^{\frac{\as}{2}}v\|_{L^{\frac{2\d}{\as-2}}}\indic_{\substack{\as>2}}\Big) \Big(\Fr_N(\XN,\mu)  + C_q\|\mu\|_{L^q}\la^{\frac{\d(q-1)}{q}-\s} \Big),\\
   \Big(\|v\|_{\dot{W}^{\frac{\d}{p}+1,p}}(1+|\log\ep|)^{1-\frac1p}  + \|\Dm^{\frac{\d-\s}{2}}v\|_{L^{\frac{2\d}{\d-\s-2}}}\Big)\Big(\Fr_N(\XN,\mu) + C_{\tl{q}}\|\mu\|_{L^{\tl{q}}}\ka^{\frac{\d(\tl{q}-1)}{\tl{q}}-\s}\Big) \Bigg)\indic_{0\le\s<\d-2}\\
   +C\Big(\|v\|_{\dot{W}^{\frac{\d}{p}+1,p}}(1+|\log\ep|)^{1-\frac1p}  + \|\Dm^{\frac{\d-\s}{2}}v\|_{L^{\frac{2\d}{\d-\s-2}}}\Big)\Fr_N(\XN,\mu)\indic_{-2<\s<0},
\end{multline}
for $q>\frac{\d}{\d-\s}$ and $\tl{q}>\frac{\d}{\d-\s-1}$, where $\la,\ka$ are as in \eqref{eq:ladef}, \eqref{eq:kadef}, respectively. Applying \cref{lem:molvepLL}\eqref{eq:molvepLL} to the second term on the right-hand side of \eqref{eq:dcomm0} and combining the resulting bound with \eqref{eq:kvepcommfin} then completes the proof.
\end{proof}

We now turn to the proof of \cref{lem:molvepLL}.

\begin{proof}[Proof of \cref{lem:molvepLL}]
Our starting point is the decomposition
\begin{equation}
\int_{(\R^\d)^2\setminus\triangle} \k_{(v-v_\ep)}(x,y)d\Big(\frac{1}{N}\sum_{i=1}^N\delta_{x_i}-\mu\Big)^{\otimes 2}(x,y) = \sum_{j=1}^3 \Te_j,
\end{equation}
where
\begin{align}
\Te_1 &= \frac{1}{N^2}\sum_{1\leq i\neq j\leq N} \k_{(v-v_\ep)}(x_i,x_j), \\
\Te_2 &= -2\frac1N\sum_{i=1}^N \int_{\R^\d}\k_{(v-v_\ep)}(x_i,y)d\mu(y), \\
\Te_3 &= \int_{\R^\d}\k_{(v-v_\ep)}(x,y)d\mu^{\otimes 2}(x,y).
\end{align}
We separately estimate each of $\Te_1,\Te_2,\Te_3$. %We do not consider the log case $\s=0$ since this is already covered by \cite{rosenzweig_mean-field_2022-1}.

\medskip
\noindent\textbullet \ $\Te_1$: For $1\leq i\leq N$, let $\rs_i$ be the nearest-neighbor type distance from \eqref{eq:rsi_def}. By \cref{lem:mollerrs}\eqref{eq:mollerrs1} with $a=1$, we have the elementary bound
\begin{align}\label{eq:kvvep}
   \forall x\neq y, \qquad \left|\k_{(v-v_\ep)}(x,y)\right| \leq \frac{C\|v\|_{\dot{\mathcal{C}}^1}\ep}{|x-y|^{\s+1}}.
\end{align}
We now consider cases based on the value of $\s$.

If $0<\s<\d$, then applying the bound \eqref{eq:kvvep} to each summand of $\Te_1$,
\begin{align}
    |\Te_1| \le \frac{C\|v\|_{\dot{\mathcal{C}}^1}\ep}{N^2}\sum_{i=1}^N \sum_{j\ne i} |x_i-x_j|^{-\s-1} &\le \frac{C\|v\|_{\dot{\mathcal{C}}^1}\ep}{N}\sum_{i=1}^N \rs_i^{-\s-1}\nn\\
    &\le \frac{C\|v\|_{\dot{\mathcal{C}}^1}\ep}{N}\Big(\sum_{i=1}^N\rs_i^{-\s}\Big)^{\frac{\s+1}{\s}} \nn\\
    &\le \frac{C'(\s N^2)^{\frac{\s+1}{\s}}\ep}{N}\Big(\Fr_N(\XN,\mu)  + C_p\|\mu\|_{L^p}\la^{\frac{\d(p-1)}{p}-\s}\Big)^{\frac{\s+1}{\s}}, \label{eq:T1fin}
\end{align}
where the second inequality follows from the definition of $\rs_i$, the third inequality from the embedding $\|\cdot\|_{\ell^{\s+1}} \le \|\cdot\|_{\ell^{\s}}$, and the fourth inequality from \cref{cor:MEcount}. Above, $C,C' >0$ depend only on $\d,\s$ and $\frac{\d}{\d-\s}<p\le \infty$.
If $-2<\s\le -1$, then we use $1>|\s|-1\ge 0$ to instead argue that
\begin{align}
    |\Te_1| \le \frac{C\|v\|_{\dot{\mathcal{C}}^1}\ep}{N^2}\sum_{i=1}^N \sum_{j\ne i} |x_i-x_j|^{-\s-1} &\le \frac{C\|v\|_{\dot{\mathcal{C}}^1}\ep}{N^2}\sum_{i=1}^N \sum_{j\ne i} (|x_i|^{|\s|-1} + |x_j|^{|\s|-1}) \nn\\
    &\le \frac{2C\|v\|_{\dot{\mathcal{C}}^1}\ep}{N}\sum_{i=1}^N |x_i|^{|\s|-1} \nn\\
    &\le {2C'\|v\|_{\dot{\mathcal{C}}^1}\ep}\Bigg(\int_{\R^\d}|x|^{|\s|-1}d\mu +(1+\|\mu\|_{L^1})^{1-\frac{2(|\s|-1)}{|\s|}}  \Fr_N^{\frac{|\s|-1}{|\s|}}(\XN,\mu)\Bigg), \label{eq:T1fin'}
\end{align}
where we have used $|\s|-1<\frac{|\s|}{2}$ and \cref{prop:moms} to obtain the final line.

If $-1<\s<0$, then in addition to \eqref{eq:kvvep}, we observe the bound
\begin{align}\label{eq:kvvep'}
     \forall |x-y|\le 1, \qquad \left|\k_{(v-v_\ep)}(x,y)\right| \leq C\|v\|_{\dot{\mathcal{C}}^1}|x-y|^{-\s}(1+|\log|x-y||).
\end{align}
which is a consequence of $\dot{\mathcal{C}}^1$ functions being log-Lipschitz (e.g. see \cite[Proposition 2.107]{bahouri_fourier_2011}). Interpolating between \eqref{eq:kvvep} and \eqref{eq:kvvep'}, we find that for any $\vartheta \in [0,1]$,
\begin{align}
    \forall |x-y|\le 1, \qquad \left|\k_{(v-v_\ep)}(x,y)\right| \le C\|v\|_{\dot{\mathcal{C}}^1}|x-y|^{-\s}\left(\frac{\ep}{|x-y|}\right)^{\vartheta}(1+|\log|x-y||)^{1-\vartheta}.
\end{align}
{Note that we may discard the $\log$ factor up to taking $\vartheta$ smaller in the exponent of $\ep$. Indeed, for $1\ge\vartheta>\vartheta'$, we have that
\begin{align}
    \sup_{|x-y|\le 1}|x-y|^{\vartheta-\vartheta'}(1+|\log|x-y||)^{1-\vartheta'} \le C',
\end{align} 
and therefore,
\begin{align}
     |x-y|^{-\s}\left(\frac{\ep}{|x-y|}\right)^{\vartheta'}(1+|\log|x-y||)^{1-\vartheta'} &= \ep^{\vartheta'}|x-y|^{|\s|-\vartheta} |x-y|^{\vartheta-\vartheta'}(1+|\log|x-y||)^{1-\vartheta'} \nn\\
     &\le C' \ep^{\vartheta'}|x-y|^{|\s|-\vartheta} .
\end{align}}
Hence, choosing $0<\vartheta'<\vartheta<1$ so that $r\ge|\s|-\vartheta >0$, which translates to $|\s|-r\le \vartheta<|\s|$, and arguing similarly to the previous case,
\begin{align}
    |\Te_1| &\le \frac{1}{N^2}\sum_{i=1}^N \sum_{\substack{j: |x_i-x_j|\le 1}} \left|\k_{(v-v_\ep)}(x,y)\right| + \frac{1}{N^2}\sum_{i=1}^N\sum_{\substack{j: |x_i-x_j|> 1}} \left|\k_{(v-v_\ep)}(x,y)\right| \nn\\
    &\le \frac{C\|v\|_{\dot{\mathcal{C}}^1}\ep^{{\vartheta'}}}{N}\sum_{i=1}^N |x_i|^{|\s|-\vartheta}  + C\|v\|_{\dot{\mathcal{C}}^1}\ep \nn\\
    &\le C\|v\|_{\dot{\mathcal{C}}^1}\ep^{{\vartheta'}}\Bigg(\int_{\R^\d}|x|^{|\s|-\vartheta}d\mu + (1+\|\mu\|_{L^1})^{1-\frac{2(|\s|-\vartheta)}{|\s|}}\Fr_N^{\frac{|\s|-\vartheta}{|\s|}}(\XN,\mu)\Bigg) + C\|v\|_{\dot{\mathcal{C}}^1}\ep. \label{eq:T1fin''}
\end{align}

Combining the bounds \eqref{eq:T1fin}, \eqref{eq:T1fin'}, \eqref{eq:T1fin''} for $0<\s<\d$, $-2<\s\le -1$, and $-1<\s<0$, respectively, we conclude
\begin{multline}\label{eq:T1finn}
    |\Te_1| \le \frac{C(N^2)^{\frac{\s+1}{\s}}\ep}{N}\Big(\Fr_N(\XN,\mu)  + C_p\|\mu\|_{L^p}\la^{\frac{\d(p-1)}{p}-\s}\Big)^{\frac{\s+1}{\s}}\indic_{0<\s<\d}\\
   + {C\|v\|_{\dot{\mathcal{C}}^1}\ep}\Bigg(\int_{\R^\d}|x|^{|\s|-1}d\mu + (1+\|\mu\|_{L^1})^{1-\frac{2(|\s|-1)}{|\s|}}\Fr_N^{\frac{|\s|-1}{|\s|}}(\XN,\mu)\Bigg)\indic_{-2<\s\le -1}\\
    +\Bigg(C_{\vartheta,\vartheta'}\|v\|_{\dot{\mathcal{C}}^1}\ep^{{\vartheta'}}\Bigg(\int_{\R^\d}|x|^{|\s|-\vartheta}d\mu + (1+\|\mu\|_{L^1})^{1-\frac{2(|\s|-\vartheta)}{|\s|}}\Fr_N^{\frac{|\s|-\vartheta}{|\s|}}(\XN,\mu)\Bigg) + C\|v\|_{\dot{\mathcal{C}}^1}\ep\Bigg)\indic_{-1<\s<0},
\end{multline}
where $C=C(\d,\s)>0$ and $C_{\vartheta,\vartheta'}>0$ additionally depends on $\vartheta,\vartheta'$.

\medskip
\noindent\textbullet \ $\Te_2$: Unpacking the definition of $\k_{(v-v_\ep)}$,
\begin{align}\label{eq:T20}
\int_{\R^\d}\k_{(v-v_\ep)}(\cdot,y)d\mu(y) = (v-v_\ep)\cdot\nabla h^{\mu} + \div h^{(v-v_\ep)\mu},
\end{align}
where $h^{f}\coloneqq \g\ast f$. For the first term on the right-hand side, we crudely bound
\begin{align}
\|(v-v_\ep)\cdot\nabla h^{\mu}\|_{L^\infty} &\leq \|v-v_\ep\|_{L^\infty} \|\nabla h^{\mu}\|_{L^\infty} \nn\\
&\le \|v\|_{\dot{\mathcal{C}}^1}\ep\begin{cases} C_q \|\mu\|_{L^1}^{1-\frac{(\s+1)q}{\d(q-1)}} \|\mu\|_{L^q}^{\frac{(\s+1)q}{\d(q-1)}}, & {-1\le \s<\d-1} \\ C_\theta\|\mu\|_{L^1}^{1-\frac{\s-\d+1}{\theta}} \|\mu\|_{\dot{\mathcal{C}}^{\theta}}^{\frac{\s-\d+1}{\theta}}, & {\s\ge \d-1}, \end{cases} \label{eq:T21}
\end{align}
for $q>\frac{\d}{\d-\s-1}$ ($q\ge1$ if $\s=-1$) and $\theta>\s+1-\d$, where in the final inequality, we have used \cref{lem:mollerrs}\eqref{eq:mollerrs1} and the interpolation estimate
\begin{align}\label{eq:nabhf}
    \|\nab h^f\|_{L^\infty} \le \begin{cases} C_q\|f\|_{L^1}^{1-\frac{(\s+1)q}{\d(q-1)}} \|f\|_{L^q}^{\frac{(\s+1)q}{\d(q-1)}}, & {-1\le \s<\d-1} \\ C_\theta\|f\|_{L^1}^{1-\frac{\s-\d+1}{\theta}} \|f\|_{\dot{\mathcal{C}}^{\theta}}^{\frac{\s-\d+1}{\theta}}, & {\s\ge \d-1}. \end{cases} 
\end{align}

%\begin{align}
%    |(v-v_\ep)(x)\cdot\nab h^{\mu}(x)| \le \|v-v_\ep\|_{L^\infty}\Bigg(\int_{|x-y|\le 2|x|}|x-y|^{-\s}d\mu(y) + \int_{|x-y|> 2|x|}|x-y|^{-\s}d\mu(y)\Bigg)
%\end{align}

For the second term on the right-hand side of \eqref{eq:T20}, we again use the interpolation estimate \eqref{eq:nabhf} and H\"older's inequality to obtain
\begin{align}
    &\|\div h^{(v-v_\ep)\mu}\|_{L^\infty}\nn\\
     &\le \begin{cases} C_q\|(v-v_\ep)\mu\|_{L^1}^{1-\frac{(\s+1)q}{\d(q-1)}} \|(v-v_\ep)\mu\|_{L^q}^{\frac{(\s+1)q}{\d(q-1)}} , & {-1\le \s<\d-1} \\ C_\theta\|(v-v_\ep)\mu\|_{L^1}^{1-{\frac{\s-\d+1}{\theta}}} \|(v-v_\ep)\mu\|_{\dot{\mathcal{C}}^\theta}^{\frac{\s-\d+1}{\theta}}, & {\s\ge \d-1} \end{cases} \nn\\
    &\le \begin{cases} C_q\|v-v_\ep\|_{L^\infty} \|\mu\|_{L^1}^{1-\frac{(\s+1)q}{\d(q-1)}} \|\mu\|_{L^q}^{\frac{(\s+1)q}{\d(q-1)}},  & {-1\le\s<\d-1} \\  C_\theta (\|v-v_\ep\|_{L^\infty} \|\mu\|_{L^1})^{1-\frac{\s-\d+1}{\theta}} (\|v-v_\ep\|_{L^\infty}\|\mu\|_{\dot{\mathcal{C}}^\theta} + \|v-v_\ep\|_{\dot{\mathcal{C}}^\theta} \|\mu\|_{L^\infty})^{\frac{\s-\d+1}{\theta}}, & {\s\ge \d-1} \end{cases} \nn\\
    &\le \begin{cases} C_q C \ep\|v\|_{\dot{\mathcal{C}}^1}\|\mu\|_{L^1}^{1-\frac{(\s+1)q}{\d(q-1)}} \|\mu\|_{L^q}^{\frac{(\s+1)q}{\d(q-1)}},  & {-1\le \s<\d-1} \\ C_\theta(C\ep\|v\|_{\dot{\mathcal{C}}^1} \|\mu\|_{L^1})^{1-\frac{\s-\d+1}{\theta}}(C\ep\|v\|_{\dot{\mathcal{C}}^1}\|\mu\|_{\dot{\mathcal{C}}^\theta} + C_\theta'\|v\|_{\dot{\mathcal{C}}^1}\ep^{1-\theta}\|\mu\|_{L^\infty})^{\frac{\s-\d+1}{\theta}} , & {\s\ge \d-1} \end{cases} \label{eq:T22}
\end{align}
where we have also used the algebra property of H\"older-Zygmund spaces. 

%From \eqref{eq:T21}, \eqref{eq:T22}, it follows that if $-1\le \s<\d$, then
%\begin{align}
%    |\Te_2| \le 
%\end{align}

{If $-2<\s<-1$, then we instead use \eqref{eq:kvvep} to bound
\begin{align}
    \int_{\R^\d}|k_{(v-v_\ep)}(x,y)|d|\mu|(y) &\le C\ep\|v\|_{\dot{\mathcal{C}}^1}\int_{\R^\d}|x-y|^{-\s-1}d|\mu|(y) . %\nn\\
    %&\le C'\ep \|v\|_{\dot{\mathcal{C}}^1}\Big( |x|^{|\s|-1}\|\mu\|_{L^1} + \int_{\R^\d}|y|^{|\s|-1}d|\mu|\Big). \label{eq:T23'}
\end{align}
Using \cref{prop:moms} with $x_0=y$
\begin{align}
    &\frac1N\sum_{i=1}^N \int_{\R^\d}|x_i-y|^{|\s|-1}d|\mu|(y) \nn\\
    &\le C \int_{\R^\d}\bigg(\int_{\R^\d}|x-y|^{|\s|-1}d\mu + (1+\|\mu\|_{L^1})^{1-\frac{2(|\s|-1)}{|\s|}}\Fr_N^{\frac{|\s|-1}{|\s|}}(\XN,\mu)\bigg)d|\mu|(y) \nn\\
    &\le C'\|\mu\|_{L^1}\bigg(\int_{\R^\d}|x|^{|\s|-1}d|\mu| +(1+\|\mu\|_{L^1})^{1-\frac{2(|\s|-1)}{|\s|}}\Fr_N^{\frac{|\s|-1}{|\s|}}(\XN,\mu)\bigg).
\end{align}
Hence,
\begin{multline}\label{eq:T23}
    \frac1N\sum_{i=1}^N \int_{\R^\d}|k_{(v-v_\ep)}(x_i,y)|d|\mu|(y) \\
    \le C''\ep\|v\|_{\dot{\mathcal{C}}^1}\|\mu\|_{L^1}\bigg(\int_{\R^\d}|x|^{|\s|-1}d|\mu| +(1+\|\mu\|_{L^1})^{1-\frac{2(|\s|-1)}{|\s|}}\Fr_N^{\frac{|\s|-1}{|\s|}}(\XN,\mu)\bigg).
\end{multline}
}
Combining the estimates \eqref{eq:T20}, \eqref{eq:T21}, \eqref{eq:T22}, \eqref{eq:T23}, simplifying and relabeling constants, we conclude that
\begin{align}\label{eq:T2fin}
|\Te_2|  \le \ep\|v\|_{\dot{\mathcal{C}}^1} \begin{cases} {\|\mu\|_{L^1}\Big(\int_{\R^\d}|x|^{|\s|-1}d\mu + (1+\|\mu\|_{L^1})^{\frac{2}{|\s|}}\Fr_N^{\frac{|\s|-1}{|\s|}}(\XN,\mu)\Big)}, & {-2<\s<-1} \\ C_q\|\mu\|_{L^1}^{1-\frac{(\s+1)q}{\d(q-1)}} \|\mu\|_{L^q}^{\frac{(\s+1)q}{\d(q-1)}}, & {-1\le \s<\d-1} \\ C_\theta\|\mu\|_{L^1}^{1-\frac{\s-\d+1}{\theta}} \|\mu\|_{\dot{\mathcal{C}}^\theta}^{\frac{\s-\d+1}{\theta}} + C_\theta \ep^{\d-\s-1}\|\mu\|_{L^1}^{1-\frac{\s-\d+1}{\theta}}\|\mu\|_{L^\infty}^{\frac{\s-\d+1}{\theta}}, & {\s\ge \d-1} \end{cases}
\end{align}
for $q>\frac{\d}{\d-\s-1}$ ($q\ge 1$ if $\s=-1$) and $\theta>\s+1-\d$. 

\medskip
\noindent\textbullet \ $\Te_3$: Unpacking the definition of $\k_{(v-v_\ep)}$ and desymmetrizing,
\begin{equation}
\Te_3 = \int_{\R^\d} (v-v_\ep)\cdot\nabla h^{\mu} d\mu,
\end{equation}
and by H\"older's inequality,
\begin{align}
\left|\int_{\R^\d} (v-v_\ep)\cdot\nabla h^{\mu} d\mu\right| &\leq \|v-v_\ep\|_{L^\infty} \|\nabla h^{\mu}\|_{L^\infty} \|\mu\|_{L^1}.
\end{align}
Using the estimate \eqref{eq:T21}, we conclude that
\begin{align}\label{eq:T31}
|\Te_3| \le  \ep\|v\|_{\dot{\mathcal{C}}^1}\|\mu\|_{L^1}\begin{cases}C_q \|\mu\|_{L^1}^{1-\frac{(\s+1)q}{\d(q-1)}} \|\mu\|_{L^q}^{\frac{(\s+1)q}{\d(q-1)}}, & {-1\le \s<\d-1} \\ C_\theta\|\mu\|_{L^1}^{1-\frac{\s-\d+1}{\theta}} \|\mu\|_{\dot{\mathcal{C}}^{\theta}}^{\frac{\s-\d+1}{\theta}}, & {\s\ge \d-1}, \end{cases} 
\end{align}
for $q>\frac{\d}{\d-\s-1}$ and $\theta>\s+1-\d$.

If $-2<\s<-1$, then we instead argue similarly to $\Te_2$ to obtain
\begin{align}\label{eq:T32}
    |\Te_3| \le C\ep\|v\|_{\dot{\mathcal{C}}^1}\|\mu\|_{L^1}\int_{\R^\d}|x|^{|\s|-1}d\mu.
\end{align}

Combining the estimates \eqref{eq:T31}, \eqref{eq:T32}, we conclude that
\begin{align}\label{eq:T3fin}
    |\Te_3| \le \ep\|v\|_{\dot{\mathcal{C}}^1}\|\mu\|_{L^1}\begin{cases} \int_{\R^\d}|x|^{|\s|-1}d\mu, & {-2<-\s<-1} \\ C_q \|\mu\|_{L^1}^{1-\frac{(\s+1)q}{\d(q-1)}} \|\mu\|_{L^q}^{\frac{(\s+1)q}{\d(q-1)}}, & {-1\le \s<\d-1} \\ C_\theta\|\mu\|_{L^1}^{1-\frac{\s-\d+1}{\theta}} \|\mu\|_{\dot{\mathcal{C}}^{\theta}}^{\frac{\s-\d+1}{\theta}}, & {\s\ge \d-1},\end{cases}
\end{align}
for $q>\frac{\d}{\d-\s-1}$ and $\theta>\s+1-\d$.

\medskip
\noindent\textbullet \ Conclusion of proof: Combining the estimates \eqref{eq:T1finn}, \eqref{eq:T2fin}, \eqref{eq:T3fin} for $\Te_1, \Te_2,\Te_3$, respectively, we arrive at
\begin{multline}
\Bigg|\int_{(\R^\d)^2\setminus\triangle} \k_{(v-v_\ep)}(x,y)d\Big(\frac{1}{N}\sum_{i=1}^N\delta_{x_i}-\mu\Big)^{\otimes 2}\Bigg| \\
\leq  \frac{C(N^2)^{\frac{\s+1}{\s}}\|v\|_{\dot{\mathcal{C}}^1}\ep}{N}\Big(\Fr_N(\XN,\mu)  + C_p\|\mu\|_{L^p}\la^{\frac{\d(p-1)}{p}-\s}\Big)^{\frac{\s+1}{\s}}\indic_{0<\s<\d} \\
+   {C\|v\|_{\dot{\mathcal{C}}^1}\ep}(1+\|\mu\|_{L^1})^{\frac{2}{|\s|}}\Bigg(\int_{\R^\d}|x|^{|\s|-1}d\mu +  \Fr_N^{\frac{|\s|-1}{|\s|}}(\XN,\mu)\Bigg)\indic_{-2<\s\le -1}\\
    +C_{\vartheta,\vartheta'}\|v\|_{\dot{\mathcal{C}}^1}\Bigg(\ep^{{\vartheta'}}\bigg(\int_{\R^\d}|x|^{|\s|-\vartheta}d\mu + (1+\|\mu\|_{L^1})^{\frac{2\vartheta}{|\s|}}\Fr_N^{\frac{|\s|-\vartheta}{|\s|}}(\XN,\mu)\bigg) + \ep\Bigg)\indic_{-1<\s<0}\\
+ \ep\|\mu\|_{L^1}\|v\|_{\dot{\mathcal{C}}^1} \begin{cases} {C\Big(\int_{\R^\d}|x|^{|\s|-1}d\mu + (1+\|\mu\|_{L^1})^{\frac{2}{|\s|}}\Fr_N^{\frac{|\s|-1}{|\s|}}(\XN,\mu)\Big)}, & {-2<\s\le -1} \\ C_q\|\mu\|_{L^1}^{1-\frac{(\s+1)q}{\d(q-1)}} \|\mu\|_{L^q}^{\frac{(\s+1)q}{\d(q-1)}}, & {-1\le \s<\d-1} \\ C_\theta\|\mu\|_{L^1}^{1-\frac{\s-\d+1}{\theta}} \|\mu\|_{\dot{\mathcal{C}}^\theta}^{\frac{\s-\d+1}{\theta}} + C_\theta \ep^{\d-\s-1}\|\mu\|_{L^1}^{1-\frac{\s-\d+1}{\theta}}\|\mu\|_{L^\infty}^{\frac{\s-\d+1}{\theta}}, & {\s\ge \d-1} \end{cases},
\end{multline}
for $\vartheta \in [|\s|-r,|\s|)$, {$0<\vartheta'<\vartheta$}, $p>\frac{\d}{\d-\s}$, $q>\frac{\d}{\d-\s-1}$, and $\theta>\s+1-\d$. The constant $C$ depends only on $\d,\s$, while the constants $C_{\vartheta,\vartheta'}, C_p,C_q,C_\theta$ additionally depend on $(\vartheta,\vartheta'), p,q,\theta$, respectively. Comparing this inequality to the statement of the lemma, we see that the proof is complete.
\end{proof}

\subsection{Mean-field convergence}\label{ssec:BWmp}
Using the defective commutator estimate of \cref{thm:FI}, we now show convergence of the empirical measure for the mean-field particle dynamics \eqref{eq:MFode} to a (necessarily unique) solution of the limiting PDE \eqref{eq:MFlim} in the modulated energy distance. This proves \cref{thm:MF}. In contrast to previous works  \cite[Theorem 1.5]{rosenzweig_sharp_nodate},  \cite[Theorem 1.5]{hess-childs_sharp_2025}, the rate of convergence is not the optimal $O(N^{\frac{\s}{\d}-1})$ rate. This is not surprising giving we have to pay a price for relaxing the assumption that the mean-field vector field is Lipschitz.

To close a differential inequality for $\Fr_N(\XN^t,\mu^t)$, which is the key step of the theorem, we will use the following generalization of the Gr\"onwall-Bellman inequality.

\begin{lemma}\label{lem:Gron}
    Let $a\ge 0$. Suppose that $C_1,C_2:[0,T]\rightarrow [0,\infty)$ are locally integrable, and $x:[0,T]\rightarrow [0,\infty)$ is absolutely continuous with
    \begin{align}
        \dot{x}^t \le C_1^t (x^t)^a + C_2^tx^t, \qquad \text{a.e. on} \ [0,T].
    \end{align}
    
    If $a>1$, set $I^t \coloneqq (a-1)\int_0^t C_1^s e^{(a-1)\int_0^s C_2^\tau d\tau}ds$. If $T_*$ is defined as the maximal time such that $I^{T_*} \le (x^0)^{1-a}$, then
    \begin{align}\label{eq:Grona>1}
        \forall t < T_*, \qquad x^t \le e^{\int_0^t C_2^s ds}\Big((x^0)^{1-a} - I^t \Big)^{-\frac{1}{a-1}}.
    \end{align}

    If $a\le 1$, then
    \begin{align}\label{eq:Grona<1}
        \forall t\le T, \qquad (x^t)^{1-a} \le (x^0)^{1-a}e^{(1-a)\int_0^t C_2^\tau d\tau} + (1-a)\int_0^t C_1^s e^{(1-a)\int_{s}^t C_2^\tau d\tau} ds.
    \end{align}
\end{lemma}
\begin{proof}
    For the assertion \eqref{eq:Grona>1}, set $z^t \coloneqq e^{-\int_0^t C_2^sds}x^t$, which satisfies
    \begin{align}
        \dot{z}^t \le e^{-\int_0^t C_2^sds}C_1^t (x^t)^a = C_1^t e^{(a-1)\int_0^tC_2^sds}(z^t)^a.
    \end{align}
    Now let $w^t \coloneqq (z^t)^{1-a}$, which is well-defined by our assumption $x^t\ge 0$. Since $1-a<0$,
    \begin{align}
        \dot{w}^t = (1-a)(z^t)^{-a}\dot{z}^t\ge (1-a)C_1^te^{(a-1)\int_0^t C_2^sds}.
    \end{align}
    Integrating both sides from $0$ to $t$ and then rearranging completes the proof.

    For the assertion \eqref{eq:Grona<1}, set $y^t \coloneqq (z^t)^{1-a}$, which satisfies the differential inequality
    \begin{align}
        \dot{y}^t \le (1-a)C_1^t + (1-a)C_2^t y^t.
    \end{align}
    Multiplying both sides of the preceding inequality by $e^{-(1-a)\int_0^t C_2^sds}$, rearranging, then integrating yields the desired conclusion.
    %Applying the usual Gr\"{o}nwall-Bellman lemma to $y^t$, we obtain
    %\begin{align}
    %    y^t \le \Big(y^0 + (1-a)\int_0^t C_1^\tau d\tau\Big)e^{(1-a)\int_0^t C_2^\tau d\tau}.
   % \end{align}
\end{proof}

With this lemma in hand, we turn to the proof of \cref{thm:MF}.

\begin{proof}[Proof of \cref{thm:MF}]
The second assertion concerning convergence in $H^{-r}$ for $r>\frac\d2$ follows immediately from the first assertion and \cref{prop:coer}. Therefore, we concentrate on the modulated energy.

Recall from \cite[Lemma 2.1]{serfaty_mean_2020} or \cite[Lemma 3.6]{rosenzweig_relative_2024} that $\Fr_N(\ux_N^t,\mu^t)$ satisfies the differential inequality
\begin{align}
\frac{d}{dt}\Fr_N(\ux_N^t,\mu^t) \leq \int_{(\R^\d)^2\setminus\triangle}\nabla\g(x-y)\cdot\pa*{u^t(x)-u^t(y)}d\Big(\frac{1}{N}\sum_{i=1}^N\delta_{x_i^t}-\mu^t\Big)^{\otimes2},
\end{align}
where $u^t\coloneqq {\mathsf{V}^t}-\M\nabla\g\ast\mu^t$. To make the presentation more digestible, let us break the argument into cases based on whether $\s$ is (super-)Coulomb, sub-Coulomb, or nonsingular.

\medskip

\textbullet ($\s\ge \d-2$) \ Note that $\frac{\d}{p}+\s+2-\d\ge \frac{\d}{p}$, with strict inequality if $\s>\d-2$, and so by Sobolev embedding,
\begin{align}
    \|\mu^t\|_{L^q} &\le C\|\mu^t\|_{\dot{W}^{\frac{\d}{p}+\s+2-\d,p}}^{\frac{\d(q-1)}{q(\s+2)}}, \qquad (q,\s)\ne (\infty,\d-2)\\
    \|u^t\|_{\dot{W}^{\frac{\d}{p}+1.p}} &\le {\|\mathsf{V}^t\|_{\dot{W}^{\frac{\d}{p}+1,p}} +} C|\M|\|\mu^t\|_{\dot{W}^{\frac{\d}{p}+\s+2-\d,p}},\\
    \|\mu^t\|_{\dot{\mathcal{C}}^{\s+2-\d}} &\le C\|\mu^t\||_{\dot{W}^{\frac{\d}{p}+\s+2-\d,p}}.
\end{align}
For a.e. $t$, we apply \cref{thm:FI} with $v=u^t$, $q<\infty$ if $\s=\d-2$ and $q=\infty$ if $\s\ge \d-2$, and $\theta = \s+2-\d$ if $\s\ge\d-1$, to obtain
\begin{multline}\label{eq:MFsupC0}
\frac{d}{dt}\Fr_N(\ux_N^t,\mu^t) \le     C_p\|u^t\|_{\dot{W}^{\frac{\d}{p}+1,p}}(1+|\log\ep_t|)^{1-\frac1p}\Big(\Fr_N(\XN^t,\mu^t)  + C_q\|\mu^t\|_{L^q}\la_t^{\frac{\d(q-1)}{q}-\s} \Big)\\
  + {CN^{\frac{2(\s+1)}{\s}-1}\ep_t\|u^t\|_{\dot{\mathcal{C}}^1}}\Big(\Fr_N(\XN^t,\mu^t) + C_q\|\mu^t\|_{L^q}\la_t^{\frac{\d(q-1)}{q}-\s}\Big)^{\frac{\s+1}{\s}} \\
+ \ep_t\|u^t\|_{\dot{\mathcal{C}}^1}\Big(C_q\|\mu^t\|_{L^q}^{\frac{(\s+1)q}{\d(q-1)}}\indic_{\s<\d-1} + C_\theta(\|\mu^t\|_{\dot{\mathcal{C}}^\theta}^{\frac{\s-\d+1}{\theta}} + \ep_t^{\d-\s-1}\|\mu^t\|_{L^\infty}^{\frac{\s-\d+1}{\theta}}) \indic_{\s\ge\d-1}\Big),
\end{multline}
where $\ep_t$ is to be chosen momentarily and $\la_t$ corresponds to \eqref{eq:ladef} with $\mu=\mu^t$. Letting
\begin{multline}
    \mathscr{E}^t \coloneqq \Fr_N(\XN^t,\mu^t)  + \sup_{t\le \tau\le T}\ep_\tau\Big(C_q\|\mu^\tau\|_{L^q}\la_\tau^{\frac{\d(q-1)}{q}-\s} + C_q\|\mu^\tau\|_{L^q}^{\frac{(\s+1)q}{\d(q-1)}}\indic_{\s<\d-1}\\
    + C_\theta(\|\mu^\tau\|_{\dot{\mathcal{C}}^\theta}^{\frac{\s-\d+1}{\theta}} + \ep_\tau^{\d-\s-1}\|\mu^\tau\|_{L^\infty}^{\frac{\s-\d+1}{\theta}}) \indic_{\s\ge\d-1}\Big),
\end{multline}
it follows from the inequality \eqref{eq:MFsupC0} that
\begin{align}\label{eq:dtEscr}
    \frac{d}{dt}\mathscr{E}^t  \le C_1^t(\mathscr{E}^t)^{\frac{\s+1}{\s}} + C_2^t\mathscr{E}^t,
\end{align}
where
\begin{align}
    C_1^t &\coloneqq {CN^{\frac{2(\s+1)}{\s}-1}\ep_t\|u^t\|_{\dot{\mathcal{C}}^1}}, \label{eq:C1tdef}\\
    C_2^t &\coloneqq C_p\|u^t\|_{\dot{W}^{\frac{\d}{p}+1,p}}(1+|\log\ep_t|)^{1-\frac1p}. \label{eq:C2tdef}
\end{align}
Appealing to \cref{lem:Gron}\eqref{eq:Grona>1} with $a=\frac{\s+1}{\s}$, we obtain
\begin{multline}\label{eq:MFsupC1}
    \forall 0\le t < T_*, \qquad \mathscr{E}^t \le \exp\Big(\int_0^{t} C_p\|u^{t'}\|_{\dot{W}^{\frac{\d}{p}+1,p}}(1+|\log\ep_{t'}|)^{1-\frac1p}dt'\Big)\\
    \times\Bigg( (\mathscr{E}^0)^{-\frac{1}{\s}} - \frac1\s\int_0^t N^{\frac{2(\s+1)}{\s}-1}\ep_{t'}\|u^{t'}\|_{\dot{\mathcal{C}}^1}e^{\frac{1}{\s}\int_0^{t'}C_p\|u^{t''}\|_{\dot{W}^{1+\frac1p,p}}(1+|\log\ep_{t''}|)^{1-\frac1p}dt''}dt' \Bigg)^{-\s}.
\end{multline}

We have quite a bit of freedom to choose $\ep_t$ to obtain a satisfactory estimate. For example, choose $\ep=\ep_t$ (i.e. time-independent) to satisfy the implicit equation $\ep = \delta N^{1-\frac{2(\s+1)}{\s}-\al}(\mathscr{E}^0)^{-\frac{1}{\s}}$ for $\delta,\al>0$ fixed. This is possible to do because unpacking the definition of $\mathscr{E}^0$, the implicit equation for $\ep$ may be rewritten as
\begin{multline}
    \ep^{\s}\bigg(\Fr_N(\XN^0,\mu^0)+\ep\sup_{0\le \tau\le T} \Big(C_q\|\mu^\tau\|_{L^q}\la_\tau^{\frac{\d(q-1)}{q}-\s} + C_q\|\mu^\tau\|_{L^q}^{\frac{(\s+1)q}{\d(q-1)}}\indic_{\s<\d-1}\\
    + C_\theta(\|\mu^\tau\|_{\dot{\mathcal{C}}^\theta}^{\frac{\s-\d+1}{\theta}} + \ep^{\d-\s-1}\|\mu^\tau\|_{L^\infty}^{\frac{\s-\d+1}{\theta}}) \indic_{\s\ge\d-1}\Big)\bigg) = (\delta N^{1-\frac{2(\s+1)}{\s}-\al})^{\s}.
\end{multline}
The left-hand side, which is a continuous function of $\ep$, tends to zero as $\ep\rightarrow 0$ and to $\infty$ as $\ep\rightarrow\infty$. So, a solution exists by the intermediate value theorem.

With this choice for $\ep_t$,
\begin{align}
     &\frac1\s\int_0^t N^{\frac{2(\s+1)}{\s}-1}\ep\|u^{t'}\|_{\dot{\mathcal{C}}^1}\exp\Big({\frac{1}{\s}\int_0^{t'}C_p\|u^{t''}\|_{\dot{W}^{1+\frac1p,p}}(1+|\log\ep|)^{1-\frac1p}dt''}\Big)dt' \nn\\
     &\le \frac{\delta N^{-\al} (\mathscr{E}^0)^{-\frac1\s}}{\s}\int_0^T \|u^{t'}\|_{\dot{\mathcal{C}}^1}\exp\Big(\frac{1}{\s}\int_0^{T}C_p\|u^{t''}\|_{\dot{W}^{1+\frac1p,p}}(1+|\log(N^{1-\frac{2(\s+1)}{\s}-\al}\delta)| + |\frac1\s\log\mathscr{E}^0|)^{1-\frac1p}\Big).
\end{align}
As $1-\frac1p < 1$, and for any constant $C>0$, $re^{C(1+|\log r|)^{1-\frac{1}{p}}} \rightarrow 0$ as $r\rightarrow 0^+$, we see that we can choose $\delta>0$ sufficiently small depending only on $\d,\s,p,\al,\int_0^T \|u^{t''}\|_{\dot{W}^{1+\frac{\d}{p},p}}dt'', |\Fr_N(\XN^0,\mu^0)|$ such that for every $0\le t\le T$,
\begin{align}
    (\mathscr{E}^0)^{-\frac{1}{\s}} - \frac1\s\int_0^t N^{\frac{2(\s+1)}{\s}-1}\ep\|u^{t'}\|_{\dot{\mathcal{C}}^1}e^{\frac{1}{\s}\int_0^{t'}C_p\|u^{t''}\|_{\dot{W}^{1+\frac1p,p}}(1+|\log\ep|)^{1-\frac1p}dt''}dt'  \ge \frac12(\mathscr{E}^0)^{-\frac{1}{\s}},
\end{align}
which allows us to upgrade the bound \eqref{eq:MFsupC1} to
\begin{multline}
    \forall 0\le t\le T, \qquad  \mathscr{E}^t \le 2^{\s}\mathscr{E}^0\exp\Big(\int_0^{t} C_p\|u^{t'}\|_{\dot{W}^{\frac{\d}{p}+1,p}}\Big(1+|\log(\delta N^{1-\frac{2(\s+1)}{\s}-\al}(\mathscr{E}^0)^{-\frac1\s})|\Big)^{1-\frac1p}dt'\Big).
\end{multline}
Evidently, the preceding right-hand side vanishes as $N\rightarrow\infty$ provided that for any $C>0$, $|\Fr_N(\XN^0,\mu^0)|e^{C|\log N|^{1-\frac1p}}\rightarrow 0$ as $N\rightarrow\infty$ (e.g. $\Fr(\XN^0,\mu^0) = O(N^{-r})$ for some $r>0$). This completes the argument in the (super-)Coulomb case.

\medskip
\textbullet ($0<\s<\d-2$) \ We divide further into two cases based on the value of $p$ in the assumption that $\mu^t \in \dot{W}^{\frac{\d}{p}+\s+2-\d,p}$.

If $p<\frac{2\d}{\d-2}$, then we may choose an exponent $r$ satisfying $\min(2,p)\le r<\frac{2\d}{\d-2}$ so that letting $\as=2(\frac{\d}{r}+1)$, we have $\as \in (\d,\d+2)$ and
\begin{align}
    \|\Dm^{\frac{\as}{2}}u^t\|_{L^{\frac{2\d}{\as-2}}} &\le {\|\Dm^{\frac\as2}\mathsf{V}^t\|_{L^{\frac{2\d}{\as-2}}} +} C|\M|\|\Dm^{\frac{\as}{2}+\s+1-\d}\mu^t\|_{L^{\frac{2\d}{\as-2}}}\nn\\
    &={\|\Dm^{\frac\as2}\mathsf{V}^t\|_{L^{\frac{2\d}{\as-2}}} +}  C|\M|\|\Dm^{\frac{\d}{r}+\s+2-\d}\mu^t\|_{L^{r}} \nn\\
    &\le{\|\Dm^{\frac\as2}\mathsf{V}^t\|_{L^{\frac{2\d}{\as-2}}} +} C'|\M|\|\Dm^{\frac{\d}{p}+\s+2-\d}\mu^t\|_{L^{p}},
\end{align}
where the final inequality is by Sobolev embedding. Recycling notation, let
\begin{align}
    \mathscr{E}^t \coloneqq \Fr_N(\XN^t,\mu^t)  + \sup_{t\le \tau\le T}\ep_\tau\Big(C_q\|\mu^\tau\|_{L^q}\la_\tau^{\frac{\d(q-1)}{q}-\s} + C_q\|\mu^\tau\|_{L^q}^{\frac{(\s+1)q}{\d(q-1)}}\Big).
\end{align}
Applying \cref{thm:FI}, we have that \eqref{eq:dtEscr} holds with $C_1^t, C_2^t$ as in \eqref{eq:C1tdef}, \eqref{eq:C2tdef}, respectively. Repeating the argument from the $\d-2\le \s<\d$ case then completes the proof.
%, and $C_2^t$ now defined by
%\begin{align}\label{eq:C2tdef'}
%    C_2^t \coloneqq C_p\|v\|_{\dot{W}^{\frac{\d}{p}+1,p}}(1+|\log\ep|)^{1-\frac1p}  + %C_{\as}\|\Dm^{\frac{\as}{2}}v\|_{L^{\frac{2\d}{\as-2}}}\indic_{\substack{\as>2}}.
%\end{align}

If $p\ge \frac{2\d}{\d-2}$, then we cannot necessarily control the term $\|\Dm^{\frac{\as}{2}}u^t\|_{L^{\frac{2\d}{\as-2}}}$, for $\as\in (\d,\d+2),$ by $\|\mu^t\|_{\dot{W}^{\frac{\d}{p}+\s+2-\d,p}}$ and instead have to use the suboptimal estimate only requiring $\|\Dm^{\frac{\d-\s}{2}}u^t\|_{L^{\frac{2\d}{\d-\s-2}}}$. Indeed, if $r\le \frac{2\d}{\d-\s-2}$, then by Sobolev embedding,
\begin{align}
    \|\Dm^{\frac{\d-\s}{2}}u^t\|_{L^{\frac{2\d}{\d-\s-2}}} &\le {\|\Dm^{\frac{\d-\s}{2}}\mathsf{V}^t\|_{L^{\frac{2\d}{\d-\s-2}}}+} C\|\Dm^{\frac{\d}{r}+1}\M\nab\g\ast\mu^t\|_{L^{r}} \nn\\
    &\le{\|\Dm^{\frac{\d-\s}{2}}\mathsf{V}^t\|_{L^{\frac{2\d}{\d-\s-2}}}+}  C'|\M|\|\mu^t\|_{\dot{W}^{\frac{\d}r + \s+2-\d,r}}\nn\\
    &\le {\|\Dm^{\frac{\d-\s}{2}}\mathsf{V}^t\|_{L^{\frac{2\d}{\d-\s-2}}}+} C'|\M|\|\mu^t\|_{L^{\frac{\d}{\d-\s-2}}}.
\end{align}
Since $\frac{2\d}{\d-2}\le p\le \frac{2\d}{\d-\s-2}$, we are done. % If $p>\frac{2\d}{\d-\s-2}$, then we use our assumption that $\mu\in L^1\cap L^q$ for $q\ge \frac{\d}{\d-\s-2}$ plus H\"older's inequality to further bound $\|\mu^t\|_{L^{\frac{\d}{\d-\s-2}}}\le \|\mu^t\|_{L^1} \|\mu^t\|_{L^q}$
%(our assumption that $\mu \in L^1\cap L^q \cap \dot{W}^{\frac{\d}{p}+\s+2-\d}$ for $q>\frac{\d}{\d-\s-1}$ and $1<p<\infty$ always allows us to reduce to the case $p\le \frac{2\d}{\d-\s-2}$)
Redefining
\begin{align}
    \mathscr{E}^t \coloneqq \Fr_N(\XN^t,\mu^t)  + \sup_{t\le \tau\le T}\ep_\tau C_q\Big(\|\mu^\tau\|_{L^{{q}}}\ka_\tau^{\frac{\d({q}-1)}{{q}}-\s} + \|\mu^\tau\|_{L^q}^{\frac{(\s+1)q}{\d(q-1)}}\Big),
\end{align}
the inequality \eqref{eq:dtEscr} holds with $C_1^t, C_2^t$ as before, and repeating the same reasoning completes the proof.
%\begin{align}
%    C_2^t = C_p\|u^t\|_{\dot{W}^{\frac{\d}{p}+1,p}}(1+|\log\ep|)^{1-\frac1p}  + C\|\Dm^{\frac{\d-\s}{2}}u^t\|_{L^{\frac{2\d}{\d-\s-2}}}.
%\end{align}

\medskip
\textbullet ($-2<\s<0$) \ We divide into two cases based on the value of $\s$. 

If $-2<\s\le -1$, then we redefine
\begin{align}
    \mathscr{E}^t \coloneqq \Fr_N(\XN^t,\mu^t) + \sup_{t\le \tau\le T} C\ep_\tau\int_{\R^\d}|x|^{|\s|-1}d\mu^\tau,
\end{align}
and we see from \cref{thm:FI} that $\mathscr{E}^t$ obeys the differential inequality,
\begin{align}
    \frac{d}{dt}\mathscr{E}^t \le C_1^t(\mathscr{E}^t)^{\frac{|\s|-1}{|\s|}} +  C_2^t\mathscr{E}^t,
\end{align}
where $C_2^t$ is as above, but now $C_1^t\coloneqq C\|u^t\|_{\dot{\mathcal{C}}^1}\ep_t$. Applying \cref{lem:Gron}\eqref{eq:Grona<1} with $a=\frac{|\s|-1}{|\s|}$, we obtain
\begin{align}\label{eq:Etnonsing1}
    \forall t\le T,\qquad (\mathscr{E}^t)^{\frac{1}{|\s|}} \le (\mathscr{E}^0)^{\frac{1}{|\s|}}\exp\Big({\frac{1}{|\s|}\int_0^t C_2^\tau d\tau}\Big) + \frac{1}{|\s|}\int_0^t C_1^{t'} \exp\Big({\frac{1}{|\s|}\int_{t'}^t C_2^\tau d\tau}\Big) dt'.
\end{align}
For any $\alpha>0$, we can choose $\ep_\tau = N^{-\al}$, and it follows that the preceding right-hand side tends to zero as $N\rightarrow\infty$ provided that for some $A>0$, $|\Fr_N(\XN^0,\mu^0)|^{\frac{1}{|\s|}}e^{A(1+\log N)^{1-\frac1p}}\rightarrow 0$ as $N\rightarrow\infty$.

%As $\frac{|\s|-1}{|\s|}<1$, we may use the elementary inequality $ab\le \frac{a^p}{p}+\frac{b^{p'}}{p'}$ on the first term on the preceding right-hand side, then apply the usual Gr\"onwall lemma to obtain
%\begin{align}\label{eq:Etnonsing1}
%    \mathscr{E}^t \le C'\Big(\int_0^t \|u^\tau\|_{\dot{\mathcal{C}}^1}\ep_\tau^{|\s|}d\tau + \mathscr{E}^0\Big) \exp\Big(C_p'\int_0^t\|u^\tau\|_{\dot{W}^{\frac{\d}{p}+1.p}}(1+|\log\ep_\tau|)^{1-\frac1p}d\tau\Big)
%\end{align}

%Applying \cref{lem:Gron} with $a= \frac{|\s|-1}{|\s|}$ and following the rest of the argument as when $\s\ge \d-2$ completes the proof in this case.
 
If $-1<\s<0$, then we instead let
\begin{align}
    \mathscr{E}^t \coloneqq \Fr_N(\XN^t,\mu^t) + \sup_{t\le \tau\le T} \Big(C_{\vartheta,\vartheta'}\ep_\tau^{{\vartheta'}}\int_{\R^\d}|x|^{|\s|-\vartheta}d\mu^\tau + C\ep_\tau(1+C_q\|\mu^\tau\|_{L^q}^{\frac{(\s+1)q}{\d(q-1)}})\Big),
\end{align}
and we see from \cref{thm:FI} that $\mathscr{E}^t$ obeys the differential inequality,
\begin{align}
    \frac{d}{dt}\mathscr{E}^t \le C_1^t(\mathscr{E}^t)^{\frac{|\s|-\vartheta}{|\s|}} +  C_2^t\mathscr{E}^t,
\end{align}
where $C_2^t$ is as above, but now $C_1^t \coloneqq C_\vartheta\|u^t\|_{\dot{\mathcal{C}}^1}\ep_t^{\vartheta}$. Applying \cref{lem:Gron}\eqref{eq:Grona<1} with $a=\frac{|\s|-\vartheta}{|\s|}$, we find that
\begin{align}\label{eq:Etnonsing2}
    \forall t\le T,\qquad (\mathscr{E}^t)^{\frac{\vartheta}{|\s|}} \le (\mathscr{E}^0)^{\frac{\vartheta}{|\s|}}\exp\Big({\frac{\vartheta}{|\s|}\int_0^t C_2^\tau d\tau}\Big) + \frac{\vartheta}{|\s|}\int_0^t C_1^{t'} \exp\Big({\frac{\vartheta}{|\s|}\int_{t'}^t C_2^\tau d\tau}\Big) dt'.
\end{align}
Choosing $\ep_\tau = N^{-\al}$ then concludes the argument. 
%Appealing to \cref{lem:Gron} with $a= \frac{|\s|-\vartheta}{|\s|}$ and proceeding as before then concludes the argument.
With this last case, the proof of the theorem is complete.
\end{proof}

%\begin{remark}
 %   An examination of the proof of the $-2<\s<0$ case reveals that we actually only need $\dot{\mathcal{C}}^1$ control, and not the stronger $\dot{W}^{\frac{\d}{p}+1,p}$, on the vector field $u^\tau$ to suitably close the differential inequality. Indeed, now $C_2^t = C(1+|\log\ep_t|)\|u^t\|_{\dot{\mathcal{C}}^1}$. We can choose $\ep_t = N^{-\al}$ for $\al>0$ as before. 
%\end{remark}

\section{Remaining questions and open problems}\label{sec:RQ}
We conclude the paper by discussing a few questions that have been so far unaddressed and posing several open problems concerning transport-commutator functional inequalities, the regularity of the transport, and the regularity of the mean-field density. We hope that this discussion will inspire future research on the subject.

\subsection{Optimal sub-Coulomb transport regularity}\label{ssec:RQsubC}
As we noted in \cref{sec:intro}, there is an unfortunate gap between the two sub-Coulomb commutator estimates \eqref{eq:FIsubC1}, \eqref{eq:FIsubC2} in \cref{thm:FI'}. In the latter, the rate in $N$ is suboptimal, but the transport regularity $\|\Dm^{\frac{\d-\s}{2}}v\|_{L^{\frac{2\d}{\d-\s-2}}}$ is expected to be optimal in light of \cref{cor:CEF2}\ref{item:CEF21}. While in the former, the rate in $N$ is optimal, but the transport regularity $\|\Dm^{\frac{\as}{2}}v\|_{L^{\frac{2\d}{\as-2}}}\indic_{\as>2}$ for $\as\in (\d,\d+2)$, although scaling the same way as $\|\Dm^{\frac{\d-\s}{2}}v\|_{L^{\frac{2\d}{\d-\s-2}}}$, is stronger by Sobolev embedding and we believe suboptimal.

The first problem we pose is to show a commutator estimate which has both the sharp error $N^{\frac\s\d-1}$ and the believed sharp transport regularity $\|\Dm^{\frac{\d-\s}{2}}v\|_{L^{\frac{2\d}{\d-\s-2}}}$ when $0\le \s<\d-2$. The dependence $\|\Dm^{\frac{\mathsf{a}}{2}}v\|_{L^{\frac{2\d}{\mathsf{a}-2}}}$, for any $\mathsf{a}\in (\d,\d+2)$, is an artifact of our passing through an intermediate commutator estimate for Bessel potentials  in the proof of \cref{thm:FI}. %By Sobolev embedding, it is stronger than $\|\Dm^{\frac{\d-\s}{2}}v\|_{L^{\frac{2\d}{\d-\s-2}}}$.
We expect the latter to be the optimal regularity dependence even with the sharp $N^{\frac{\s}{\d}-1}$ error, but it does not seem obtainable with our current proof. The solution to this problem likely requires devising a different renormalization argument compared to \cite{nguyen_mean-field_2022, hess-childs_sharp_2025}, given this part of the proof is the source of the issue. As the latter cited work is the only one that gives sharp errors, let us explain more why its approach does not seem amenable to obtaining the optimal transport regularity.

The starting point for the argument of \cite{hess-childs_sharp_2025} is a ``wavelet-type'' representation of the Riesz potential $\g$ as an average of approximate identities, which may be truncated at small scale $\eta\ge 0$ to obtain
	\begin{align}\label{eq:introgintrep}
		\forall x\ne 0,\qquad \g_\eta(x): = \mathsf{c}_{\phi,\d,\s}\int_\eta^\infty t^{\d-\s}\phi_t(x)\frac{dt}{t},
	\end{align}
	where $\phi$ is any sufficiently nice radial function with nonnegative Fourier transform, $\phi_t(x)\coloneqq t^{-\d}\phi(x/t)$, and $\mathsf{c}_{\phi,\d,\s}$ depends on $\phi,\d,\s$. 
	Given suitable assumptions on $\phi$ (see \cite[Lemma 2.2]{hess-childs_sharp_2025}), $\g_\eta$ enjoys a number of nice properties. For instance, $\g_\eta$ is bounded for $\eta>0$, $\g_\eta\leq \g$, $\g_\eta$ is positive definite, and $\g-\g_\eta$ decays rapidly at scale $\eta$ outside the ball $B(0,\eta)$. In the log case $\s=0$, the integral has to be renormalized at large scales to obtain a convergent expression. 
    
    The main technical result \cite[Corollary 3.2]{hess-childs_sharp_2025} is a commutator estimate for the truncated potential,
	\begin{equation}\label{eq:CEtrunc}
		\bigg|\int_{(\R^\d)^2}(v(x)-v(y))\cdot\nabla\g_\eta(x-y) f(x) f(y)\bigg|\leq C(v) \int_{(\R^\d)^2} \g_\eta(x-y) f(x)f(y).
	\end{equation}
	This is achieved by using the representation~\eqref{eq:introgintrep} and scaling invariance to reduce the problem to proving the estimate with $\g_\eta$ replaced by $\phi$. Taking $\phi$ to be the Bessel potential $\hat\phi(\xi) = \jp{2\pi\xi}^{-\mathsf{a}}$ with $\mathsf{a}\in(\d,\d+2)$, this then holds as a consequence of Kato-Ponce type commutator estimates for the inhomogeneous Fourier multiplier $\jp{\nab}^{\as}$ (the availability of such estimates motivates our choice of $\phi$), but with a norm on $v$ of the form $\||\nabla|^{\mathsf{a}/2} v\|_{L^{\frac{2\d}{\mathsf{a}-2}}}$. The upper bound on $\as$ is not important to us, as we would like to take $\as=\d-\s$. The lower bound is natural given $\as>\d$ is necessary for $\phi$ to be bounded, which in turn is needed for the truncated potential $\g_\eta$ to be bounded. 
	
	\cref{thm:CEF2} suggests that approaches based on an intermediate kernel $\phi$ with fast Fourier decay necessarily result in an estimate with a suboptimal regularity requirement for $v$. The question thus remains whether~\eqref{eq:CEtrunc} can be established directly, i.e.~without relying on an intermediate commutator estimate for the potential $\phi$. %Additionally, it is unclear how this truncation/renormalization scheme can be adapted to give local and higher-order estimates as proposed in the previous question.

%\subsection{The $\s=\d$ endpoint}

\subsection{Localizability}\label{ssec:RCloc}
	Another important question is the \emph{localizability} of commutator estimates. To illustrate what we mean, consider the Coulomb case $\s=\d-2$. Integrating by parts, one has the \textit{electric reformulation} of the energy
	\begin{align}
	    \int_{(\R^\d)^2}\g(x-y)f(x)f(y)=  \frac{1}{\cd}\int_{\R^\d}|\nabla h^f|^2,
	\end{align}
	where $h^f\coloneqq\g*f$. Evidently, this latter quantity may be localized to any domain $\Omega\subset\R^\d$. Exploiting the stress-energy tensor structure the commutator possesses when $\s=\d-2$, one has the localized commutator estimate
	\begin{equation}
		\bigg|\int_{(\R^\d)^2} (v(x)-v(y))\cdot\nabla\g(x-y)f(x)f(y)\bigg|\leq C\|\nabla v\|_{L^\infty}\int_{\supp v}|\nabla h^f|^2.
	\end{equation}
	This argument can also be extended to $\s>\d-2$, as well as to higher-order commutator estimates and their renormalized variants~\cite{leble_fluctuations_2018, serfaty_gaussian_2023,rosenzweig_sharp_nodate}. Such localized estimates are important for applications to CLTs for the fluctuations of Riesz gases at mesoscales (e.g. see \cite[Part 3]{serfaty_lectures_2024}).
    
    The case of a local version of the (renormalized) commutator estimate is open for the sub-Coulomb range, where the commutator no longer has the same stress-energy tensor structure as is crucial to the argument in~\cite{rosenzweig_sharp_nodate}. Moreover, the proofs in \cite{nguyen_mean-field_2022, hess-childs_sharp_2025} are not suitable for localized estimates given their reliance on intermediary nonlocal commutator estimates.
    
    In the (sub-Coulomb) ``local'' case $\s=\d-2k$, for integer $k\ge 2$, where $\g$ is the fundamental solution of $(-\Delta)^k$, one can still reformulate the energy,
    \begin{align}\label{eq:subCouelec}
        \int_{(\R^\d)^2}\g(x-y)f(x)f(y)= \frac{1}{\cd}\int_{\R^\d}|\nab^{\otimes k} h^f|^2,
    \end{align}
    where the electric field $\nabla h^f$ is replaced by the $k$-tensor field $\nab^{\otimes k}h^f$, and evidently, the latter quantity in \eqref{eq:subCouelec} may be localized as before. Preliminary calculations indicate that there is a ``higher-order'' stress-energy tensor structure to the commutator in this case, which may be exploited through integration by parts to prove an unrenormalized commutator estimate. That said, the optimal regularity dependence for $v$ remains to be determined. In the nonlocal case $\s\ne \d-2k$, this reasoning can likely be adapted using the generalization of the Caffarelli-Silvestre extension for higher powers of the fractional Laplacian \cite{yang_higher_2013}, analogous to the relation between the Coulomb case $\d-2$ and the super-Coulomb case $\s>\d-2$.
    
    What is less straightforward and the crux of the problem is finding a suitable renormalization scheme with which to combine the unrenormalized localized commutator estimate to obtain a localized version of the estimate \eqref{eq:FIsubC1} or \eqref{eq:FIsubC2}. The two renormalization schemes available for the sub-Coulomb case, the averaging scheme in \cite{nguyen_mean-field_2022} and the potential truncation scheme in \cite{hess-childs_sharp_2025}, do not seem suited to this task. The former will not produce sharp errors in $N$, while the latter cannot obviously be combined with the localized electric formulation described above.

\subsection{Higher-order estimates}\label{ssec:RQhigher}
So far, we have only considered first-order commutator estimates corresponding to the first variation \eqref{15} of the modulated energy. Also of interest are the higher-order variations corresponding to $n\ge 2$ in 
\begin{align}\label{15'}
 \frac{d^n}{dt^n}\Big|_{t=0} \Fr_N( (\I + tv)^{\oplus N} (\ux_N), (\I + tv)\# \mu) &= {\frac12}\int_{(\R^\d)^2\setminus \triangle} 
\nabla^{\otimes n} \g(x-y):  (v(x)-v(y))^{\otimes n}  d ( \mu_N- \mu)^{\otimes 2} \nn\\
&\eqqcolon\As_{n}[\XN,\mu,v],
\end{align}
where $\I+tv, (\I+tv)^{\oplus N}$ are as in \cref{ssec:introMEcomm} and $:$ denotes the inner product between tensors. As in the first-order case, one seeks to control these quantities by $C_v(\Fr_N(\ux_N, \mu)+ C_\mu N^{-\alpha})$, and we call such control a \emph{higher-order commutator estimate}.

Second-order (i.e. $n=2$) estimates are important for the aforementioned transport method for studying the fluctuations of Coulomb/Riesz gases, as well as for deriving mean-field limits with multiplicative noise \cite{rosenzweig_mean-field_2020}. Estimates beyond second order are also useful, allowing to  obtain finer estimates  on the fluctuations of Riesz gases to treat a broader class of interactions \cite{peilen_local_2025} and also to compute the asymptotics of $n$-th order cumulants \cite{rosenzweig_cumulants_nodate}.

Higher-order commutator estimates were first shown at second order in \cite{leble_fluctuations_2018,rosenzweig_mean-field_2020}  for the $\d=2$ Coulomb case, then in \cite{serfaty_gaussian_2023} for the general  Coulomb case (but with an estimate that is suboptimal in its dependence on $\Fr_N$), and later in \cite{nguyen_mean-field_2022,rosenzweig_wasserstein_nodate} for the full range $-2 < \s<\d$, as well as for more general Riesz-type interactions. The approach of \cite{nguyen_mean-field_2022} yields higher-order estimates, but these are generally not sharp in their additive error. In contrast, the estimates from \cite{serfaty_gaussian_2023} are sharp but only to second order for the $\d=2$ Coulomb case. More importantly, their proof is quite intricate and not evidently generalizable to higher order. In \cite{rosenzweig_sharp_nodate}, sharp localized estimates were shown at all orders for the (super-)Coulomb case by means of a delicate induction argument.

Although it is somewhat orthogonal to our focus on transport regularity, given its importance for studying fluctuations around the mean-field limit, we start by posing the problem of a sharp estimate for higher-order commutators in the sub-Coulomb case $\s<\d-2$. The aforementioned work \cite{hess-childs_sharp_2025} by the authors proves a sharp first-order estimate, but it is not clear how to extend the method of that paper to $n\ge 2$. Desymmetrizing \eqref{15'} and writing in terms of iterated commutators entails difficulties of considering commutators of commutators, not too mention the algebraic complexity as $n$ increases. Moreover, the approach to higher orders in \cite{nguyen_mean-field_2022} via integrating by parts to throw derivatives onto the commutator kernel is not obviously compatible with the renormalization scheme in \cite{hess-childs_sharp_2025}.

Even for the (super)Coulomb case, where a sharp localized estimate at any order has been shown by the last two authors in \cite{rosenzweig_sharp_nodate}, the transport regularity assumptions are not fully satisfactory, and this issue is intimately linked to the regularity theory for commutators, first initiated in that work. To illustrate what we mean, we recall some facts from \cite{rosenzweig_sharp_nodate}. So as to focus on the main ideas, we limit the discussion to the Coulomb case, but we note that everything may be generalized to the super-Coulomb case. 

Fixing a vector field $v$ and given a test function $f$, let $\ka^{(n),f}$ be the \emph{$n$-th order commutator} defined by
\begin{align}\label{eq:kanfdef}
    \kappa^{(n),f}(x) \coloneqq \int_{\R^\d} \nabla^{\otimes n} \g(x-y) : (v(x)-v(y))^{\otimes n} df(y).
\end{align}
$\kappa^{(n),f}$ obeys the PDE (see \cite[Section 4.1]{rosenzweig_sharp_nodate} for the calculation)
\begin{multline}\label{eq:introLkanf}
-\Delta\ka^{(n),f} = \cd(-1)^n\sum_{\sigma \in \Ss_n} \p_{i_{\sigma_1}}{v}^{i_1}\cdots\p_{i_{\sigma_n}}v^{i_n}f  -n\p_i(\p_i v \cdot\nu^{(n-1),f})  \\
- n\p_i v \cdot \p_i\nu^{(n-1),f}+ n(n-1)(\p_i v)^{\otimes 2}:\mu^{(n-2),f},
\end{multline}
where $\mathbb{S}_n$ denotes the symmetric group on $[n]$ and a repeated index denotes summation over that index. On the right-hand side, $\nu$ is a vector field on $\R^{\d}$ that ``morally'' is like $\nab\ka^{(n),f}$, while $\mu$ (not to be confused with the mean-field density) is a symmetric matrix field on $\R^{\d}$ that morally is like $\nab^{\otimes 2}\ka^{(n),f}$ (see \cite[(4.12), (4.14)]{rosenzweig_sharp_nodate}, respectively, for the precise definitions). The right-hand side of \eqref{eq:introLkanf} is good because it only depends on lower-order (i.e.~ $n-1$, $n-2$) commutators.

We believe that the right-hand side of \eqref{eq:introLkanf} can be written in divergence form involving terms built from products of the components of $\nab v$ and $\nu^{(m)}$ for $m\le n-1$. More precisely, we conjecture
\begin{align}\label{eq:introLkanfdiv}
-\Delta\ka^{(n),f} = - n\div(\nab v^{i}\nu_{i}^{(n-1),f}) +  \sum_{k=1}^{n} C_{k,n}\sum_{\sigma \in \Ss_{k+1}} (-1)^{|\sigma|} \p_{i_{\sigma(k+1)}}\Big( \p_{i_{\sigma(1)}}v^{i_1}\cdots \p_{i_{\sigma(k)}}v^{i_k} \nu_{i_{k+1}}^{(n-k),f}\Big),
\end{align}
where $C_{k,n}$ are certain combinatorial coefficients and $|\sigma|$ denotes the signature of the permutation $\sigma$.  By testing \eqref{eq:introLkanfdiv} against $h^w\coloneqq \g\ast w$, using that
\begin{align}
\int_{(\R^\d)^2}(v(x)-v(y))\cdot\nab\g(x-y)f(x)g(y) = \frac{1}{\cd}\int_{\R^{\d}}(-\Delta)\ka^{(n),f}h^g,
\end{align}
and reversing the product rule, one arrives at a stress-energy tensor structure to the higher-order commutators. If such an identity \eqref{eq:introLkanfdiv} holds and one has a bound
\begin{align}\label{eq:introL2nu}
\int_{\R^{\d}}|\nu^{(m),f}|^2 \leq C_{m}(\|\nab v\|_{L^\infty})^{2m} \int_{\supp\nab v}|\nab h^f|^2, \qquad m\le n-1,
\end{align}
then it follows immediately from integration by parts and Cauchy-Schwarz that
\begin{align}\label{eq:introL2nabka}
\int_{\R^{\d}} |\nab \ka^{(n),f}|^2  \le C_{n}(\|\nab v\|_{L^\infty})^{2n}\int_{\supp\nab v}|\nab h^f|^2.
\end{align}
Since the $\nu$'s obey a recursion in terms of the $\nab \ka$'s (see \cite[(4.13)]{rosenzweig_sharp_nodate}), the estimate \eqref{eq:introL2nabka} implies that \eqref{eq:introL2nu} holds for $m=n$ and by induction, the estimates \eqref{eq:introL2nabka}, \eqref{eq:introL2nu} hold for any integer $n\ge 1$. 

Unfortunately, we were only able to prove the identity \eqref{eq:introLkanfdiv} for $n\le 2$ (see \cite[Section 4.3]{rosenzweig_sharp_nodate}), and the computation is already quite involved once $n=2$. If we stick with the non-divergence form of the right-hand side of \eqref{eq:introLkanf}, then we encounter terms that have a factor of $\ka^{(n),f}$. For example, 
\begin{align}
\int_{\R^{\d}}\ka^{(n),f}\p_i v \cdot \p_i\nu^{(n-1),f} = -\int_{\R^{\d}}\p_i\ka^{(n),f} \p_{i}v\cdot\nu^{(n-1),f}  -\int_{\R^{\d}}\ka^{(n),f}\p_i^2 v\cdot \nu^{(n-1),f}. 
\end{align}
One cannot simply use Cauchy-Schwarz on the second term on the right-hand side, as in general, there is no way to control $\int_{\supp\nab v}|\ka^{(n),f}|^2$ by $\int_{\supp\nab v}|\nab\ka^{(n),f}|^2$. However, if one's primary interest is in localized estimates, then there is a way out of this issue by exploiting the localization from the start. %, defining $\ka^{(n),f}$ in terms of a vector field extension $\tv$ which is localized in   a ball $B$ of radius $C\ell$ in $\R^{\d+\k}$ as in \cref{rem:vextell}.
 %, assuming $\supp v$ is contained in a ball of radius $\ell$ in $\R^\d$ (see \cref{rem:vextell}). 
 Since $-\Delta\ka^{(n),f}$ has zero average, we may equivalently test the equation \eqref{eq:introLkanf} against $\ka^{(n),f} -\bar{\ka}^{(n),f}$, where $\bar{\ka}^{(n),f}$ denotes the average in a ball $B$ containing the support of $v$. The strategy is to integrate by parts and use Cauchy-Schwarz as before, setting up an induction argument for estimates satisfied by $\ka^{(m),f}, \nu^{(m),f}, \mu^{(m),f}$ (see \cite[Lemma 4.4]{rosenzweig_sharp_nodate}). Crucially, the Lebesgue measure on $B$ satisfies a Poincar\'{e} inequality, which allows to control $\int_{B} |\ka^{(n),f}-\bar{\ka}^{(n),f}|^2$ by $\int_{B}|\nab\ka^{(n),f}|^2$.  %Crucially, the measure $\zg dxdz$ satisfies a Poincar\'{e} inequality in balls, which allows to control $\int_{B}\zg |\ka^{(n),f}-\bar{\ka}^{(n),f}|^2$ by $\int_{B}\zg |\nab\ka^{(n),f}|^2$. 
 %We refer to \cref{ssec:L2regloc} for the details.

 Although this approach requires a bound for $\|\nab^{\otimes 2}v\|_{L^\infty}$, which is a stronger demand than the Lipschitz requirement that would follow if \eqref{eq:introLkanfdiv} holds, this is still sufficient for applications, such as to CLTs for linear statistics of Coulomb/Riesz gases. Nevertheless, improving the regularity to $\|\nab v\|_{L^\infty}$ would allow one to consider rougher test functions in this application, not to mention the aesthetic motivation for an estimate that is uniform in the localization.

\subsection{Scaling-critical mean-field convergence for $p=\infty$}\label{ssec:RQmf}
The last question we consider concerns \cref{thm:FI,thm:MF} in the excluded case $p=\infty$. The case $p=1$ is also excluded, but this is not particularly relevant because the critical space $\dot{W}^{\s+2,1}$ never corresponds to a nonincreasing/conserved quantity, unlike $L^\infty$, which is scaling-critical in the Coulomb case.

First, the reader should recall from \cref{rem:dcompinfty} that per our notation (see \cref{ssec:introN}), the limit of the space $\dot{W}^{\frac{\d}{p}+\s+2-\d,p}$ as $p\rightarrow\infty$ is the space
\begin{align}
    \{f\in \Sc'(\R^\d) : \|\jp{\nab}^{\s+2-\d} f\|_{L^\infty}<\infty\}.
\end{align}
This space is not well-behaved from a harmonic analysis perspective in terms of singular integral operators/Fourier multipliers being bounded on it, and a more natural $L^\infty$-type space is the homogeneous H\"older-Zygmund space $\dot{\mathcal{C}}^{\s+2-\d}$, which is well-behaved. In particular, we have
\begin{align}
    \|\M\nab\g\ast\mu\|_{\dot{\mathcal{C}}^1} \lesssim |\M| \|\mu\|_{\dot{\mathcal{C}}^{\s+2-\d}}.
\end{align}
If $\s=\d-2$ (i.e. Coulomb), then $\dot{\mathcal{C}}^0$ should be understood as the homogeneous Besov space $\dot{B}_{\infty,\infty}^0$, into which $L^\infty$ embeds.

One can still prove a defective commutator estimate like \cref{thm:FI} when the vector field $v\in \dot{\mathcal{C}}^1$, but the divergent prefactor $(1+|\log\ep|)^{1-\frac1p}$ that comes from estimating $\|\nab v_\ep\|_{L^\infty}$ is replaced by $1+|\log\ep|$. In the log case $\s=0$, this factor $1+|\log\ep|$ is still sufficient to prove a version of \cref{thm:MF} (see \cite[Proposition 1.8]{rosenzweig_mean-field_2022-1}) by essentially taking $\ep=\ep_t$ to be comparable to the modulated energy $\Fr_N(\XN^t,\mu^t)$ and using that $r|\log r|$ satisfies the conditions of the Osgood lemma. Importantly, in estimating the mollification error $v-v_\ep$ (specifically, $\Te_1$ in the proof of \cref{lem:molvepLL}), the log-Lipschitz regularity of $v$ allows one to bound for $|x_i-x_j|\le \frac12$,
\begin{align}
    \Big|\nab\g(x_i-x_j)\cdot\Big((v-v_\ep)(x_i)-(v-v_\ep)(x_j)\Big)\Big| \le C\|v\|_{\dot{\mathcal{C}}^1}\frac{|\log|x_i-x_j||}{|x_i-x_j|^{\s}} = C\|v\|_{\dot{\mathcal{C}}^1}\g(x_i-x_j).
\end{align}
This relation is, of course, only true in the log case. When $0<\s<\d$, one has the leftover factor $|\log|x_i-x_j||$. In the proof of \cite[Proposition 4.1]{rosenzweig_mean-field_2022}, this extra log factor is handled by controlling the minimal distance between points by the $N$-particle energy. Taking $\ep\propto N^{-\alpha}$ for large enough $\alpha>0$ then allows one to close a Gr\"onwall relation, but only for short times depending on the size of the initial modulated energy. The reason is that one encounters the quantity $N^{-\be}e^{Ct|\log N^{-\al}|}$, for some $\be>0$, which only vanishes as $N\rightarrow\infty$ if $t$ is small enough (independently of $N$). Even replacing this part of the argument with that used to estimate $\Te_1$ in \cref{lem:molvepLL} would not suffice if we only had bound $\|\nab v_\ep\|_{L^\infty} \lesssim (1+|\log\ep|)$. The prefactor $N^{\frac{2(\s+1)}{\s}-1}\ep$ on the right-hand side of \eqref{eq:molvepLL} when $0<\s<\d$ again requires our taking $\ep \lesssim N^{-\alpha}$ for large enough $\alpha$, and we would run into the same short-time issue as before. Thus, what is essential is having the control $\|\nab v_\ep\|_{L^\infty}\lesssim (1+|\log\ep|)^{\theta}$ for $\theta \in (0,1)$, which is made possible by the stronger assumption that $v\in \dot{W}^{\frac{\d}{p}+1,p}$.

To date, mean-field convergence under the assumption that the limiting density $\mu\in L_t^\infty L_x^\infty$ has only been shown to hold on some interval of time $[0,T_*]$, despite the fact that solutions of the mean-field equation are global. Removing this short-time restriction, which we think is purely technical, is an important open problem. More generally, it would be desirable to show convergence without any short-time restriction in the super-Coulomb case when $\mu \in L_t^\infty \dot{\mathcal{C}}_x^{\s+2-\d}$.

\bibliography{Master}
\bibliographystyle{alpha}

\end{document}